\documentclass[12pt]{amsart}
\usepackage[all]{xy}
\usepackage[parfill]{parskip}

\usepackage{amsfonts}
\usepackage{mathrsfs}
\usepackage[abs]{overpic}

\usepackage{palatino} 

\usepackage{hyperref}

\usepackage{verbatim}
\usepackage{color}

\usepackage{amsmath, amscd, graphicx, latexsym, hyperref, times, calc}
\usepackage{color,graphicx, rotate}
\usepackage[abs]{overpic}
\usepackage{tikz}
\usepackage{tikz-cd}
\usetikzlibrary{arrows, patterns}

\newtheorem{dummy}{anything}[section]
\newtheorem{theorem}[dummy]{Theorem}

\newtheorem{lemma}[dummy]{Lemma}

\theoremstyle{definition}

\newtheorem{remark}[dummy]{Remark}

\newtheorem*{acknowledgements}{Acknowledgements}

\usepackage[utf8]{inputenc}

\title{Strongly exceptional Legendrian  connected sum of two Hopf links}

\author{Youlin Li}
\address{School of Mathematical Sciences, Shanghai Jiao Tong University, Shanghai 200240, China}
\email{liyoulin@sjtu.edu.cn}

\author{Sinem Onaran}
\address{Department of Mathematics, Hacettepe University, 06800 Beytepe-Ankara, Turkey}
\email{sonaran@hacettepe.edu.tr}

\begin{document}

\maketitle

\begin{abstract}
In this paper, we give a complete coarse classification of strongly exceptional Legendrian realizations of connected sum of two Hopf links in contact 3-spheres. This is the first classification result about exceptional Legendrian representatives for connected sums of link families. 
\end{abstract}

\section{Introduction}
A Legendrian link in an overtwisted contact 3-manifold is \textit{exceptional} (a.k.a. \textit{non-loose}) if its complement is tight. There have been several classification for exceptional Legendrian knots and links in overtwisted contact 3-spheres, including unknots \cite{ef}, \cite{d}, torus knots \cite{go2}, \cite{m}, \cite{emm} and Hopf links \cite{go}.  While there has been very little progress in the classification of Legendrian links with two or more components in either tight or overtwisted contact 3-spheres, a few papers, \cite{det}, \cite{DG1}, \cite{DG2}, \cite{go},  have tackled the problem.

In this paper, we study the classification of Legendrian realizations of connected sum of two Hopf links up to coarse equivalence in any contact $3$-sphere. This is one of the first families of connected sum of links for which a classification is known. Two Legendrian realizations $K_{0}\cup K_{1}\cup K_2$ and $K'_{0}\cup K'_{1}\cup K'_2$ of the connected sum of two Hopf links in some contact $3$-sphere $S^3$ are coarsely equivalent if there is a contactomorphism of $S^3$ sending $K_{0}\cup K_{1}\cup K_2$ to $K'_{0}\cup K'_{1}\cup K'_2$ as an ordered, oriented link.

Let $A_3=K_{0}\cup K_{1}\cup K_2\subset S^3$ be the oriented connected sum of two Hopf links, where $K_0$ is the central component. It is shown in Figure~\ref{figure:Link0}. The orientations of the components are also indicated. We think of $K_1$ and $K_2$ as two oriented meridians of $K_0$.

\begin{figure}[htb] {\small
\begin{overpic}
{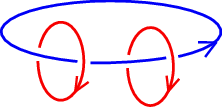}
\put(60, 55){$K_0$}
\put(6, 0){$K_1$}
\put(50, 0){$K_2$}
\end{overpic}}
\caption{The link $A_3=K_{0}\cup K_{1}\cup K_2$ in $S^3$.  }
\label{figure:Link0}
\end{figure}

We consider the Legendrian realizations of $A_3$ in all contact 3-spheres. For $i=0,1,2$,  denote the Thurston-Bennequin invariant of $K_i$ by $t_i$, and the rotation number of $K_{i}$ by $r_{i}$.

Let $(M,\xi)$ be a contact 3-manifold and $[T]$ an isotopy class of embedded tori in $M$. The \textit{Giroux torsion} of $(M,\xi)$ is the supremum of $n\in\mathbb{N}_0$ for which there is a contact embedding of $$(T^2\times [0,1], \ker(\sin(n\pi z)dx+\cos(n\pi z)dy))$$ into $(M,\xi)$, with $T^{2}\times \{z\}$ being in the class $[T]$.

An exceptional Legendrian link in an overtwisted contact 3-manifold is called \textit{strongly exceptional} if its complement has zero Giroux torsion. This paper focuses on the classification of strongly exceptional Legendrian realizations of the $A_3$ link in contact 3-spheres up to coarse equivalence. We use the notation $\xi_{st}$ to refer to the standard tight contact structure on $S^3$. The countably many overtwisted contact structures on $S^3$ are determined by their $d_3$-invariants in  $\mathbb{Z}+\frac{1}{2}$ \cite[Section 2]{go}. If the $d_3$-invariant of an overtwisted contact 3-sphere is $d$, then we denote this contact 3-sphere by $(S^3, \xi_{d})$. Note that the $d_3$-invariant of $\xi_{st}$ is $-\frac{1}{2}$.

We enumerate all the strongly exceptional Legendrian $A_3$ links up to coarse equivalence.

\begin{theorem} \label{Theorem:t_1<0t_2<0} Suppose $t_{1}<0$ and $t_{2}<0$, then the number of strongly exceptional Legendrian $A_3$ links is
\begin{align*}
\begin{split}
\left\{  
\begin{array}{ll}
 2t_{1}t_{2}-2t_{1}-2t_{2}+2, & ~\text{if}~ t_{0}\geq 2,\\
 t_{1}t_{2}-2t_{1}-2t_{2}+2, & ~\text{if}~ t_{0}=1,\\
 -2t_{1}-2t_{2}+2, & ~\text{if}~ t_0=0,\\
-t_{0}t_{1}t_{2}, & ~\text{if}~ t_0\leq-1.
\end{array}
\right.
\end{split}
\end{align*}
Moreover, if $t_0\leq -1$, then the $-t_{0}t_{1}t_{2}$ Legendrian $A_3$ links are in the standard tight contact 3-sphere $(S^3, \xi_{st})$.
\end{theorem}

\begin{theorem}\label{Theorem:t_1=t_2=1}
Suppose $t_1=t_2=1$, then the number of strongly exceptional Legendrian $A_3$ links is
\begin{align*}
\begin{split}
\left\{  
\begin{array}{ll}
 8, & ~\text{if}~ t_{0}\geq 6,\\
 7, & ~\text{if}~ t_{0}=5,\\
 6, & ~\text{if}~ t_0=4,\\
4-t_0, & ~\text{if}~ t_0\leq3.
\end{array}
\right.
\end{split}
\end{align*}
\end{theorem}

\begin{theorem}\label{Theorem:t_1>1t_2=1}
Suppose $t_1>1$ and $t_2=1$, then the number of strongly exceptional Legendrian $A_3$ links is
\begin{align*}
\begin{split}
\left\{  
\begin{array}{ll}
 12, & ~\text{if}~ t_0\geq5~\text{and}~ t_1=2,\\
 10, & ~\text{if}~ t_0=4~\text{and}~ t_1=2,\\
 8, & ~\text{if}~ t_0=3~\text{and}~ t_1=2,\\
 16, & ~\text{if}~ t_0\geq5~\text{and}~ t_1\geq3,\\
14, & ~\text{if}~ t_0=4~\text{and}~ t_1\geq3,\\
12, & ~\text{if}~ t_0=3~\text{and}~ t_1\geq4,\\
11, & ~\text{if}~ t_0=t_1=3,\\
6-2t_0, & ~\text{if}~ t_0\leq2.
\end{array}
\right.
\end{split}
\end{align*}
\end{theorem}

\begin{theorem} \label{Theorem:t_1>1t_2>1} Suppose $t_{1}>1$ and $t_{2}>1$, the number of strongly exceptional Legendrian $A_3$ links is
\begin{align*}
\begin{split}
\left\{  
\begin{array}{ll}
 18, & ~\text{if}~ t_0\geq4~\text{and}~ t_1=t_2=2,\\
 14, & ~\text{if}~ t_0=3~\text{and}~ t_1=t_2=2,\\
 10, & ~\text{if}~ t_0=2~\text{and}~ t_1=t_2=2,\\
24, & ~\text{if}~ t_0\geq4, t_1\geq3 ~\text{and}~ t_{2}=2,\\
20, & ~\text{if}~ t_0=3, t_1\geq3 ~\text{and}~ t_{2}=2,\\
16, & ~\text{if}~ t_0=2, t_1\geq3 ~\text{and}~ t_{2}=2,\\
 32, & ~\text{if}~ t_0\geq4, t_1\geq3 ~\text{and}~ t_{2}\geq3,\\
28, & ~\text{if}~ t_0=3, t_1\geq3 ~\text{and}~ t_{2}\geq3,\\
24, & ~\text{if}~ t_0=2, t_1\geq3~\text{and}~ t_{2}\geq3,\\
8-4t_0, & ~\text{if}~ t_0\leq1.
\end{array}
\right.
\end{split}
\end{align*}
\end{theorem}

\begin{theorem} \label{Theorem:t_1<0t_2=1}
Suppose $t_{1}<0$ and $t_{2}=1$, then the number of strongly exceptional Legendrian $A_3$ links is
\begin{align*}
\begin{split}
\left\{  
\begin{array}{ll}
4-4t_1, & ~\text{if}~ t_0\geq4,\\
4-3t_1, & ~\text{if}~ t_0=3,\\
4-2t_1, & ~\text{if}~ t_0=2,\\
t_0t_1-2t_1, & ~\text{if}~ t_0\leq1.
\end{array}
\right.
\end{split}
\end{align*}
\end{theorem}

\begin{theorem} \label{Theorem:t_1<0t_2>1} Suppose $t_{1}<0$ and $t_{2}>1$, then the number of strongly exceptional Legendrian $A_3$ links is
\begin{align*}
\begin{split}
\left\{  
\begin{array}{ll}
6-6t_1, & ~\text{if}~ t_0\geq3, t_2=2, \\
6-4t_1, & ~\text{if}~ t_0=2, t_2=2, \\
6-2t_1, & ~\text{if}~ t_0=1, t_2=2, \\
8-8t_1, & ~\text{if}~ t_0\geq3, t_2\geq3, \\
8-6t_1, & ~\text{if}~ t_0=2, t_2\geq3, \\
8-4t_1, & ~\text{if}~ t_0=1, t_2\geq4, \\
8-3t_1,  & ~\text{if}~ t_0=1, t_2=3, \\
2t_0t_1-2t_1, & ~\text{if}~ t_0\leq0.
\end{array}
\right.
\end{split}
\end{align*} 
\end{theorem}

\begin{theorem} \label{Theorem:t1=0} Suppose $t_1=0$, then the number of strongly exceptional Legendrian $A_3$ links is
\begin{align*}
\begin{split}
\left\{  
\begin{array}{ll}
2-2t_{2}, & ~\text{if}~ t_2\leq0, \\
4, & ~\text{if}~ t_2=1, \\
6, & ~\text{if}~ t_2=2, \\
8, & ~\text{if}~ t_2\geq3.
\end{array}
\right.
\end{split}
\end{align*}
\end{theorem}

By exchange of the roles of $K_1$ and $K_2$ as necessary, we have covered all cases. Therefore, we have completely classified strongly exceptional Legendrian $A_3$ links. 
The reader can look up the explicit rotation numbers and corresponding $d_3$-invariants in Lemmas~\ref{t0>1t1<0t2<0}-\ref{t0<0t1<0t2<0},  \ref{t0>5t1=1t2=1}-\ref{t0<2t1>1t2>1}, \ref{t0>3t1<0t2=1}-\ref{t0<1t1<0t2>1}, \ref{t1=0t2<1}-\ref{t1=0t2>2} of
Section~\ref{Section:realizations}. In particular, we have:

\begin{theorem}\label{classification}
The strongly exceptional Legendrian $A_3$ links are determined up to coarse equivalence by their Thurston-Bennequin invariants and rotation numbers.
\end{theorem}

\begin{remark}  Strongly exceptional Legendrian $A_3$ links exist only in overtwisted contact 3-spheres with $d_3 = -\frac{3}{2}, -\frac{1}{2}, \frac{1}{2}, \frac{3}{2}, \frac{5}{2}$. 
\end{remark}

\begin{remark}\label{Remark:Stabilization}
Suppose $t_1, t_2\neq0$. If $t_{0}+\lceil-\frac{1}{t_1}\rceil+\lceil-\frac{1}{t_2}\rceil\geq2$, then any strongly exceptional Legendrian $A_3$ link can be destabilized at the component $K_0$  to another strongly exceptional one. If $t_{0}+\lceil-\frac{1}{t_1}\rceil+\lceil-\frac{1}{t_2}\rceil<1$, then any strongly exceptional Legendrian $A_3$ link can be destabilized at the component $K_0$ to a strongly exceptional Legendrian link with $t_{0}+\lceil-\frac{1}{t_1}\rceil+\lceil-\frac{1}{t_2}\rceil=1$. In the cases either $t_{1}=0$ or $t_{2}=0$, any strongly exceptional Legendrian $A_3$ link can be destabilized at the component $K_0$ to another strongly exceptional one. Furthermore, if $t_{1}=0$, then any strongly exceptional Legendrian $A_3$ link can be destabilized at the component $K_2$ to another strongly exceptional one unless $t_2=0$. On the other hand, a positive (or negative) stabilization at the component $K_0$ (and $K_2$ in the case $t_1=0$) of a strongly exceptional Legendrian $A_3$ link is strongly exceptional if and only if the resulted rotation numbers are indeed the rotation numbers of a strongly exceptional Legendrian $A_3$ link. Therefore, one can read out the mountain ranges of $K_0$ (and $K_2$ in the case $t_1=0$) through the Thurston-Bennequin invariants, rotation numbers and $d_3$-invariants shown in Section~\ref{Section:realizations}. Section~\ref{Section:Stabilizations} explains how strongly exceptional Legendrian representatives relate to each other.  Detailed analysis of the (de)stabilizations, as well as detailed analysis of the mountain range of $K_2$ for the links in Theorem~\ref{Theorem:t1=0}, will be presented in Section~\ref{Section:Stabilizations}. 
\end{remark}

The following is the structure of this paper. Section 2 presents upper bounds for appropriate tight contact structures on $\Sigma\times S^1$. In section 3, we discuss various methods to realize the strongly exceptional Legendrian  $A_3$ links. Section 4 focuses on the realization of  the strongly exceptional Legendrian  $A_3$ links, including the calculation of their rotation numbers and the $d_3$-invariants of their ambient contact $S^3$. In section 5, we explore the stabilizations among the strongly exceptional Legendrian $A_3$ links. Finally, the last section provides a detailed computation as a sample, showcasing the calculation of rotation numbers and $d_3$-invariants.

\medskip
\begin{acknowledgements}
The authors would like to thank John Etnyre for some correspondence. They are also grateful to the
referee for valuable suggestions. The first author was partially supported by Grants No. 12271349 of the National Natural Science Foundation of China. The second author was partially supported by the Turkish Fulbright Commission, IMSA Visitor Program, T\"UB\.ITAK 2219 and T\"UB\.ITAK Grant No. 119F411.
\end{acknowledgements}

\section{Tight contact structures on $\Sigma\times S^1$}

For $i=0,1,2$, let $N(K_i)$ be the standard neighborhood of $K_i$ in a contact 3-sphere.  The Seifert longitude and meridian of $K_i$ are denoted by $\lambda_i$ and $\mu_i$, respectively.  The exterior of the link $A_3=K_{0}\cup K_{1}\cup K_{2}$, $\overline{S^{3}\setminus (N(K_0)\cup N(K_1)\cup N(K_2))}$, is diffeomorphic to $\Sigma\times S^1$, where $\Sigma$ is a pair of pants. Suppose $\partial\Sigma=c_0\cup c_1\cup c_2$ as shown in Figure~\ref{Figure:Sigma}.  Let $h$ denote the $S^1$ factor, namely the vertical circle. Then $\lambda_0=c_0$, $\lambda_1=\lambda_2=h$, $\mu_0=h$, $\mu_1=-c_1$, $\mu_2=-c_2$. Suppose $\partial(\Sigma\times S^1)=T_0\cup T_1\cup T_2$, where $T_{i}=c_{i}\times S^{1}$. Then the dividing set of $T_0$ has slope $t_0$, i.e., has the homology $[c_{0}]+t_{0}[h]$, and the dividing set of $T_i$ has slope $-\frac{1}{t_i}$, i.e., has the homology $-t_{i}[c_{i}]+[h]$, for $i=1,2$. Furthermore, each boundary torus of the exterior of a Legendrian $A_3$ link is minimal convex. Namely, its dividing set consists of exactly two parallel simple closed curves.

\begin{figure}[htb]
\begin{overpic}
{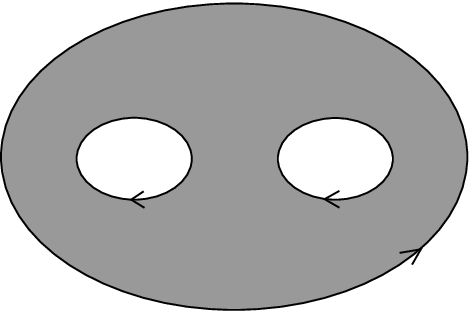}
\put(190, 12){$c_0$}
\put(52, 60){$c_1$}
\put(150, 60){$c_2$}
\end{overpic}

\caption{A pair of pants $\Sigma$. }
\label{Figure:Sigma}
\end{figure}

Following \cite{wu}, we say that a tight contact structure $\xi$ on $\Sigma\times S^1$ with minimal convex boundary is \textit{appropriate} if there is no contact embedding of $$(T^2\times [0,1], \ker(\sin(\pi z)dx+\cos(\pi z)dy))$$ into $(M,\xi)$, where $T^{2}\times \{0\}$ is isotopic to a boundary component of $\Sigma\times S^1$. A Legendrian representation of the $A_3$ link in an overtwisted contact 3-sphere is strongly exceptional if and only if its exterior is an appropriate tight contact $\Sigma\times S^1$. 

In this section, we study the appropriate tight contact structures on $\Sigma\times S^1$ with minimal convex boundary. The  boundary slopes are $s_0=s(T_0)=t_0$, $s_1=s(T_1)=-\frac{1}{t_1}$ and $s_2=s(T_2)=-\frac{1}{t_2}$, where $t_0, t_1, t_2$ are integers.

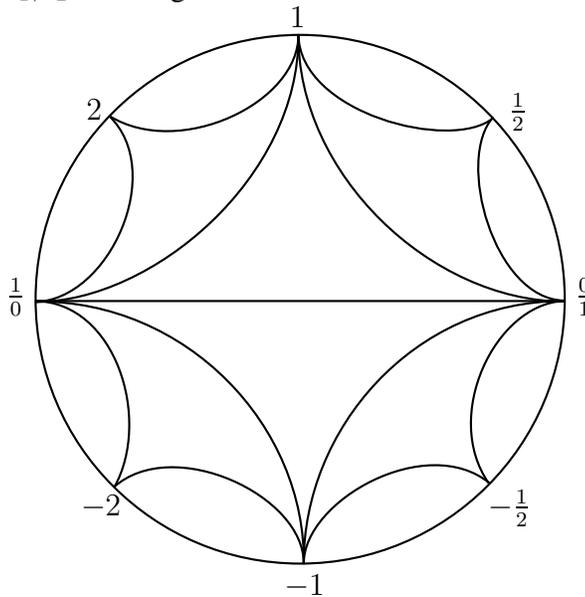
\begin{figure}[htb]
\begin{overpic}
{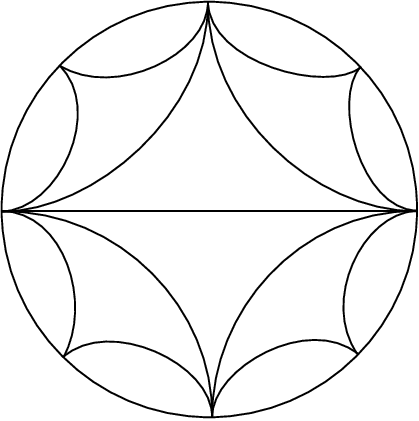}
\put(205, 100){$\frac{0}{1}$}
\put(-10, 100){$\frac{1}{0}$}
\put(95, -10){$-1$}
\put(97, 205){$1$}
\put(20, 170){$2$}
\put(180, 170){$\frac{1}{2}$}
\put(172, 20){$-\frac{1}{2}$}
\put(18, 20){$-2$}
\end{overpic}
\caption{Farey graph on the Poincare disk $\mathbb{H}^2$. }
\label{Figure:farey}
\end{figure}

\begin{lemma}\cite{h1}\label{Lemma:bypass}
Let $T^2$ be a convex surface in a contact 3-manifold with $\# \Gamma_{T^2}=2$ and slope $s$. If a bypass $D$ is attached to $T^2$ from the front (the back, resp.) along a Legendrian ruling curve of slope $r\neq s$, then the resulting convex surface $\tilde{T}^2$ will have $\# \Gamma_{\tilde{T}^2}=2$ and the slope $s'$
which is obtained as follows: Take the arc $[r,s]\subset\partial \mathbb{H}^2$ obtained by starting from $r$  and moving counterclockwise (clockwise, resp.) until we hit $s$, where $\mathbb{H}^2$ is the Poincare disk shown in Figure~\ref{Figure:farey}. On this arc, let $s'$ be the point that is closest to $r$ and has an edge from $s'$ to $s$.

\end{lemma}

Every vertical circle in a contact $\Sigma\times S^1$ has a canonical framing  that arises from the product structure. Let $\gamma$ be a Legendrian circle that lies in the vertical direction. The twisting number $t(\gamma)$ of $\gamma$ measures the amount by which the contact framing of $\gamma$ deviates from the canonical framing.  If $t(\gamma)=0$, then we say that $\gamma$ a \textit{$0$-twisting vertical Legendrian circle}.

\begin{lemma}\label{vertical}
Suppose $\xi$ is an appropriate tight contact structure on $\Sigma\times S^1$ with boundary slopes $s_0=t_0$, $s_i=-\frac{1}{t_i}$ for $i=1,2$. If $t_{1}, t_{2}\neq0$ and $t_{0}+\lceil-\frac{1}{t_1}\rceil+\lceil-\frac{1}{t_2}\rceil\geq2$, then $\xi$ has a $0$-twisting vertical Legendrian circle.
\end{lemma}

\begin{proof}
We assume the Legendrian rulings on $T_{1}$ and $T_{2}$ to have infinite slopes. Consider a convex vertical annulus $A$ such that the boundary consists of a Legendrian ruling on $T_1$ and a Legendrian ruling on $T_2$. The dividing set of $A$ intersects $T_i$, $i=1,2$, in exactly $2|t_{i}|$ points. If every dividing curve of $A$ is boundary parallel, then there exists a $0$-twisting vertical Legendrian circle in $A$. So we assume that there exist dividing arcs on $A$, which connect the two boundary components of $A$. If there is a boundary parallel dividing curve on $A$, then we perform a bypass attachment (attached from the back of $T_i$) to eliminate it. 

(1) Suppose $t_{1}<0$ and $t_{2}<0$. By Lemma~\ref{Lemma:bypass}, we can obtain a submanifold $\tilde{\Sigma}\times S^{1}$ of $\Sigma\times S^{1}$ whose boundary is $T_{0}\cup \tilde{T}_{1}\cup \tilde{T}_{2}$, where both $\tilde{T}_{1}$ and $\tilde{T}_{2}$ have slopes $-\frac{1}{t_3}$ for some integer $t_{3}\in [\max\{t_{1},t_{2}\},-1]$. Moreover, each dividing curve on $\tilde{A}=A\cap(\tilde{\Sigma}\times S^{1})$ connects the two boundary components. Let $N$ be a neighborhood of $\tilde{T}_{1}\cup\tilde{T}_{2}\cup \tilde{A}$, and $\partial N=\tilde{T}_{1}\cup\tilde{T}_{2}\cup \tilde{T}$. Then, by edge-rounding, $\tilde{T}$ has slope $\frac{1}{t_3}+\frac{1}{t_3}+\frac{1}{-t_3}=\frac{1}{t_3}$ (as seen form $T_0$). Therefore, the thickened torus $\tilde{\Sigma}\times S^{1}\setminus N$ has boundary slopes $t_0$ and $\frac{1}{t_3}$. Since $t_0\geq0>\frac{1}{t_3}$, there must exists a $0$-twisting vertical Legendrian circle in this thickened torus, and hence in $\Sigma\times S^{1}$.

(2) Suppose $t_{1}=1$ and $t_{2}=1$. It follows from \cite[Lemma 5.1]{h2}.

(3) Suppose $t_{1}>1$ and $t_{2}=1$. By Lemma~\ref{Lemma:bypass}, we can obtain a submanifold $\tilde{\Sigma}\times S^{1}$ of $\Sigma\times S^{1}$ whose boundary is $T_{0}\cup \tilde{T}_{1}\cup T_{2}$, where $\tilde{T}_{1}$ has slope $0$. Moreover, each dividing curve on $\tilde{A}=A\cap(\tilde{\Sigma}\times S^{1})$ connects the two boundary components. Let $N$ be a neighborhood of $\tilde{T}_{1}\cup T_{2}\cup \tilde{A}$, and $\partial N=\tilde{T}_{1}\cup T_{2}\cup \tilde{T}$. Then, by edge-rounding, $\tilde{T}$ has slope $0+1+1=2$ (as seen form $T_0$). Therefore, the thickened torus $\tilde{\Sigma}\times S^{1}\setminus N$ has boundary slopes $t_0$ and $2$. Since $t_0\geq3>2$, there must exist a $0$-twisting vertical Legendrian circle in this thickened torus, and hence in $\Sigma\times S^{1}$. 

(4) Suppose $t_{1}>1$ and $t_{2}>1$. We divide this case into two subcases:

(i) There exist boundary parallel dividing curves on $A$. By Lemma~\ref{Lemma:bypass}, we can obtain a submanifold $\tilde{\Sigma}\times S^{1}$ of $\Sigma\times S^{1}$ whose boundary is $T_{0}\cup \tilde{T}_{1}\cup \tilde{T}_{2}$, where both $\tilde{T}_{1}$ and $\tilde{T}_{2}$ have slopes $0$. Moreover, each dividing curve on $\tilde{A}=A\cap(\tilde{\Sigma}\times S^{1})$ connects the two boundary components. Let $N$ be a neighborhood of $\tilde{T}_{1}\cup\tilde{T}_{2}\cup \tilde{A}$, and $\partial N=\tilde{T}_{1}\cup\tilde{T}_{2}\cup \tilde{T}$. Then, by edge-rounding, $\tilde{T}$ has slope $0+0+1=1$ (as seen form $T_0$). Therefore, the thickened torus $\tilde{\Sigma}\times S^{1}\setminus N$ has boundary slopes $t_0$ and $1$. Since $t_0\geq2>1$, there must exist a $0$-twisting vertical Legendrian circle in this thickened torus, and hence in $\Sigma\times S^{1}$. 

(ii) There exists no boundary parallel dividing curve on $A$. Then $t_1=t_2$ and all dividing curves on $A$ connect the two boundary components of $A$. Let $N$ be a neighborhood of $T_{1}\cup T_{2}\cup \tilde{A}$, and $\partial N=T_{1}\cup T_{2}\cup \tilde{T}$.  Then, by edge-rounding, $\tilde{T}$ has slope $\frac{1}{t_1}+\frac{1}{t_1}+\frac{1}{t_1}=\frac{3}{t_1}$ (as seen form $T_0$). Therefore, the thickened torus $\Sigma\times S^{1}\setminus N$ has boundary slopes $t_0$ and $\frac{3}{t_1}$. Since $t_0\geq2>\frac{3}{t_1}$, there must exist a $0$-twisting vertical Legendrian circle in this thickened torus, and hence in $\Sigma\times S^{1}$. 

(5) Suppose $t_{1}<0$ and $t_{2}=1$. There are boundary parallel dividing curves on $A$. By Lemma~\ref{Lemma:bypass}, we can obtain a submanifold $\tilde{\Sigma}\times S^{1}$ of $\Sigma\times S^{1}$ whose boundary is $T_{0}\cup \tilde{T}_{1}\cup T_{2}$, where both $\tilde{T}_{1}$  have slopes $1$. Moreover, each dividing curve on $\tilde{A}=A\cap(\tilde{\Sigma}\times S^{1})$ connects the two boundary components. Let $N$ be a neighborhood of $\tilde{T}_{1}\cup T_{2}\cup \tilde{A}$, and $\partial N=\tilde{T}_{1}\cup T_{2}\cup \tilde{T}$. Then, by edge-rounding, $\tilde{T}$ has slope $1+(-1)+1=1$ (as seen form $T_0$). Therefore, the thickened torus $\tilde{\Sigma}\times S^{1}\setminus N$ has boundary slopes $t_0$ and $1$. Since $t_0\geq2>1$, there must exist a $0$-twisting vertical Legendrian circle in this thickened torus, and hence in $\Sigma\times S^{1}$. 

(6) Suppose $t_{1}<0$ and $t_{2}>1$. We divide this case into two subcases.

(i) If there exist boundary parallel dividing curves on $A$ whose boundary points belong to $A\cap T_2$, we can use Lemma~\ref{Lemma:bypass} to obtain a submanifold $\tilde{\Sigma}\times S^{1}$ of $\Sigma\times S^{1}$ whose boundary is 
$T_{0}\cup \tilde{T}_{1}\cup \tilde{T}_{2}$, where $\tilde{T}_{1}$ has slope $1$ and $\tilde{T}_{2}$ has slope $0$. Furthermore, each dividing curve on $\tilde{A}=A\cap(\tilde{\Sigma}\times S^{1})$ connects the two boundary components. Let $N$ be a neighborhood of $\tilde{T}_{1}\cup\tilde{T}_{2}\cup \tilde{A}$, and $\partial N=\tilde{T}_{1}\cup\tilde{T}_{2}\cup \tilde{T}$. By performing edge-rounding, $\tilde{T}$ will have slope $-1+0+1=0$ (as seen form $T_0$). Therefore, the thickened torus $\tilde{\Sigma}\times S^{1}\setminus N$ has boundary slopes $t_0$ and $1$. Since $t_0\geq1>0$, there must exist a $0$-twisting vertical Legendrian circle in this thickened torus, and hence in $\Sigma\times S^{1}$. 

(ii) If there are no boundary parallel dividing curves on $A$ whose boundary points belong to $A\cap T_2$, we can use Lemma~\ref{Lemma:bypass} to obtain a submanifold $\tilde{\Sigma}\times S^{1}$ of $\Sigma\times S^{1}$ whose boundary is $T_{0}\cup \tilde{T}_{1}\cup T_{2}$, where $\tilde{T}_{1}$ has slope $\frac{1}{t_2}$. Furthermore, each dividing curve on $\tilde{A}=A\cap(\tilde{\Sigma}\times S^{1})$ connects the two boundary components. Let $N$ be a neighborhood of $\tilde{T}_{1}\cup T_{2}\cup \tilde{A}$, and $\partial N=\tilde{T}_{1}\cup T_{2}\cup \tilde{T}$. By performing edge-rounding, $\tilde{T}$ will have slope $-\frac{1}{t_2}+\frac{1}{t_2}+\frac{1}{t_2}=\frac{1}{t_2}$ (as seen form $T_0$). Therefore, the thickened torus $\tilde{\Sigma}\times S^{1}\setminus N$ has boundary slopes $t_0$ and $1$. Since $t_0\geq1>\frac{1}{t_2}$, there must exist a $0$-twisting vertical Legendrian circle in this thickened torus, and hence in $\Sigma\times S^{1}$.
\end{proof}

\begin{lemma} \label{Lemma:slopes1} If $\xi$ is a tight contact structure on $\Sigma\times S^1$ with boundary slopes $s_0=t_0$, $s_i=-\frac{1}{t_i}$ for $i=1,2$, has a $0$-twisting vertical Legendrian circle, where $t_1,t_2\neq0$. Then it admits a factorization $\Sigma\times S^1= L'_{0}\cup L'_{1}\cup L'_{2}\cup \Sigma'\times S^1$, where $L'_{i}$ are disjoint thickened tori with minimal twisting and minimal convex boundary $\partial L'_{i}=T_{i}-T'_{i}$, and all the components of $\partial \Sigma'\times S^1=T'_{0}\cup T'_{1}\cup T'_{2}$ have boundary slopes $\infty$.
\end{lemma}

\begin{proof}
The proof is similar to that of \cite[Lemma 5.1, Part 1]{h2}.
\end{proof}

Let  $\xi$ be a contact structure on $\Sigma\times S^1$ with boundary slopes $s_0=t_0$, $s_i=-\frac{1}{t_i}$ for $i=1,2$, where $t_1, t_2\neq 0$. Assume it admits a factorization $\Sigma\times S^1= L'_{0}\cup L'_{1}\cup L'_{2}\cup \Sigma'\times S^1$, where $L'_{i}$ are disjoint thickened tori with minimal twisting and minimal convex boundary $\partial L'_{i}=T_{i}-T'_{i}$, and all the components of $\partial \Sigma'\times S^1=T'_{0}\cup T'_{1}\cup T'_{2}$ have boundary slopes $\infty$.  
Then, in the thickened torus $L'_i$, $i=1,2$, there exists a basic slice $B'_i$ with one boundary component $T'_{i}$ and another boundary slope $\lceil-\frac{1}{t_i}\rceil$. This is because $\lceil-\frac{1}{t_i}\rceil$ is counter-clockwise of $-\frac{1}{t_i}$ and clockwise of $\infty$ in the Farey graph shown in Figure~\ref{Figure:farey}.  Let $C'_i$, $i\in\{1,2\}$, be the continued fraction block in $L'_i$ that contains $B'_i$. The basic slices in $C'_i$ can be shuffled. Namely, any basic slice in $C'_i$ can be shuffled to be $B'_i$. 

\begin{lemma} \label{Lemma:slopes2} (1) Suppose $t_{0}+\lceil-\frac{1}{t_1}\rceil+\lceil-\frac{1}{t_2}\rceil=3$. If the signs of $L'_0$, $B'_1$ and $B'_2$ are the same, then the restriction of $\xi$ to $L'_0 \cup B'_1\cup B'_2 \cup \Sigma'\times S^{1}$ remains unchanged if we change the three signs simultaneously.\\
(2) Suppose $t_{0}+\lceil-\frac{1}{t_1}\rceil+\lceil-\frac{1}{t_2}\rceil\leq2$. If the signs of $L'_0$, $B'_1$, and $B'_2$ are the same, then $\xi$ is overtwisted.
\end{lemma}

\begin{proof} The restriction of $\xi$ on $L'_0\cup B'_1\cup B'_2\cup \Sigma'\times S^1$ has boundary slopes $t_0$, $\lceil-\frac{1}{t_1}\rceil$ and $\lceil-\frac{1}{t_2}\rceil$. So the lemma follows by applying \cite[Lemma 5.1]{h2} to $L'_0\cup B'_1\cup B'_2\cup \Sigma'\times S^1$. 
\end{proof}

\begin{lemma} \cite{emm}\label{Lemma:slopes3}
There is a unique appropriate tight contact structure on $\Sigma\times S^1$ whose three boundary slopes are all $\infty$ up to isotopy (not fixing the boundary point-wise, but preserving it set-wise).
\end{lemma}

\begin{lemma}\label{Lemma:tor>0}
Let $\xi$ be a contact structure on $\Sigma\times S^1$. Assume that each  $T_i$ is minimal convex with dividing curves of finite slope $t_0$, $-\frac{1}{t_1}$ and $-\frac{1}{t_2}$. If $\xi$ has $0$-twisting vertical Legendrian circles and $t_{0}+\lceil-\frac{1}{t_1}\rceil+\lceil-\frac{1}{t_2}\rceil\leq1$, then $\xi$ is not appropriate tight.
\end{lemma}

\begin{proof} As there is a $0$-twisting vertical Legendrian circle, 
there exists a minimal convex torus $T'_i$, parallel to $T_i$, with slope  $\lceil-\frac{1}{t_i}\rceil$, $i=1,2$.   Consider a convex annulus $\tilde{A}$ with a boundary consisting of a Legendrian ruling on $T'_1$ and a Legendrian ruling on $T'_2$. Let $N$ be a neighborhood of $T'_{1}\cup T'_{2}\cup \tilde{A}$, and $\partial N=T'_{1}\cup T'_{2}\cup \tilde{T}$. Then through edge-rounding, $\tilde{T}$ has slope $-\lceil-\frac{1}{t_1}\rceil-\lceil-\frac{1}{t_2}\rceil+1$ (as seen form $T_0$). We obtain a thickened torus with boundary slopes $t_0$ and $-\lceil-\frac{1}{t_1}\rceil-\lceil-\frac{1}{t_2}\rceil+1$, and a boundary parallel convex torus with slope $\infty$.  Thus, from $t_0\leq -\lceil-\frac{1}{t_1}\rceil-\lceil-\frac{1}{t_2}\rceil+1$, it follows that the Giroux torsion of this thickened torus is at least $1$. Hence the Lemma holds.  
\end{proof}

\begin{lemma}\label{Lemma:slopes}
Let $\xi$ be an appropriate tight contact structure on $\Sigma\times S^1$. Assume that each $T_i$ is
minimal convex with dividing curves of finite slope $t_0$, $-\frac{1}{t_1}$ and $-\frac{1}{t_2}$. Suppose $\xi$ has no $0$-twisting vertical Legendrian circle. Then there exist collar neighborhoods $L''_i$ of 
$T_i$  for  $i=1,2$ satisfying that $\Sigma\times S^1=\Sigma''\times S^1\cup L''_1\cup L''_2$ and the boundary slope of $\Sigma''\times S^1$ are $t_0$, $\lceil{-\frac{1}{t_1}}\rceil$ and $\lceil{-\frac{1}{t_2}}\rceil$.
\end{lemma}

\begin{proof}
We modify the Legendrian rulings on $T_{0}$ and $T_{i}$ to have infinite slopes. Consider a convex vertical annulus $A$ whose boundary consists of Legendrian rulings on $T_0$ and $T_i$. The dividing set of $A$ intersects $T_0$ in exactly $2$ points. The dividing set of $A$ intersects $T_i$, $i=1,2$, in exactly $2|t_{i}|$ points. As $\xi$ has no $0$-twisting vertical Legendrian circle, there exist dividing arcs on $A$ that connect the two boundary components of $A$. If there is a boundary parallel dividing curve on $A$, then its endpoints must belong to $A\cap T_i$ for some $i=1,2$. We perform a bypass (attached from the back of $T_i$) to eliminate it. Applying Lemma~\ref{Lemma:bypass}, we obtain a thickened torus $L''_{i}$ for $i=1,2$ that satisfies $\Sigma\times S^1=\Sigma''\times S^1\cup L''_1\cup L''_2$ and the boundary slope of $\Sigma''\times S^1$ are $t_0$, $\lceil{-\frac{1}{t_1}}\rceil$ and $\lceil{-\frac{1}{t_2}}\rceil$.
\end{proof}

Now we present upper bounds for appropriate
tight contact structures on $\Sigma\times S^1$.

\subsection{$t_{1}<0$ and $t_{2}<0$.}
\begin{lemma}\label{Lemma:t1<0t2<0}
Suppose $t_{1}<0$ and $t_{2}<0$, then there are at most
\begin{align*}
\begin{split}
\left\{  
\begin{array}{ll}
2t_{1}t_{2}-2t_{1}-2t_{2}+2, & ~\text{if}~ t_{0}\geq 2,\\
t_{1}t_{2}-2t_{1}-2t_{2}+2, & ~\text{if}~ t_{0}=1,\\
-2t_{1}-2t_{2}+2, & ~\text{if}~ t_0=0,\\
-t_{0}t_{1}t_{2}, & ~\text{if}~ t_0\leq-1,
\end{array}
\right.
\end{split}
\end{align*}
appropriate tight contact structures on $\Sigma\times S^1$ with the given boundary slopes.
\end{lemma}

\begin{proof} By Lemma~\ref{vertical}, if $t_{0}\geq0$, then the tight contact structures on $\Sigma\times S^1$ always exist $0$-twisting vertical Legendrian circles.  

If an appropriate contact structure $\xi$ on $\Sigma\times S^1$ has a $0$-twisting vertical Legendrian circle, then Lemma~\ref{Lemma:slopes1} tells us that   $\Sigma\times S^1$ can be factored into $L'_{0}\cup L'_{1}\cup L'_{2} \cup \Sigma'\times S^1$, where the boundary slopes of $\Sigma'\times S^1$ are all $\infty$, the boundary slopes of $L'_{0}$ are $\infty$ and $t_0$ the boundary slopes of $L'_{i}$ are $\infty$ and $-\frac{1}{t_i}$ for $i=1,2$. Moreover, There are $2$ minimally twisting tight contact structures on $L'_{0}$.

If $t_{i}< 0$, $i=1,2$, we have
$$\begin{bmatrix} 
0 & -1 \\
1 & 1
\end{bmatrix}
\begin{bmatrix} 
0  \\
1 
\end{bmatrix}=\begin{bmatrix} 
-1  \\
1
\end{bmatrix}, 
\begin{bmatrix} 
0 & -1 \\
1 & 1
\end{bmatrix}
\begin{bmatrix} 
-t_{i}  \\
1 
\end{bmatrix}=\begin{bmatrix} 
-1  \\
-t_{i}+1
\end{bmatrix}.$$
The thickened torus $L_i$ is a continued fraction block with $-t_{i}$ basic slices, and therefore admits $-t_{i}+1$ minimally twisting  tight contact structures.

By applying Lemma~\ref{Lemma:slopes3}, we can conclude that there are at most $2t_{1}t_{2}-2t_{1}-2t_{2}+2$ appropriate tight contact structures on $\Sigma\times S^1$ if $t_0\geq2$. If $t_0=1$ and there are  basic slices in $L'_i$ which have the same signs as that of $L'_0$ for $i=1,2$, then after shuffling, we can assume that $L'_0$, $B'_1$ and $B'_2$ have the same signs. According to Lemma~\ref{Lemma:slopes2}, a tight contact structure that has positive basic slices in $L'_{i}$ for $i=0,1,2$ are isotopic to a tight contact structure which is obtained by changing a positive basic slice in $L'_{i}$ for $i=0,1,2$ to a negative basic slice. Therefore, there are at most $t_{1}t_{2}-2t_{1}-2t_{2}+2$ appropriate tight contact structures on $\Sigma\times S^1$ if $t_0=1$.
If $t_0=0$, then by Lemma~\ref{Lemma:slopes2}, a contact structure which has positive basic slices in $L'_{i}$ for $i=0,1,2$ are overtwisted. Thus, there are at most $-2t_{1}-2t_{2}+2$ appropriate tight contact structures on $\Sigma\times S^1$ if $t_0=0$.

Suppose $t_0\leq-1$. By Lemma~\ref{Lemma:tor>0}, there are no  appropriate tight contact structures having $0$-twisting vertical Legendrian circle. We consider the appropriate tight contact structures without a $0$-twisting vertical Legendrian circle. By Lemma~\ref{Lemma:slopes}, we can factorize $\Sigma\times S^1=\Sigma''\times S^1\cup L''_{1}\cup L''_{2}$, where the boundary slopes of $\Sigma''\times S^1$ are $t_0$, $1$ and $1$, the boundary slopes of $L''_{i}$ are $1$ and $-\frac{1}{t_i}$ for $i=1,2$. Since $t_0<0$, by \cite[Lemma 5.1]{h2}, there are exactly $-t_0$ tight contact structures on $\Sigma''\times S^1$ without any $0$-twisting  vertical Legendrian circle. By \cite[Theorem 2.2]{h1}, there are $-t_{i}$ minimally twisting tight contact structures on $L''_{i}$ for $i=1,2$. Therefore, there are at most $-t_{0}t_{1}t_{2}$ tight contact structures on $\Sigma\times S^1$ without any $0$-twisting vertical Legendrian curve and with boundary slopes $s_0=t_0$, $s_i=-\frac{1}{t_i}$ for $i=1,2$.
\end{proof}

To denote the $2t_{1}t_{2}-2t_{1}-2t_{2}+2$ contact structures on $\Sigma\times S^1$ with $0$-twisting vertical Legendrian circle, We use the decorations $(\pm)(\underbrace{\pm \cdots \pm}_{-t_{1}})(\underbrace{\pm \cdots \pm}_{-t_{2}})$. See Figure~\ref{Figure:Sigma1} for an example. The sign in the first bracket corresponds to the sign of the basic slice $L'_0$, while the signs in the second and the third brackets correspond to the signs of the basic slices in $L'_1$ and $L'_2$, respectively. We order the basic slices in $L'_1$ and $L'_2$ from the innermost boundary to the outmost boundary. As both $L'_1$ and $L'_2$ are continued fraction blocks, the signs  in the second and the third brackets can be shuffled. For example, the decorations $(+)(+--)(--)$ and $(+)(--+)(--)$ denote the same contact structures.

\begin{figure}[htb]
\begin{overpic}
{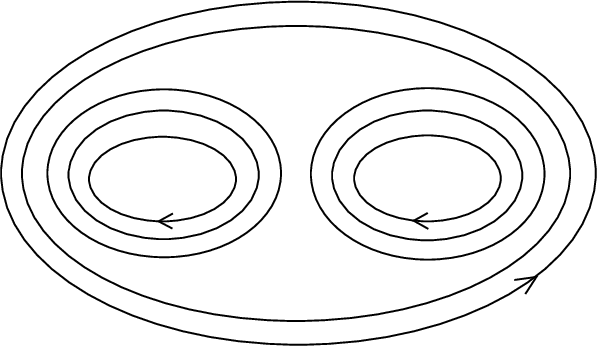}
\put(1, 80){$+$}
\put(23, 80){$-$}
\put(33, 80){$-$}
\put(150, 80){$-$}
\put(160, 80){$-$}
\put(110, 145){$\infty$}
\put(110, 156){$t_0$}
\put(80, 128){$\infty$}
\put(80, 115){$1$}
\put(75, 104){$1/2$}
\put(200, 128){$\infty$}
\put(200, 115){$1$}
\put(195, 104){$1/2$}
\end{overpic}

\caption{A pair of pants $\Sigma$, where $t_{0}=0$, $t_{1}=t_{2}=-2$.}
\label{Figure:Sigma1}
\end{figure}
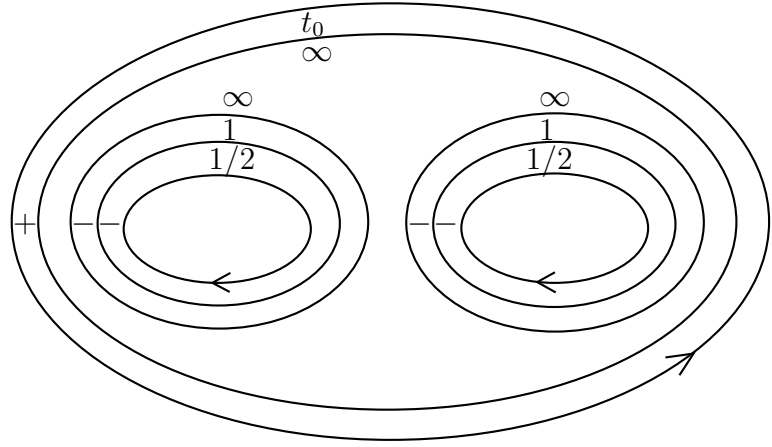

\subsection{$t_{1}>0$ and $t_{2}>0$.}
\begin{lemma}\label{Lemma:t1=t2=1}
Suppose $t_1=t_2=1$; then there are exactly
\begin{align*}
\begin{split}
\left\{  
\begin{array}{ll}
8, & ~\text{if}~ t_{0}\geq 6,\\
7, & ~\text{if}~ t_{0}=5,\\
6, & ~\text{if}~ t_0=4,\\
4-t_0, & ~\text{if}~ t_0\leq3,
\end{array}
\right.
\end{split}
\end{align*}
appropriate tight contact structures on $\Sigma\times S^1$ with the given boundary slopes. 
\end{lemma}

\begin{proof}
The boundary slopes of $\Sigma\times S^1$ are $t_0$, $-1$ and $-1$. If $t_0\leq3$, according to  \cite[Lemma 5.1]{h2}, there are exactly $4-t_0$ appropriate tight contact structures on $\Sigma\times S^1$ without $0$-twisting vertical Legendrian circle. By Lemma~\ref{Lemma:tor>0}, there are no appropriate tight contact structures on $\Sigma\times S^1$ with $0$-twisting 
 vertical Legendrian circle.  If $t_0\geq4$, then any tight contact structure on $\Sigma\times S^1$ has a $0$-twisting vertical Legendrian circle. By applying \cite[Lemma 5.1]{h2} again, we can conclude that when $t_0=4$,  there are exactly $6$ appropriate tight contact structures on $\Sigma\times S^1$.  When $t_0=5$, there are exactly $7$ appropriate tight contact structures on $\Sigma\times S^1$. When $t_0\geq6$, there are exactly $8$ appropriate tight contact structures on $\Sigma\times S^1$. 
\end{proof}

We use the decorations $(\pm)(\pm)(\pm)$ to denote the $8$ contact structures on $\Sigma\times S^1$ with a $0$-twisting vertical Legendrian circle.

\begin{lemma}\label{Lemma:t_1>1t_2=1}
Suppose $t_1>1$ and $t_2=1$, then there are at most
\begin{align*}
\begin{split}
\left\{  
\begin{array}{ll}
12, & ~\text{if}~ t_0\geq5~\text{and}~ t_1=2,\\
10, & ~\text{if}~ t_0=4~\text{and}~ t_1=2,\\
8, & ~\text{if}~ t_0=3~\text{and}~ t_1=2,\\
16, & ~\text{if}~ t_0\geq5~\text{and}~ t_1\geq3,\\
14, & ~\text{if}~ t_0=4~\text{and}~ t_1\geq3,\\
12, & ~\text{if}~ t_0=3~\text{and}~ t_1\geq4,\\
11, & ~\text{if}~ t_0=t_1=3,\\
6-2t_0, & ~\text{if}~ t_0\leq2,
\end{array}
\right.
\end{split}
\end{align*}
appropriate tight contact structures on $\Sigma\times S^1$ with the given boundary slopes. 
\end{lemma}

\begin{proof} 
The boundary slopes of $\Sigma\times S^1$ are $s_0=t_0$, $s_{1}=-\frac{1}{t_1}$ and $s_{2}=-1$. 

If $t_{0}\geq3$, then the tight contact structures on $\Sigma\times S^1$ always exist $0$-twisting vertical Legendrian circles.

If $t_{1}>1$, we have
$$\begin{bmatrix} 
1 & 1 \\
-2 & -1
\end{bmatrix}
\begin{bmatrix} 
0  \\
1 
\end{bmatrix}=\begin{bmatrix} 
1  \\
-1
\end{bmatrix}, 
\begin{bmatrix} 
1 & 1 \\
-2 & -1
\end{bmatrix}
\begin{bmatrix} 
t_{1}  \\
-1 
\end{bmatrix}=\begin{bmatrix} 
t_{1}-1  \\
-2t_{1}+1
\end{bmatrix},$$
$$\frac{-2t_{1}+1}{t_{1}-1}=[-3,\underbrace{-2,\cdots,-2}_{t_{1}-2}].$$ 
If $t_{1}=2$, then $L'_{1}$ is a continued fraction block with two basic slices with slopes $-\frac{1}{2}$, $0$ and $\infty$, and thus admits exactly $3$ tight contact structures. If $t_{1}\geq3$, then $L'_{1}$ consists of two continued fraction blocks, each of which has one basic slice. The slopes are $-\frac{1}{t_1}$, $0$ and $\infty$. Therefore, it admits exactly $4$ tight contact structures. 

If $t_0\geq5$ and $t_1=2$, then there are at most $2\times 3\times 2=12$ tight contact structures. The number of such contact structures depends on the signs of the basic slices in $L'_{i}$ for $i=0,1,2$. If $t_0=4$ and $t_1=2$, then there are at most $10$ tight contact structures by deleting  $2$ duplications. If $t_0\leq3$ and $t_1=2$, then there are at most $8$ tight contact structures by deleting  $4$ overtwisted cases. 

If $t_0\geq5$ and $t_1\geq3$, then there are at most $2\times 4\times 2=16$ tight contact structures. The number of such contact structures depends depend on the signs of the basic slices in $L'_{i}$ for $i=0,1,2$. If $t_0=4$ and $t_1\geq3$, then there are at most $14$ tight contact structures by deleting  $2$ duplications. If $t_0\leq3$ and $t_1\geq3$, then there are at most $12$ tight contact structures by deleting  $4$ overtwisted cases. 

Suppose $t_0\leq2$. By Lemma~\ref{Lemma:tor>0}, there are no appropriate tight contact structures with a $0$-twisting vertical Legendrian circle. We consider the appropriate tight contact structures without $0$-twisting vertical Legendrian circle.  By Lemma~\ref{Lemma:slopes}, we can factorize  $\Sigma\times S^1=\Sigma''\times S^1\cup L''_{1}$, where the boundary slopes of $\Sigma''\times S^1$ are $t_0$, $0$ and $-1$, and the boundary slopes of $L''_{1}$ are $0$ and $-\frac{1}{t_1}$. Since $t_0<3$,  according to \cite[Lemma 5.1]{h2},  there are exactly $3-t_0$ tight contact structures on $\Sigma''\times S^1$ without $0$-twisting vertical Legendrian circle. There are $2$ minimally twisting tight contact structures on $L''_{1}$. Therefore,  there are at most $6-2t_0$ appropriate tight contact structures on $\Sigma\times S^1$ without $0$-twisting vertical Legendrian circle and with boundary slopes $s_0=t_0$, $s_i=-\frac{1}{t_i}$ for $i=1,2$.

If $t_1=2$, then  we denote the $12$ contact structures on $\Sigma\times S^1$ with a $0$-twisting vertical Legendrian circle using the decorations $(\pm)(\pm\pm)(\pm)$. For $t_1\geq3$, we use the decorations $(\pm)((\pm)(\pm))(\pm)$ to denote the $16$ contact structures on $\Sigma\times S^1$ with a $0$-twisting vertical Legendrian circle. In the latter case, $((\pm)(\pm))$ refers to the two signed basic slices in $L'_1$ that do not form a continued fraction block.

If $t_0=t_1=3$ and $t_2=1$, we claim the two decorations $(+)((-)(+))(+)$ and $(-)((+)(-))(-)$ denote the same contact structure on $\Sigma\times S^1$. As before, there is a convex vertical annulus $A$ such that $\partial A$ consists of a Legendrian ruling on $T_0$ and a Legendrian ruling on $T_2$, and the dividing set on $A$ run from one boundary component to the other. If we cut $L'_0\cup L'_1\cup L'_2 \cup \Sigma'\times S^1$ along $A$ we will obtain a thickened torus admitting a factorization into two basic slices with slopes $-\frac{1}{3}$, $0$ and $0$, $-1$, and opposite signs. Here the slope $-1$ is obtained by $-s_0-s_2+1=-3-(-1)+1$. The three slopes can be transformed into $\frac{1}{3}$, $\frac{1}{2}$ and  $1$ as follows,
$$\begin{bmatrix} 
2 & 3 \\
1 & 2
\end{bmatrix}
\begin{bmatrix} 
3  \\
-1 
\end{bmatrix}=\begin{bmatrix} 
3  \\
1
\end{bmatrix},
\begin{bmatrix} 
2 & 3 \\
1 & 2
\end{bmatrix}
\begin{bmatrix} 
1  \\
0 
\end{bmatrix}=\begin{bmatrix} 
2  \\
1
\end{bmatrix},
\begin{bmatrix} 
2 & 3 \\
1 & 2
\end{bmatrix}
\begin{bmatrix} 
-1  \\
1 
\end{bmatrix}=\begin{bmatrix} 
1  \\
1
\end{bmatrix}.
$$
So these two basic slices  form a continued fraction block and can be interchanged. Similar to the argument in \cite[Page 135]{h2}, this leads to an exchange between $(+)((-)(+))(+)$ and $(-)((+)(-))(-)$ while preserving the isotopy classes of contact structures.
\end{proof}

\begin{lemma}\label{Lemma:t_1>1t_2>1}
Suppose $t_1>1$ and $t_2>1$, then there are at most
\begin{align*}
\begin{split}
\left\{  
\begin{array}{ll}
18, & ~\text{if}~ t_0\geq4~\text{and}~ t_1=t_2=2,\\
14, & ~\text{if}~ t_0=3~\text{and}~ t_1=t_2=2,\\
10, & ~\text{if}~ t_0=2~\text{and}~ t_1=t_2=2,\\
24, & ~\text{if}~ t_0\geq4~\text{and}~ t_1\geq3, t_{2}=2,\\
20, & ~\text{if}~ t_0=3~\text{and}~ t_1\geq3, t_{2}=2,\\
16, & ~\text{if}~ t_0=2~\text{and}~ t_1\geq3, t_{2}=2,\\
32, & ~\text{if}~ t_0\geq4~\text{and}~ t_1\geq3, t_{2}\geq3,\\
28, & ~\text{if}~ t_0=3~\text{and}~ t_1\geq3, t_{2}\geq3,\\
24, & ~\text{if}~ t_0=2~\text{and}~ t_1\geq3, t_{2}\geq3,\\
8-4t_0, & ~\text{if}~ t_0\leq1,
\end{array}
\right.
\end{split}
\end{align*}
appropriate tight contact structures on $\Sigma\times S^1$ with the given boundary slopes. 
\end{lemma}

\begin{proof} If $t_{0}\geq2$, then the tight contact structures on $\Sigma\times S^1$  always exist a $0$-twisting vertical Legendrian circles.

If $t_0\geq4$ and $t_{1}=t_{2}=2$, then there are at most $2\times 3\times 3=18$ tight contact structures. If $t_0\geq4$, $t_{1}\geq3$ and $t_{2}=2$, then there are at most $2\times 4\times 3=24$ tight contact structures.  If $t_0\geq4$, $t_{1}\geq3$ and $t_{2}\geq3$, then there are at most $2\times 4\times 4=32$ tight contact structures. The number of such contact structures depends on the signs of the basic slices in $L'_{i}$ for $i=0,1,2$.
For the other cases, the upper bound can be obtained by deleting the duplications or the overtwisted contact structures.

Suppose $t_0\leq1$. By Lemma~\ref{Lemma:tor>0}, there are no  appropriate tight contact structures with a $0$-twisting vertical Legendrian circle. We consider the appropriate tight contact structures without a $0$-twisting vertical Legendrian circle. By Lemma~\ref{Lemma:slopes}, we can factorize  $\Sigma\times S^1=\Sigma''\times S^1\cup L''_{1}\cup L''_{2}$, where the boundary slopes of $\Sigma''\times S^1$ are $t_0$, $0$ and $0$, and the boundary slopes of $L''_{i}$ are $0$ and $-\frac{1}{t_i}$. Since $t_0\leq1$,  according to \cite[Lemma 5.1]{h2}, there are exactly $2-t_0$ tight contact structures on $\Sigma''\times S^1$ without a $0$-twisting  vertical Legendrian circle. There are $2$ minimally twisting tight contact structures on $L''_{i}$. Therefore, there are at most $8-4t_0$ appropriate tight contact structures on $\Sigma\times S^1$ without a $0$-twisting vertical Legendrian circle and with boundary slopes $s_0=t_0$, $s_i=-\frac{1}{t_i}$ for $i=1,2$.
\end{proof}

If $t_1=t_2=2$, then  the $18$ contact structures on $\Sigma\times S^1$ with a $0$-twisting vertical Legendrian circle are denoted using the decorations $(\pm)(\pm\pm)(\pm\pm)$. For $t_1\geq3$ and $t_2=2$, we use the decorations $(\pm)((\pm)(\pm))(\pm\pm)$ to represent the $24$ contact structures on $\Sigma\times S^1$ with a $0$-twisting vertical Legendrian circle. When $t_1\geq3$ and $t_2\geq3$,  we use the decorations $(\pm)((\pm)(\pm))((\pm)(\pm))$ to signify the $32$ contact structures on $\Sigma\times S^1$ with a $0$-twisting vertical Legendrian circle. See Figure~\ref{Figure:Sigma2} for an example.

\begin{figure}[htb]
\begin{overpic}
{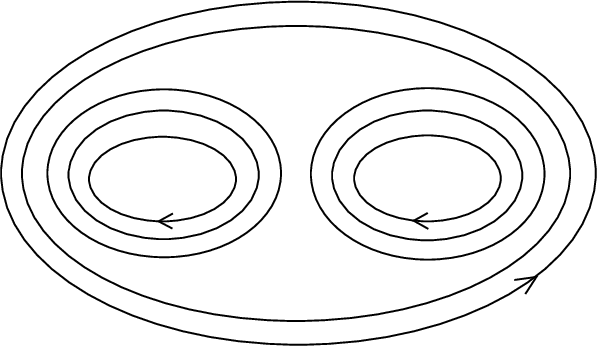}
\put(1, 80){$+$}
\put(23, 80){$-$}
\put(33, 80){$-$}
\put(150, 80){$-$}
\put(160, 80){$-$}
\put(110, 145){$\infty$}
\put(110, 156){$t_0$}
\put(80, 128){$\infty$}
\put(80, 115){$0$}
\put(68, 104){$-1/3$}
\put(200, 128){$\infty$}
\put(200, 115){$0$}
\put(195, 104){$-1/3$}
\end{overpic}

\caption{A pair of pants $\Sigma$, where $t_{0}=0$, $t_{1}=t_{2}=3$.}
\label{Figure:Sigma2}
\end{figure}

\subsection{$t_{1}<0$ and $t_{2}>0$.}

\begin{lemma}\label{Lemma:t_1<0t_2=1}
Suppose $t_1<0$ and $t_2=1$, then there are at most
\begin{align*}
\begin{split}
\left\{  
\begin{array}{ll}
4-4t_1, & ~\text{if}~ t_0\geq4,\\
4-3t_1, & ~\text{if}~ t_0=3,\\
4-2t_1, & ~\text{if}~ t_0=2,\\
t_0t_1-2t_1, & ~\text{if}~ t_0\leq1,
\end{array}
\right.
\end{split}
\end{align*}
appropriate tight contact structures on $\Sigma\times S^1$ with the given boundary slopes. 
\end{lemma}

\begin{proof}
The boundary slopes of $\Sigma\times S^1$ are $s_0=t_0$, $s_{1}=-\frac{1}{t_1}>0$ and $s_{2}=-1$. 

If $t_{0}\geq2$, then the tight contact structures on $\Sigma\times S^1$ always contain a $0$-twisting vertical Legendrian circles.

If $t_0\geq4$, $t_{1}<0$ and $t_{2}=1$, then there are at most $2\times (1-t_1)\times 2=4(1-t_1)$ tight contact structures. They depend on the signs of the basic slices in $L'_{i}$ for $i=0,1,2$. For the other cases, the upper bound can be obtained by deleting the duplication or the overtwisted contact structures. 

Suppose $t_0\leq1$. By Lemma~\ref{Lemma:tor>0}, there are no  appropriate tight contact structures with a $0$-twisting vertical Legendrian circle.  We consider the appropriate tight contact structures without a $0$-twisting vertical Legendrian circle. By Lemma~\ref{Lemma:slopes}, we can factorize $\Sigma\times S^1=\Sigma''\times S^1\cup L''_{1}$, where the boundary slopes of $\Sigma''\times S^1$ are $t_0$, $0$ and $1$, and the boundary slopes of $L''_{1}$ are $0$ and $-\frac{1}{t_1}$. Since $t_0\leq1$, according to \cite[Lemma 5.1]{h2},  there are exactly $2-t_0$ tight contact structures on $\Sigma''\times S^1$ without a $0$-twisting vertical Legendrian circle. There are $-t_1$ minimally twisting tight contact structures on $L''_{1}$. Therefore, there are at most $-2t_1+t_0t_1$ tight contact structures on $\Sigma\times S^1$ without a $0$-twisting vertical Legendrian circle and with boundary slopes $s_0=t_0$, $s_i=-\frac{1}{t_i}$ for $i=1,2$.
\end{proof}

We use the decorations $(\pm)(\underbrace{\pm\cdots\pm}_{-t_{1}})(\pm)$ to denote the $4-4t_1$ contact structures on $\Sigma\times S^1$ with a $0$-twisting vertical Legendrian circle. 

\begin{lemma}\label{Lemma:t_1<0t_2>1}
Suppose $t_1<0$ and $t_2>1$, then there are at most
\begin{align*}
\begin{split}
\left\{  
\begin{array}{ll}
6-6t_1, & ~\text{if}~ t_0\geq3, t_2=2, \\
6-4t_1, & ~\text{if}~ t_0=2, t_2=2, \\
6-2t_1, & ~\text{if}~ t_0=1, t_2=2, \\
8-8t_1, & ~\text{if}~ t_0\geq3, t_2\geq3, \\
8-6t_1, & ~\text{if}~ t_0=2, t_2\geq3, \\
8-4t_1, & ~\text{if}~ t_0=1, t_2\geq4, \\
8-3t_1, & ~\text{if}~ t_0=1, t_2=3, \\
2t_0t_1-2t_1, & ~\text{if}~ t_0\leq0, t_2\geq3, 
\end{array}
\right.
\end{split}
\end{align*}
appropriate tight contact structures on $\Sigma\times S^1$ with the given boundary slopes. 

\end{lemma}

\begin{proof}
The boundary slopes of $\Sigma\times S^1$ are $s_0=t_0$, $s_{1}=-\frac{1}{t_1}>0$ and $s_{2}=-\frac{1}{t_2}\in(-1,0)$. 

If $t_{0}\geq1$, then the tight contact structures on $\Sigma\times S^1$  always contain a $0$-twisting vertical Legendrian circles.

If $t_0\geq3$, $t_{1}<0$ and $t_{2}=2$, then there are at most $2\times (1-t_1)\times 3=6(1-t_1)$ appropriate tight contact structures. If $t_0\geq3$, $t_{1}<0$ and $t_{2}\geq3$, then there are at most $2\times (1-t_1)\times 4=8(1-t_1)$  appropriate  tight contact structures. The number of such contact structures depends on the signs of the basic slices in $L'_{i}$ for $i=0,1,2$. For the other cases, the upper bound can be obtained by deleting the duplication or the overtwisted contact structures. 

Suppose $t_0\leq0$. By Lemma~\ref{Lemma:tor>0}, there are no  appropriate tight contact structures with a $0$-twisting vertical Legendrian circle. We consider the appropriate tight contact structures without a $0$-twisting vertical Legendrian circle. By Lemma~\ref{Lemma:slopes}, we can factorize  $\Sigma\times S^1=\Sigma''\times S^1\cup L''_{1}\cup L''_{2}$, where the boundary slopes of $\Sigma''\times S^1$ are $t_0$, $1$ and $0$, the boundary slopes of $L''_{1}$ are $1$ and $-\frac{1}{t_1}$, and the boundary slopes of $L''_{2}$ are $0$ and $-\frac{1}{t_2}$. Since $t_0\leq0$, according to \cite[Lemma 5.1]{h2}, there are exactly $1-t_0$ tight contact structures on $\Sigma''\times S^1$ without a $0$-twisting vertical Legendrian circle. There are $-t_1$ minimally twisting tight contact structures on $L''_{1}$. There are $2$ minimally twisting tight contact structures on $L''_{2}$. Therefore, there are at most $-2t_1+2t_0t_1$ appropriate tight contact structures on $\Sigma\times S^1$ without a $0$-twisting vertical Legendrian circle and with boundary slopes $s_0=t_0$, $s_i=-\frac{1}{t_i}$ for $i=1,2$.

When $t_{2}=2$, the $6-6t_1$ contact structures on $\Sigma\times S^1$ with a $0$-twisting vertical Legendrian circle are denoted using the decorations $(\pm)(\underbrace{\pm\cdots\pm}_{-t_{1}})(\pm\pm)$. For $t_{2}\geq3$,  we use the decorations $(\pm)(\underbrace{\pm\cdots\pm}_{-t_{1}})((\pm)(\pm))$ to represent the $8-8t_1$ contact structures on $\Sigma\times S^1$  with a $0$-twisting vertical Legendrian circle. 

If $t_0=1,$ $t_1<0$ and $t_2=3$, we claim the two decorations $$(+)(\underbrace{+ \cdots +}_{l}\underbrace{- \cdots -}_{k})((-)(+)) ~\text{and}~ (-)(\underbrace{- \cdots -}_{k+1}\underbrace{+ \cdots +}_{l-1})((+)(-)),$$ where $l\geq1, k\geq0,  k+l=-t_1$, denote the same contact structure on $\Sigma\times S^1$. We consider $L'_0\cup B'_1\cup L'_2 \cup \Sigma'\times S^1$ in $\Sigma\times S^1$ with the first decoration, where $B'_1$ is the inner most basic slice in $L'_1$ with two boundary slopes $\infty$ and $1$. We can assume the sign of $B'_1$ is positive since $L'_1$ is a continued fraction block containing at least one positive basic slice. As before, there is a convex vertical annulus $A$ such that $\partial A$ consists of a Legendrian ruling on $T_0$ and a Legendrian ruling on the boundary component of $B'_1$ with slope $1$, and the dividing set on $A$ run from one boundary component to the other. If we cut $L'_0\cup B'_1\cup L'_2 \cup \Sigma'\times S^1$ along $A$, we will obtain a thickened torus admitting a factorization into two basic slices with slopes $-\frac{1}{3}$, $0$ and $0$, $-1$, and opposite signs. Here the slope $-1$ is obtained by $-1-1+1$. Using the same reasoning as in the proof of Lemma~\ref{Lemma:t_1>1t_2=1}, we have an exchange from the first decoration to the second without altering the isotopy classes of contact structures.
\end{proof}

\subsection{$t_{1}=0$.}
\begin{lemma}\label{Lemma:t1eq0}
Suppose $t_1=0$; then there are at most
\begin{align*}
\begin{split}
\left\{  
\begin{array}{ll}
8, & ~\text{if}~ t_2\geq3, \\
6, & ~\text{if}~ t_2=2, \\
4, & ~\text{if}~ t_2=1, \\
2-2t_2, & ~\text{if}~ t_2\leq0,  
\end{array}
\right.
\end{split}
\end{align*}
appropriate tight contact structures on $\Sigma\times S^1$ with the given boundary slopes. All of them have $0$-twisting vertical Legendrian circles.
\end{lemma}

\begin{proof} Since $s_1=\infty$, the appropriate tight contact structures on $\Sigma\times S^1$ always contain $0$-twisting vertical Legendrian circles.

The boundary slopes of $\Sigma\times S^1$ are $t_0$, $\infty$ and $-\frac{1}{t_{2}}$. We can  factorize  $\Sigma\times S^1=L'_{0}\cup L'_{2}\cup \Sigma'\times S^1$, where the boundary slopes of $\Sigma'\times S^1$ are all $\infty$, the boundary slopes of $L'_{0}$ are $\infty$ and $t_0$, and the boundary slopes of $L'_{2}$ are $\infty$ and $-\frac{1}{t_2}$. There are exactly $2$ minimally twisting  tight contact structures on $L'_{0}$. If $t_{2}\leq0$, $=1$, $=2$ or $\geq3$,  then there are $1-t_2$, $2$, $3$ or $4$ minimally twisting tight contact structures on $L'_{2}$, respectively. Therefore, if $t_{2}\leq0$, $=1$, $=2$ or $\geq3$,  then there are $2-2t_2$, $4$, $6$ or $8$ appropriate tight contact structures on $\Sigma\times S^1$, respectively.
\end{proof}

If $t_{2}\geq3$, the $8$ contact structures on $\Sigma\times S^1$ are denoted using the decorations $(\pm)((\pm)(\pm))$. For $t_{2}=2$,  we use the decorations $(\pm)(\pm\pm)$ to represent the $6$ contact structures on $\Sigma\times S^1$. When $t_{2}=1$,  we use the decorations $(\pm)(\pm)$ to denote the $4$ contact structures on $\Sigma\times S^1$. If $t_{2}\leq0$,  we use the decorations $(\pm)(\underbrace{\pm\cdots\pm}_{-t_{2}})$ to denote the $2-2t_2$ contact structures on $\Sigma\times S^1$.

\subsection{Some tight contact structures}
We use the notation $(T^2\times [0,1], s_0, s_1)$ to represent a basic slice with boundary slopes $s_0$ and $s_1$ on $T^2\times\{i\}$, where $i=0,1$. There is a geodesic in the Farey graph connecting $s_0$ and $s_1$. Moreover, any boundary parallel convex torus of this slice has a dividing slope within the range of $[s_0, s_1]$ corresponding to the clockwise arc on the boundary of the Poincare disk shown in Figure~\ref{Figure:farey}.

\begin{lemma}\label{Lemma:tight}
There are $6$ tight contact structures on $\Sigma\times S^1$ with boundary slopes $t_0$, $-\frac{1}{t_1}$ and $-\frac{1}{t_2}$, where $t_1, t_2\neq0$, and satisfying that 
\begin{itemize}
\item $\Sigma\times S^1$ can be decomposed as $L'_{0}\cup L'_{1}\cup L'_{2}\cup \Sigma'\times S^1$, where $\Sigma'\times S^1$ have boundary slopes $\infty$,
\item $L'_{0}$ is a basic slice, 
\item $L'_{i}$, $i=1,2$, is a thickened torus, all of whose basic slices have the same signs,
\item the signs of $L'_0$, $L'_1$ and $L'_2$ are $\pm\mp\mp$, $\pm\mp\pm$ or $\pm\pm\mp$.
\end{itemize}
\end{lemma}

\begin{proof}
Suppose they have $0$-twisting vertical Legendrian circles.  By Lemma~\ref{Lemma:slopes1}, each of them can be decomposed as $L'_{0}\cup L'_{1}\cup L'_{2}\cup \Sigma'\times S^1$, where the boundary slopes of $\Sigma'\times S^1$ are all $\infty$, $L'_{0}$ is a basic slice $(T^{2}\times [0,1]; \infty, t_0)$, and the innermost basic slice $B'_i$ of $L'_{i}$ is $(T^{2}\times [0,1]; \infty, \lceil-\frac{1}{t_i}\rceil)$ for $i=1,2$.
Using Part 2 of \cite[Lemma 5.1]{h2}, we know that there are $6$  universally tight contact structures on $L'_{0}\cup B'_{1}\cup B'_{2}\cup \Sigma'\times S^1$ which are determined by the signs of $L'_{0}$, $B'_{1}$ and $B'_{2}$. Note that the three signs are not the same.  Each of them can be extended to a universally tight $\tilde{\Sigma}\times S^1$ whose boundary slopes are all $\infty$. The contact structure on $L'_{0}\cup L'_{1}\cup L'_{2}\cup \Sigma'\times S^1$ can be embedded into $\tilde{\Sigma}\times S^1$. Hence the given contact $\Sigma\times S^1$ is tight.
\end{proof}

\begin{lemma}\label{Lemma:tight1}
There are $4$ tight contact structures on $\Sigma\times S^1$ with boundary slopes $t_0$, $\infty$ and $-\frac{1}{t_2}$, where $t_2\neq0$, and satisfying that 
\begin{itemize}
\item $\Sigma\times S^1$ can be decomposed as $L'_{0}\cup L'_{2}\cup \Sigma'\times S^1$, where $\Sigma'\times S^1$ have boundary slopes $\infty$,
\item $L'_{0}$ is a basic slice, 
\item $L'_{2}$ is a thickened torus, all of whose basic slices have the same signs, 
\item the signs of $L'_0$ and $L'_2$ are $\pm\pm$ or $\pm\mp$.
\end{itemize}
\end{lemma}

\begin{proof}
Using \cite[Lemma 5.2]{h2}, the proof is similar to that of Lemma~\ref{Lemma:tight}.
\end{proof}





\section{Methods of construction of strongly exceptional Legendrian $A_3$ links}

In practice, contact surgery diagrams are a common tool for representing strongly exceptional Legendrian links. Several works, such as \cite{go1}, \cite{go}, \cite{go2}, \cite{m} and \cite{emm}  employ this technique. In this paper, we utilize contact surgery diagrams to construct strongly exceptional Legendrian $A_3$ links. It is worth noting that if an exceptional Legendrian $A_3$ link can be constructed by this technique, then it must be strongly exceptional. This is because conducting contact surgery along such a Legendrian $A_3$ link results in a tight contact 3-manifold, whereas a Giroux torsion domain in $\Sigma\times S^1$ gives rise to an overtwisted disk after the surgery.  Given a contact surgery diagram for an exceptional Legendrian $A_3$ link, the Thurston-Bennequin invariants and rotation numbers can be calculated using  \cite[Lemma 6.6]{loss}. Furthermore, the $d_3$-invariant of the ambient contact 3-sphere can be obtained according to \cite{DGS}.

Additionally, we introduce three other methods. The first method involves performing Legendrian  connected sums of two Legendrian knots. The concept of Legendrian connected sum of Legendrian knots was defined in \cite[Section 3]{eh}.

\begin{figure}[htb]
\begin{overpic}
{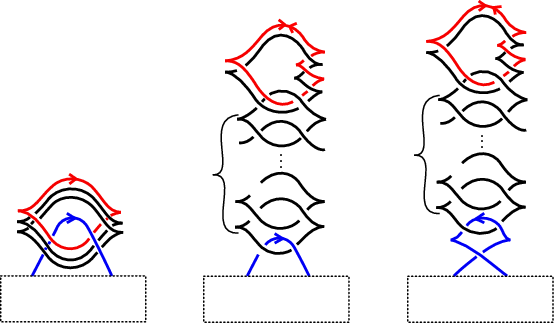}
\put(1, 7){$K'_{0}\# K''_{0}\cup K_2$}
\put(99, 7){$K'_{0}\# K''_{0}\cup K_2$}
\put(197, 7){$K'_{0}\# K''_{0}\cup K_2$}
\put(0, 62){$K_1$}
\put(100, 134){$K_1$}
\put(200, 144){$K_1$}
\put(59, 33){$+1$}
\put(59, 43){$+1$}
\put(154, 118){$+1$}
\put(154, 90){$-1$}
\put(154, 76){$-1$}
\put(154, 51){$-1$}
\put(154, 38){$-1$}
\put(251, 127){$+1$}
\put(251, 98){$-1$}
\put(251, 83){$-1$}
\put(251, 60){$-1$}
\put(251, 46){$-1$}
\end{overpic}
\caption{In the middle and right picture, for $t_1$ even, $K'_0$ and $K_1$  bear the same orientation, for $t_{1}$ odd, the opposite one.  }
\label{Figure:link40inot}
\end{figure}
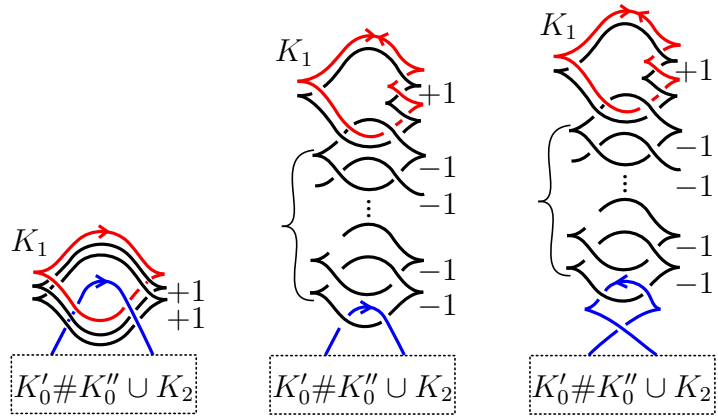

\begin{lemma}\label{Lemma:connected}
Let $K'_{0}\cup K_1$ be a strongly exceptional Legendrian Hopf link in a contact $(S^3, \xi_{\frac{1}{2}})$ with $(t'_{0}, r'_0)=(t_{1}, r_1)=(1, 0)$ or $(t'_{0}, r'_0)=(0, \pm1), t_{1}\geq2, r_{1}=\pm(t_{1}-1)$. Let $K''_{0}\cup K_2$ be a strongly exceptional Legendrian Hopf link in a contact $S^3$. Then the Legendrian connected sum $(K'_{0}\# K''_{0})\cup K_1\cup K_2$ is a strongly exceptional Legendrian $A_3$ link in a contact $S^3$.
\end{lemma}
\begin{proof}
Suppose $t'_{0}=0, t_{1}\geq1$. Let $t''_0$ be the Thurston-Bennequin invariant of $K''_0$. If the pair $(t''_0, t_2)$ is not $(2,1)$ or $(1,2)$, then any strongly exceptional Legendrian Hopf link $K''_0\cup K_2$ has a contact surgery diagram \cite{go}. As a result, $(K'_{0}\# K''_{0})\cup K_1\cup K_2$ has a contact surgery diagram as shown in the middle and right of Figure~\ref{Figure:link40inot}. We then perform contact $(-1)$-surgery along $K_1$ and cancel the contact $(+1)$-surgery along the Legendrian unknots. By ignoring the Legendrian unknots with contact $(-1)$-surgeries, we obtain a contact surgery diagram for the Legendrian link $K''_{0}\cup K_2$. As per \cite{go}, some contact surgeries along $K''_{0}\cup K_2$ will result in  closed tight contact 3-manifolds. Since contact $(-1)$-surgery on closed contact 3-manifold preserves tightness \cite{w}, some contact surgery along $(K'_{0}\# K''_{0})\cup K_1\cup K_2$ will yield a tight contact 3-manifold. Therefore, $(K'_{0}\# K''_{0})\cup K_1\cup K_2$ is strongly exceptional. 

In the case where $(t''_0, t_2)$ is either $(2,1)$ or $(1,2)$, \cite{go} tells us that its exterior is a universally tight thickened torus and can therefore be contact embedded into a tight contact $T^3$. The contact $(-1)$-surgery along links in a tight contact $T^3$ results in a tight 3-manifold. As such, the contact $(-1)$-surgery along links in the exterior of $K''_0\cup K_2$ will also yield a tight 3-manifold. Therefore, the contact $(-1)$-surgery along $K_1$ will result in a tight contact 3-manifold. This means that  $(K'_{0}\# K''_{0})\cup K_1\cup K_2$ is strongly exceptional. 

Assuming $t'_{0}=t_{1}=1$. If the pair $(t''_0, t_2)$ is not $(2,1)$ or $(1,2)$, then $(K'_{0}\# K''_{0})\cup K_1\cup K_2$ will have a contact surgery diagram as shown in the left of Figure~\ref{Figure:link40inot}.  We then perform contact $(-\frac{1}{2})$-surgery along $K_1$ and cancel the contact $(+1)$-surgery along the two Legendrian unknots. By doing so,  we obtain a contact surgery diagram for the strongly exceptional Legendrian link $K''_0\cup K_2$. This means that the exterior of  $(K'_{0}\# K''_{0})\cup K_1\cup K_2$ is appropriate tight.

If the pair $(t''_0, t_2)$ is either $(2,1)$ or $(1,2)$,  we can apply the same argument as in the previous case.
\end{proof}

We recall that the $d_3$-invariant of the contact connected sum of two contact 3-spheres $(S^3,\xi)$ and $(S^3, \xi')$ is given by $d_{3}(\xi)+d_{3}(\xi')+\frac{1}{2}$. Suppose $K''_0$ has Thurston-Bennequin invariant $t''_0$ and rotation number $r''_0$, then $K'_0\#K''_0$ has Thurston-Bennequin invariant $t'_0+t''_0+1$ and rotation number $r'_0+r''_0$.

The second method involves adding local Legendrian meridians. In a contact 3-sphere, consider a Legendrian knot intersecting a Darboux ball in a simple arc. A Legendrian unknot within the Darboux ball, which serves as a meridian of the given Legendrian knot, is called a local Legendrian meridian. The following lemma is straightforward.

\begin{lemma}\label{Lemma:meridian}
Suppose $K_{0}\cup K_{2}$ is a strongly exceptional Legendrian Hopf link. Let $K_1$ be a local Legendrian meridian of $K_0$. Then $K_{0}\cup K_{1}\cup K_{2}$ is an  strongly exceptional Legendrian $A_3$ link with $t_{1}<0$ and $r_{1}\in\{t_{1}+1, t_{1}+3, \cdots, -t_{1}-1\}$.
\end{lemma}

The third method involves extending an (appropriate) tight contact $\Sigma\times S^1$ admitting $0$-twisting vertical Legendrian circle to an overtwisted contact 
$S^3$.

Suppose an (appropriate) tight contact structure $\xi$ on $\Sigma\times S^{1}$ has a $0$-twisting vertical Legendrian circle $\gamma$. We attach three contact solid tori $D^{2}_{i}\times S^{1}$, $i=0,1,2$, to $(\Sigma\times S^{1}, \xi)$ such that $\partial D^{2}_{0}$ is identified to $h$, $\partial D^{2}_{1}$ is identified to $c_1$, and $\partial D^{2}_{2}$ is identified to $c_2$. Then the resulting manifold $\Sigma\times S^{1}\cup D^{2}_{0}\times S^{1}\cup D^{2}_{1}\times S^{1}\cup D^{2}_{2}\times S^{1}$ is diffeomorphic to $S^3$.


If the contact structure on $D^{2}_{i}\times S^{1}$ has a minimal convex boundary with slope given by a longitude (i.e., the dividing set of the convex boundary intersects the meridional circle in exactly two points), then it admits a unique tight contact structure. Additionally, the core of such contact solid torus is Legendrian.

Since the dividing set of $T_i$ intersects the meridional disk of $D^{2}_{i}\times S^{1}$ in exactly two points, the contact structure $\xi$ on $\Sigma\times S^1$ uniquely extends to a contact structure on $S^3$. However, Since $\partial D^{2}_{0}$ is identified to $h$, the Legendrian vertical circle $\gamma$ bounds an overtwisted disk in $S^3$. Therefore, the resulting contact structure on $S^3$ is overtwisted. 

\begin{lemma}\label{Lemma:construction}
Let $\xi$ be an (appropriate) tight contact structure on $\Sigma\times S^{1}$ admits a $0$-twisting vertical Legendrian circle. Extending $\xi$ to a contact 3-sphere as above by adding three tight contact solid tori. Let $K_i$, $i=0,1,2$, be the core of three attached contact solid tori. Then $K_{0}\cup K_{1}\cup K_2$ is a (strongly) exceptional Legendrian  $A_3$ link in an overtwisted contact 3-sphere.
\end{lemma}

Moreover, we have the following observations.

\begin{lemma}\label{Lemma:appro}
Let $\xi_1$ and $\xi_2$ be two tight contact structures on $\Sigma\times S^1$ with $0$-twisting vertical Legendrian circles. Suppose they both have minimal convex boundaries with slopes $t_0$, $-\frac{1}{t_1}$ and $-\frac{1}{t_2}$. Suppose their factorizations $L'_{0}\cup L'_{1}\cup L'_{2}\cup \Sigma'\times S^1$ (or $L'_{0}\cup L'_{2}\cup \Sigma'\times S^1$ when $t_1=0$) differ only in the signs of basic slices in $L'_{0}\cup L'_{1}\cup L'_{2}$ (or $L'_{0}\cup L'_{2}$ when $t_1=0$). If $\xi_1$ is appropriate tight, then so is $\xi_2$. 
\end{lemma}

\begin{proof}
This is because the computation of Giroux torsion of an embedded torus $T$ in a contact 3-manifold only depends on the slopes of the convex tori parallel to $T$. 
\end{proof}

\begin{lemma}\label{Lemma:stabilization}
Suppose $\mathcal{L}$ is an exceptional Legendrian $A_3$ link whose exterior contains a $0$-twisting Legendrian vertical circle. Then the components $K_0$ and $K_{i}$ with $t_i\neq0$, where $i=1,2$, of $\mathcal{L}$ can always be destabilized.
\end{lemma}

\begin{proof}
There is a basic slice $L'_0$ in the exterior of $\mathcal{L}$ which is $(T^{2}\times [0,1], \infty, t_0)$. We can find a basic slice $(T^{2}\times [0,1], t_0+1, t_0)$ in $L'_0$. So the component $K_0$ can be destabilized.  For $i=1,2$, since there is a basic slice $(T^{2}\times [0,1], -\frac{1}{t_i +1}, -\frac{1}{t_i})$  in the thickened torus $L'_i$, the component $K_i$ can be destabilized. 
\end{proof}

\section{Realizations of strongly exceptional Legendrian $A_3$ links}\label{Section:realizations}

In this section, we construct strongly exceptional Legendrian $A_3$ links. 

Throughout this paper, in the contact surgery diagrams representing a Legendrian $A_3$ link, if a component is a Legendrian push-off of some $K_i$, $i=0, 1, 2$, then its contact surgery coefficient is $+1$, otherwise its contact surgery coefficient is $-1$.

\subsection{$t_{1}<0$ and $t_{2}<0$.}

The boundary slopes of $\Sigma\times S^1$ are $s_0=t_0$, $s_1=-\frac{1}{t_1}\in(0, 1]$ and $s_2=-\frac{1}{t_2}\in(0, 1]$. 

\begin{lemma}\label{Lemma:relEuler1}
For any $t_0\in\mathbb{Z}$, there are $6$ exceptional Legendrian $A_3$ links whose exteriors have $0$-twisting vertical Legendrian circles, and have decorations $\pm(+)(\underbrace{- \cdots -}_{-t_{1}})(\underbrace{- \cdots -}_{-t_{2}})$, $\pm(+)(\underbrace{- \cdots -}_{-t_{1}})(\underbrace{+ \cdots +}_{-t_{2}})$ and $\pm(+)(\underbrace{+ \cdots +}_{-t_{1}})(\underbrace{- \cdots -}_{-t_{2}})$. Their rotation numbers are $$r_{0}=\pm(t_0-1), r_1=\pm(1-t_{1}), r_{2}=\pm(1-t_{2}); r_{0}=\pm(t_0-1), r_1=\pm(1-t_{1}), r_{2}=\pm(t_{2}+1); $$
$$r_{0}=\pm(t_0-1), r_1=\pm(t_{1}+1), r_{2}=\pm(1-t_{2}).$$
The corresponding $d_3$-invariants are independent of $t_0$ if $t_1$ and $t_2$ are fixed.
\end{lemma}

\begin{proof}
The first statement follows from Lemma~\ref{Lemma:tight} and Lemma~\ref{Lemma:construction}. The rotation number of a Legendrian knot in a contact 3-sphere is the evaluation of the relative Euler class on a Seifert surface of the knot. We compute the rotation numbers in a similar way as that in \cite[Section 2.5]{emm}. The Seifert surface of $K_0$ can be obtained by capping the pair of pants $\Sigma$ by two disks along the boundary components $c_1$ and $c_2$. The Seifert surface of $K_i$, $i=1,2$, is a union of a meridian disk of $K_0$ and an annulus.  For instance, if the signs of $L'_0$, $L'_1$ and $L'_2$ are $+--$, see Figure~\ref{Figure:Sigma1} for an example, then the rotation numbers can be computed using relative Euler class as follows. We denote $\frac{a}{b}\ominus\frac{c}{d}$ to be $\frac{a-c}{b-d}$, and $\frac{a}{b}\bullet\frac{c}{d}$ to be $ad-bc$ \cite[Section 2.5]{emm}. The denominators are assumed non-negative. The rotation number of $K_0$ is 
\begin{align*}
r_{0}&=-(\frac{-1}{-t_1}\ominus\frac{-1}{-t_{1}-1})\bullet\frac{0}{1}-(\frac{-1}{-t_{1}-1}\ominus\frac{-1}{-t_{1}-2})\bullet\frac{0}{1}-\cdots-(\frac{-1}{1}\ominus\frac{-1}{0})\bullet\frac{0}{1}\\
&-(\frac{-1}{-t_2}\ominus\frac{-1}{-t_{2}-1})\bullet\frac{0}{1}-(\frac{-1}{-t_{2}-1}\ominus\frac{-1}{-t_{2}-2})\bullet\frac{0}{1}-\cdots-(\frac{-1}{1}\ominus\frac{-1}{0})\bullet\frac{0}{1}\\
&+(\frac{1}{0}\ominus\frac{t_0}{1})\bullet\frac{0}{1}=1-t_{0}.
\end{align*}
The rotation number of $K_1$ is \begin{align*}
r_1=(\frac{-t_0}{1}\ominus\frac{-1}{0})\bullet\frac{1}{0}-(\frac{1}{0}\ominus\frac{1}{1})\bullet\frac{1}{0}-(\frac{1}{1}\ominus\frac{1}{2})\bullet\frac{1}{0}-\cdots-(\frac{1}{-t_1-1}\ominus\frac{1}{-t_{1}})\bullet\frac{1}{0}=t_{1}-1.
\end{align*}
The rotation number of $K_2$ is \begin{align*}
r_2=(\frac{-t_0}{1}\ominus\frac{-1}{0})\bullet\frac{1}{0}-(\frac{1}{0}\ominus\frac{1}{1})\bullet\frac{1}{0}-(\frac{1}{1}\ominus\frac{1}{2})\bullet\frac{1}{0}-\cdots-(\frac{1}{-t_{2}-1}\ominus\frac{1}{-t_{2}})\bullet\frac{1}{0}=t_{2}-1.
\end{align*}
In the computation above, when calculating $r_0$, it is necessary to reverse the signs of the dividing slopes in the thickened tori $L'_1$ and $L'_2$. Similarly, when calculating $r_1$ and $r_2$, the signs of the dividing slopes in the thickened torus $L'_0$ should be reversed.

The last statement follows directly from Lemma~\ref{Lemma:stabilization}.
\end{proof}

In a similar way, we can use relative Euler classes and the given decorations to compute the rotation numbers of any other Legendrian $A_3$ links whose exteriors contain a $0$-twisting vertical Legendrian circle.



\begin{proof}[Proof of Theorem~\ref{Theorem:t_1<0t_2<0}] Recall that the numbers of strongly exceptional Legendrian $A_3$ links have upper bounds listed in Lemma~\ref{Lemma:t1<0t2<0}. We will show that these upper bounds can be attained.

\begin{lemma}\label{Lemma:topological}
The oriented link $K_0\cup K_1\cup K_2$ in the surgery diagram in Figure~\ref{Figure:link51} is a topological  $A_3$ link in $S^3$.
\end{lemma}
\begin{figure}[htb]
\begin{overpic}
{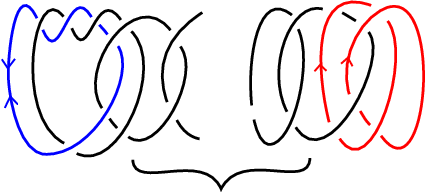}
\put(185, 10){$K_2$}
\put(162, 10){$K_1$}
\put(-13, 48){$K_0$}
\put(100, 50){$\cdots$}
\put(105, -8){$n$}
\put(40, 10){$-1$}
\put(63, 90){$-2$}
\put(90, 90){$-2$}
\put(122, 92){$-2$}
\put(142, 92){$-2$}
\end{overpic}
\caption{ For $n$ even, $K_0$ and $K_{i}$, $i=1,2$, bear the same orientation, for $n$ odd, the opposite one.  }
\label{Figure:link51}
\end{figure}
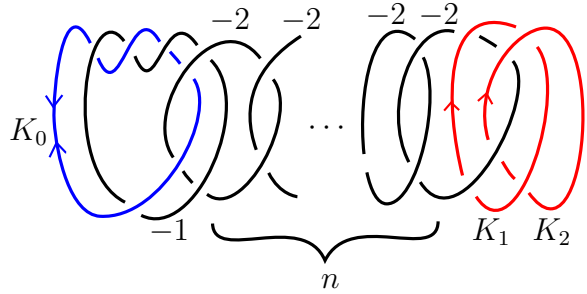
\begin{proof}
The proof is similar to that of \cite[Lemma 5.1, part (i)]{go}.
\end{proof}

(1) Suppose $t_0\geq2$. 

\begin{lemma}\label{t0>1t1<0t2<0}
If $t_0\geq2, t_1<0, t_2<0$, there exist $2t_{1}t_{2}-2t_{1}-2t_{2}+2$ strongly exceptional Legendrian $A_3$ links in $(S^3, \xi_{\frac{1}{2}})$ whose rotation numbers are $$r_{0}=\pm(t_{0}-1), r_{i}\in\pm\{t_{i}+1, t_{i}+3,\cdots, -t_{i}+1\}, i=1,2.$$ 
\end{lemma} 

\begin{proof}
There are $2t_{1}t_{2}-2t_{1}-2t_{2}+2$ strongly exceptional Legendrian $A_3$ links as illustrated in Figure~\ref{Figure:link10inot}. According to Lemma~\ref{Lemma:topological},  $K_0\cup K_1\cup K_2$ forms a topological $A_3$ link. By performing the same calculations as in the proof of Theorem 1.2 (b1) in \cite{go}, we can determine that their rotation numbers are as listed.  The corresponding $d_3$-invariant is $\frac{1}{2}$. The strong exceptionality property arises from carrying out contact $(-1)$-surgery along $K_0$ which cancels the contact $(+1)$-surgery.
\end{proof}  


\begin{figure}[htb]
\begin{overpic}
{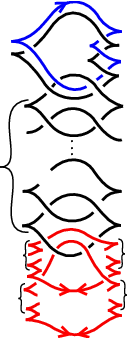}
\put(30, 6){$K_2$}
\put(30, 26){$K_1$}
\put(0, 148){$K_0$}
\put(-32, 78){$t_{0}-2$}
\put(60, 137){$+1$}
\put(60, 107){$-1$}
\put(60, 91){$-1$}
\put(60, 65){$-1$}
\put(60, 50){$-1$}
\put(0, 35){$k_1$}
\put(62, 35){$l_1$}
\put(-3, 15){$k_2$}
\put(62, 15){$l_2$}
\end{overpic}
\caption{ $t_0\geq2$, $t_1\leq0$, $t_2\leq0$. For $i=1,2$, $k_{i}+l_{i}=-t_{i}$.  For $t_0$ even, $K_0$ and $K_{i}$, $i=1,2$, bear the same orientation, for $t_{0}$ odd, the opposite one.  }
\label{Figure:link10inot}
\end{figure}
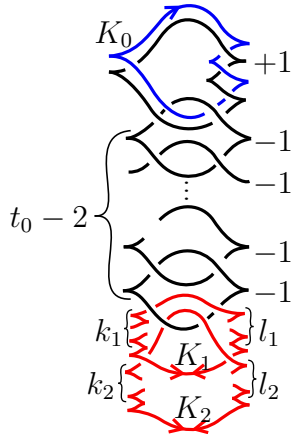

(2) Suppose $t_0=1$. 

\begin{lemma}\label{t0=1t1<0t2<0}
If $t_0=1, t_1<0, t_2<0$, then there exist $t_{1}t_{2}-2t_{1}-2t_{2}+2$ strongly exceptional Legendrian $A_3$ links  in $(S^3, \xi_{\frac{1}{2}})$ whose rotation numbers are $$r_{0}=0, r_{i}\in\{t_{i}+1, t_{i}+3,\cdots, -t_{i}+1\}, i=1,2;$$
$$r_{0}=0, r_{1}=t_{1}-1, r_{2}\in\{t_{2}-1, t_{2}+1, \cdots,  -t_{2}-1\};$$ 
$$r_{0}=0,  r_{1}\in\{t_{1}+1, t_{1}+3, \cdots,  -t_{1}-1\}, r_{2}=t_{2}-1. $$
\end{lemma}

\begin{proof}
There are $t_{1}t_{2}-2t_{1}-2t_{2}+2$ strongly exceptional Legendrian $A_3$ links as shown in Figure~\ref{Figure:link11inot}. The linking number of the components $K_1$ and $K_2$ in Figure~\ref{Figure:link11inot} is $-2$. Using similar Kirby diagrams as in \cite[Lemma 5.1, part (iii), Figure 3]{go}, we can show that it is a topological $A_3$ link. By performing the same calculations as in the proof of Theorem 1.2 (b2) in \cite{go}, we can determine that their rotation numbers are as listed. In the left diagram of Figure~\ref{Figure:link11inot}, $$r_0=0, r_{i}\in\{ t_i+1, t_i+3, \cdots, -t_{i}+1\} ~\text{for}~ i=1,2,$$ while in the right diagram,  $$r_0=0, r_{i}\in\{t_i-1, t_i+1, \cdots, -t_{i}-1\}~\text {for}~ i=1,2.$$ There are exactly $t_{1}t_{2}$ Legendrian $A_3$ links represented by both the left and the right diagrams. Moreover, 
the corresponding $d_3$-invariant is $\frac{1}{2}$. 
\end{proof}

\begin{figure}[htb]
\begin{overpic}
{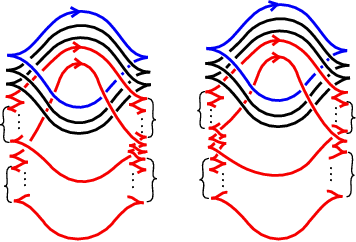}
\put(30, 6){$K_2$}
\put(30, 35){$K_1$}
\put(55, 110){$K_0$}
\put(127, 7){$K_2$}
\put(127, 36){$K_1$}
\put(154, 110){$K_0$}
\put(-12, 82){$+1$}
\put(-12, 72){$+1$}
\put(170, 82){$+1$}
\put(170, 72){$+1$}
\put(-10, 55){$k_1$}
\put(75, 55){$l_1$}
\put(-10, 25){$k_{2}$}
\put(76, 28){$l_{2}$}
\put(172, 55){$l_1$}
\put(86, 60){$k_1$}
\put(86, 28){$k_{2}$}
\put(172, 30){$l_{2}$}
\end{overpic}
\caption{$t_0=1$, $t_1<0$, $t_2<0$. For $i=1,2$, $k_{i}+l_{i}=-t_{i}+1$. In the left diagram, $l_1, l_2\geq1$, while in the right diagram, $k_1, k_2\geq1$. }
\label{Figure:link11inot}
\end{figure}

(3) Suppose $t_0=0$. 

\begin{lemma}\label{t0=0t1<0t2<0}
If $t_0=0, t_1<0, t_2<0$,  then there exist $-2t_1-2t_2+2$ strongly exceptional Legendrian $A_3$ links in $(S^3, \xi_{\frac{1}{2}})$ whose rotation numbers are $$r_0=\pm 1, r_{1}=\pm(t_{1}-1), r_{2}\in\{t_{2}+1, t_{2}+3, \cdots, -t_{2}-1\};$$
$$r_0=\pm 1, r_{1}\in\{t_{1}+1, t_{1}+3, \cdots, -t_{1}-1\}, r_{2}=\pm(t_{2}-1);$$ 
$$r_0=\pm 1, r_{1}=\pm(t_{1}-1), r_{2}=\pm(t_{2}-1).$$
\end{lemma}

\begin{proof}
By \cite[Theorem 1.2]{go}, there are two strongly exceptional Legendrian Hopf link $K_0\cup K_1$ in $(S^3, \xi_{\frac{1}{2}})$ with $(t_0, r_0)=(0, \pm1)$, $t_1<0$ and $r_{1}=\pm(t_{1}-1)$. Let $K_2$ be a local Legendrian meridian of $K_0$. Then by Lemma~\ref{Lemma:meridian} there are $-2t_2$ strongly exceptional Legendrian $A_3$ links in $(S^3, \xi_{\frac{1}{2}})$ whose rotation numbers are $$r_0=\pm 1, r_{1}=\pm(t_{1}-1), r_{2}\in\{t_{2}+1, t_{2}+3, \cdots, -t_{2}-1\}.$$ 

Similarly, there are $-2t_1$ strongly exceptional Legendrian $A_3$ links in $(S^3, \xi_{\frac{1}{2}})$ whose rotation numbers are $$r_0=\pm 1, r_{1}\in\{t_{1}+1, t_{1}+3, \cdots, -t_{1}-1\}, r_{2}=\pm(t_{2}-1).$$

Moreover, by Lemma~\ref{Lemma:relEuler1} and Lemma~\ref{Lemma:appro}, there are $2$ strongly exceptional Legendrian $A_3$ links in $(S^3, \xi_{\frac{1}{2}})$ whose rotation numbers $(r_0, r_1, r_2)$ are
$(\pm 1, \pm(t_{1}-1), \pm(t_{2}-1)).$
\end{proof}

\begin{figure}[htb]
\begin{overpic}
{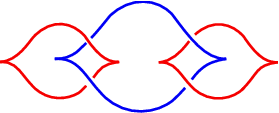}
\put(111, 45){$K_2$}
\put(10, 45){$K_1$}
\put(60, 58){$K_0$}
\end{overpic}
\caption{A Legendrian $A_3$ link  in $(S^3, \xi_{st})$. }
\label{Figure:linkinstd}
\end{figure}

(4) Suppose $t_{0}<0$. 

\begin{lemma}\label{t0<0t1<0t2<0}
If $t_{i}<0$ for $i=0,1,2$, then there exist $-t_{0}t_{1}t_{2}$ strongly exceptional Legendrian $A_3$ links in $(S^3, \xi_{st})$ whose rotation numbers are $$r_{i}\in\{t_{i}+1, t_{i}+3, \cdots, -t_{i}-1\}, ~\text{for}~ i=0,1,2.$$ 
\end{lemma}

\begin{proof}
By stabilizations of the Legendrian $A_3$ link shown in Figure~\ref{Figure:linkinstd}, we obtain $-t_{0}t_{1}t_{2}$ strongly exceptional Legendrian $A_3$ links in $(S^3, \xi_{st})$. Their rotation numbers are as listed.
\end{proof}

So there are exactly $-t_{0}t_{1}t_{2}$ Legendrian $A_3$ links in  contact 3-spheres whose complements are appropriate tight if $t_{i}<0$ for $i=0,1,2$.  

The proof of Theorem~\ref{Theorem:t_1<0t_2<0} is completed.
\end{proof}

\subsection{$t_{1}>0$ and $t_{2}>0$.}

The boundary slopes of $\Sigma \times S^1$ are $s_0=t_0$, $s_1=-\frac{1}{t_1}\in[-1,0)$ and $s_2=-\frac{1}{t_2}\in[-1,0)$.

\begin{lemma}\label{Lemma:3.2.1} 
For any $t_0\in\mathbb{Z}$, there are $6$ exceptional Legendrian $A_3$ links whose exteriors have $0$-twisting vertical Legendrian circles, and the signs of basic slices in $L'_0, L'_1, L'_2$ are $\pm(+--), \pm(++-)$ and $\pm(+-+)$, respectively.
Their rotation numbers are 
$$r_{0}=\pm(t_{0}+3), r_{1}=\pm(t_{1}+1), r_{2}=\pm(t_{2}+1); r_{0}=\pm(t_{0}-1), r_{1}=\pm(1-t_{1}), r_{2}=\pm(t_{2}+1);$$
$$r_{0}=\pm(t_{0}-1),  r_{1}=\pm(t_{1}+1), r_{2}=\pm(1-t_{2}).$$
The corresponding $d_3$-invariants are independent of $t_0$ if $t_1$ and $t_2$ are fixed.
\end{lemma}

\begin{proof}
The first statement can be inferred from Lemma~\ref{Lemma:tight} and Lemma~\ref{Lemma:construction}. For example, when the signs of $L'_0$, $L'_1$ and $L'_2$ are $+--$,  the rotation numbers can be computed using the relative Euler class as follows. See Figure~\ref{Figure:Sigma2} for the decoration. The rotation number of $K_0$ is 
$$r_{0}=-(\frac{1}{t_1}\ominus\frac{0}{1})\bullet\frac{0}{1}-(\frac{0}{1}\ominus\frac{-1}{0})\bullet\frac{0}{1}
-(\frac{1}{t_2}\ominus\frac{0}{1})\bullet\frac{0}{1}-(\frac{0}{1}\ominus\frac{-1}{0})\bullet\frac{0}{1}
+(\frac{1}{0}\ominus\frac{t_0}{1})\bullet\frac{0}{1}=-t_{0}-3.$$
The rotation number of $K_1$ is 
$$r_1=(\frac{-t_0}{1}\ominus\frac{-1}{0})\bullet\frac{1}{0}-(\frac{1}{0}\ominus\frac{0}{1})\bullet\frac{1}{0}-(\frac{0}{1}\ominus\frac{-1}{t_1})\bullet\frac{1}{0}=-t_{1}-1.$$
The rotation number of $K_2$ is $$
r_2=(\frac{-t_0}{1}\ominus\frac{-1}{0})\bullet\frac{1}{0}-(\frac{1}{0}\ominus\frac{0}{1})\bullet\frac{1}{0}-(\frac{0}{1}\ominus\frac{-1}{t_2})\bullet\frac{1}{0}=-t_{2}-1.
$$   
\end{proof}

\subsubsection{$t_{1}=t_{2}=1$.}
\begin{proof}[Proof of Theorem~\ref{Theorem:t_1=t_2=1}] The upper bound of strongly exceptional Legendrian $A_3$ links is given by Lemma~\ref{Lemma:t1=t2=1}.  We will show that these upper bounds can be attained.

\begin{figure}[htb]
\begin{overpic}
[width=0.97\textwidth]
{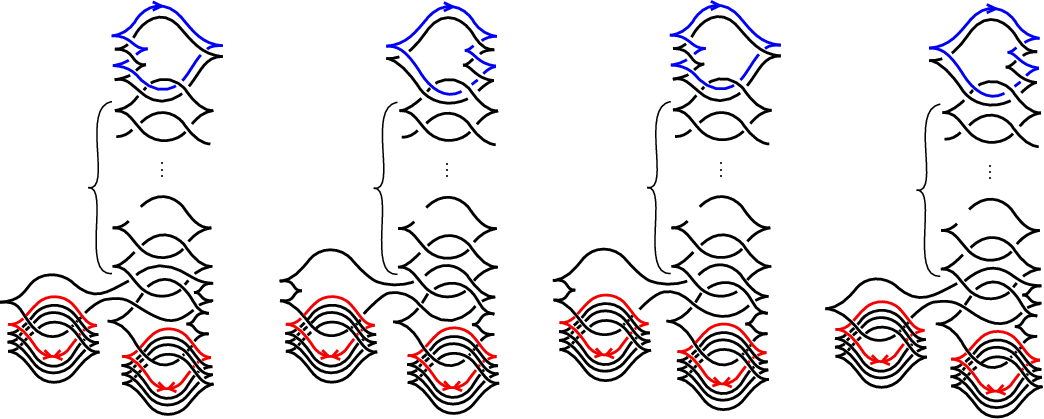}
\put(5, 95){$t_{0}-6$}
\put(124, 95){$t_{0}-6$}
\put(238, 95){$t_{0}-6$}
\put(351, 95){$t_{0}-6$}
\put(36, 164){$K_{0}$}
\put(-12, 39){$K_{1}$}
\put(88, 29){$K_{2}$}
\put(88, 0){$+1$}
\put(88, 9){$+1$}
\put(88, 18){$+1$}
\put(88, 39){$-1$}
\put(88, 60){$-1$}
\put(88, 74){$-1$}
\put(88, 109){$-1$}
\put(88, 123){$-1$}
\put(90, 143){$+1$}
\put(-13, 29){$+1$}
\put(-13, 20){$+1$}
\put(-13, 11){$+1$}
\put(-13, 50){$-1$}
\put(155, 164){$K_{0}$}
\put(104, 38){$K_{1}$}
\put(207, 24){$K_{2}$}
\put(270, 164){$K_{0}$}
\put(218, 38){$K_{1}$}
\put(320, 24){$K_{2}$}
\put(385, 164){$K_{0}$}
\put(333, 37){$K_{1}$}
\put(435, 24){$K_{2}$}

\end{overpic}
\caption{$t_0\geq6$, $t_1=t_2=1$.  For $t_0$ even, $K_0$ and $K_{i}$, $i=1,2$, bear the same orientation, for $t_{0}$ odd, the opposite one.} 
\label{Figure:link19inot}
\end{figure}

(1) Suppose $t_0\geq6$.

\begin{lemma}\label{t0>5t1=1t2=1}
If $t_0\geq6, t_1=t_2=1$, then there exist $8$ strongly exceptional Legendrian $A_3$ links whose rotation numbers and corresponding $d_3$ invariants $(r_0, r_1, r_2; d_3)$ are 
$$(\pm(t_{0}+3), \pm2, \pm2; -\frac{3}{2}),
(\pm(t_{0}-1), \pm2, 0; \frac{1}{2}),
(\pm(t_{0}-1), 0, \pm2; \frac{1}{2}), (\pm(t_{0}-5), 0, 0; \frac{5}{2}).$$
\end{lemma}

\begin{proof}
There exist $8$ strongly exceptional Legendrian $A_3$ links shown in Figure~\ref{Figure:link19inot}. Using the trick of Lemma~\ref{Lemma:topological}, the upper branch in each of the surgery diagrams can be topologically reduced to a single unknot, and the lower two branches in each of the surgery diagrams can be split. Furthermore, using the trick in the proof of \cite[Lemma 5.1, part (ii), Figure 5]{go}, we can show that $K_0\cup K_1\cup K_2$ is a topological $A_3$ link. Their rotation numbers are 
$$r_0=\pm(t_{0}+3), r_{1}=\pm2, r_{2}=\pm2;
r_0=\pm(t_{0}-1), r_{1}=\pm2, r_{2}=0;$$
$$r_0=\pm(t_{0}-1), r_{1}=0, r_{2}=\pm2; r_0=\pm(t_{0}-5), r_{1}=r_{2}=0.$$
The corresponding $d_3$-invariants are $-\frac{3}{2}, \frac{1}{2}, \frac{1}{2}, \frac{5}{2}$. These $d_3$-invariants are calculated using the algorithm described in \cite{DGS}.
\end{proof}

(2) Suppose $t_0=5$. 

\begin{lemma}\label{t0=5t1=1t2=1}
If $t_0=5, t_1=t_2=1$, then there exist $7$ strongly exceptional Legendrian $A_3$ links whose rotation numbers and corresponding $d_3$ invariants $(r_0, r_1, r_2; d_3)$ are
$$(\pm4, 0, \pm2; \frac{1}{2}), (0, 0, 0; \frac{5}{2}), (\pm4, \pm2, 0; \frac{1}{2}), (\pm8, \pm2, \pm2; -\frac{3}{2}).$$   
\end{lemma}

\begin{proof}
By \cite[Theorem 1.2, (c1), (c2)]{go}, there is a Legendrian Hopf link $K'_0\cup K_1$ in $(S^3, \xi_{\frac{1}{2}})$ with $(t'_0, r'_0)=(t_1, r_1)=(1,0)$,  two Legendrian Hopf links $K''_0\cup K_2$ in  $(S^3, \xi_{-\frac{1}{2}})$ with $(t''_0, r''_0)=(3, \pm4), (t_2, r_2)=(1, \pm2)$, and a Legendrian Hopf links $K''_0\cup K_2$ in  $(S^3, \xi_{\frac{3}{2}})$ with $(t''_0, r''_0)=(3, 0 ), (t_2, r_2)=(1, 0)$. Connected summing $K'_0$ and $K''_0$, by Lemma~\ref{Lemma:connected}, we obtain $3$ strongly exceptional Legendrian $A_3$ links  with $t_0=5, t_1=t_2=1$. Their rotation numbers and corresponding $d_3$-invariants $(r_0, r_1, r_2; d_3)$ are 
$(\pm4, 0, \pm2; \frac{1}{2})$ and $ (0, 0, 0; \frac{5}{2}).$

By exchanging the roles of $K_1$ and $K_2$, we obtain  $2$ strongly exceptional Legendrian $A_3$ links in $(S^3, \xi_{\frac{1}{2}})$ whose rotation numbers $(r_0, r_1, r_2)$ are $(\pm4, \pm2, 0).$

By Lemma~\ref{Lemma:3.2.1} and  Lemma~\ref{Lemma:appro}, there are $2$ strongly exceptional Legendrian $A_3$ links in $(S^3, \xi_{-\frac{3}{2}})$ whose rotation numbers $(r_0, r_1, r_2)$ are
$(\pm8, \pm2, \pm2).$
Their exteriors have decorations $\pm(+)(-)(-)$.
\end{proof}


(3) Suppose $t_0=4$. 

\begin{lemma}\label{t0=4t1=1t2=1}
If $t_0=4, t_1=t_2=1$, then there exist $6$ strongly exceptional Legendrian $A_3$ links whose rotation numbers and corresponding $d_3$ invariants $(r_0, r_1, r_2; d_3)$ are
$$(\pm3, 0, \pm2; \frac{1}{2}), (\pm3, \pm2, 0; \frac{1}{2}), (\pm7, \pm2, \pm2; -\frac{3}{2}).$$   
\end{lemma}

\begin{proof}
Suppose $t_0=4$. By \cite[Theorem 1.2]{go}, there is a Legendrian Hopf link $K'_0\cup K_1$ in $(S^3, \xi_{\frac{1}{2}})$ with $(t'_0, r'_0)=(t_1, r_1)=(1,0)$, and two Legendrian Hopf links $K''_0\cup K_2$ in $(S^3, \xi_{-\frac{1}{2}})$ with $(t''_0, r''_0)=(2, \pm3), (t_2, r_2)=(1, \pm2)$. Connected summing $K'_0$ and $K''_0$, by Lemma~\ref{Lemma:connected}, we obtain $2$ strongly exceptional Legendrian $A_3$ link with $t_0=4, t_1=t_2=1$  in $(S^3, \xi_{\frac{1}{2}})$. Their rotation numbers $(r_0, r_1, r_2)$ are $(\pm3, 0, \pm2).$ 

By exchanging the roles of $K_1$ and $K_2$ we obtain $2$ strongly exceptional Legendrian $A_3$ links  in $(S^3, \xi_{\frac{1}{2}})$  whose rotation numbers $(r_0, r_1, r_2)$ are
$(\pm3, \pm2, 0).$

By Lemma~\ref{Lemma:3.2.1} and Lemma~\ref{Lemma:appro}, there are $2$  strongly exceptional Legendrian $A_3$ links in $(S^3, \xi_{-\frac{3}{2}})$  whose rotation numbers $(r_0, r_1, r_2)$ are
$(\pm7, \pm2, \pm2).$
Their exteriors have decorations $\pm(+)(-)(-)$. 
\end{proof}

(4) Suppose $t_{0}\leq3$.

\begin{lemma}\label{t0<4t1=1t2=1}
If $t_{0}\leq3, t_1=t_2=1$, then there exist $4-t_0$ strongly exceptional Legendrian $A_3$ links in $(S^3, \xi_{\frac{3}{2}})$ whose rotation numbers are $$r_{0}\in\{t_{0}-3, t_{0}-1,\cdots, 3-t_{0}\}, r_{1}=r_{2}=0.$$
\end{lemma}

\begin{proof}
Suppose $t_{0}\leq3$. By \cite[Theorem 1.2, (c1), (b2)]{go}, there is a Legendrian Hopf link $K'_0\cup K_1$ in $(S^3, \xi_{\frac{1}{2}})$  with $(t'_0, r'_0)=(t_1, r_1)=(1, 0)$, and a Legendrian Hopf link $K''_0\cup K_2$  in $(S^3, \xi_{\frac{1}{2}})$  with $t''_0\leq1, r''_0\in\{t''_0-1, t''_0+1, \cdots, -t''_0+1\}, (t_2, r_2)=(1,0)$.  By Lemma~\ref{Lemma:connected}, we can construct  $4-t_{0}$ strongly exceptional Legendrian $A_3$ links in $(S^3, \xi_{\frac{3}{2}})$ with $t_0\leq3, t_1=t_2=1$. 
Their rotation numbers are as listed.
\end{proof} 
These $4-t_0$ strongly exceptional Legendrian $A_3$ links are obtained by stabilizations along $K_0$ of the Legendrian $A_3$ link with $t_0=3, t_1=t_2=1$.

The proof of Theorem~\ref{Theorem:t_1=t_2=1} is completed.
\end{proof}

\subsubsection{$t_1\geq2$ and $t_2=1$.}

The boundary slopes of $\Sigma\times S^1$ are $s_0=t_0$, $s_1=-\frac{1}{t_1}$ and $s_2=-1$.

\begin{proof}[Proof of Theorem~\ref{Theorem:t_1>1t_2=1}]
The upper bound of strongly exceptional Legendrian $A_3$ links is given by Lemma~\ref{Lemma:t_1>1t_2=1}. We will show that these upper bounds can be attained except in the cases $(t_0, t_1, t_2)=(3,3,1)$.

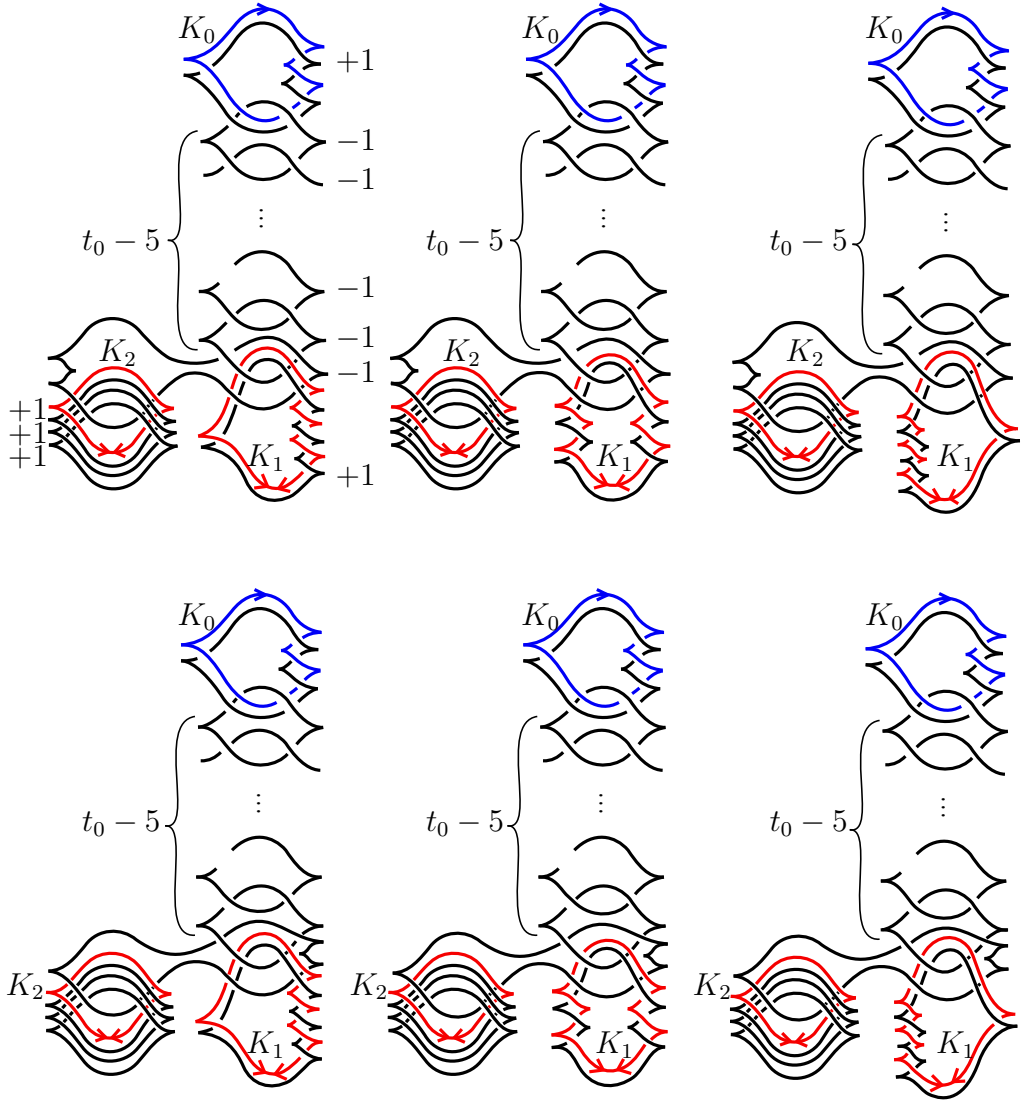
\begin{figure}[htb]
\begin{overpic}
{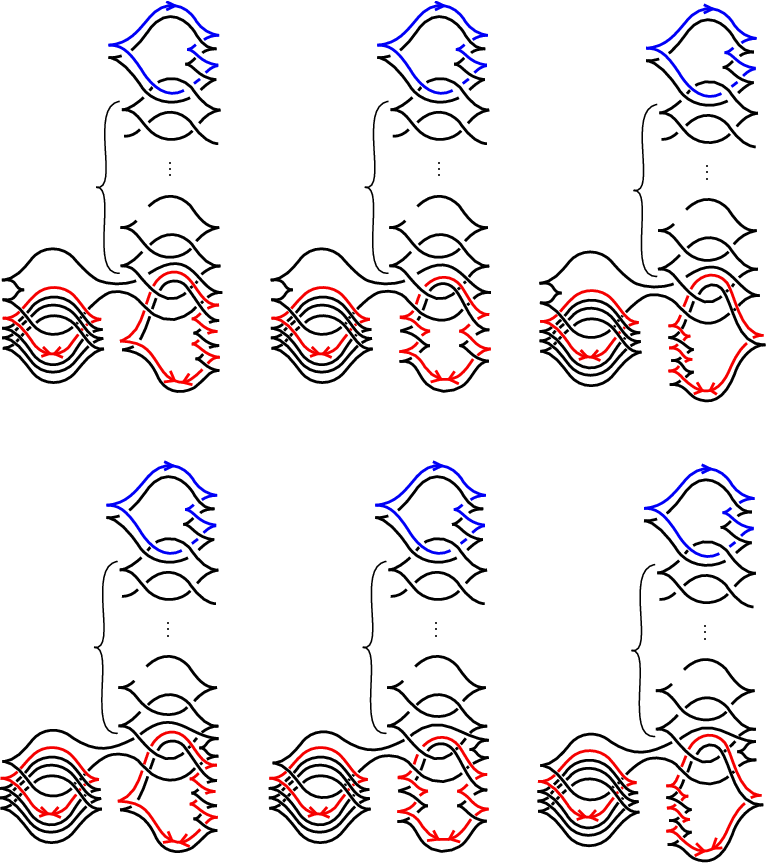}
\put(50, 402){$K_{0}$}
\put(180, 402){$K_{0}$}
\put(310, 402){$K_{0}$}
\put(50, 180){$K_{0}$}
\put(180, 180){$K_{0}$}
\put(310, 180){$K_{0}$}
\put(76, 240){$K_{1}$}
\put(208, 240){$K_{1}$}
\put(337, 240){$K_{1}$}
\put(76, 18){$K_{1}$}
\put(208, 18){$K_{1}$}
\put(337, 18){$K_{1}$}
\put(20, 280){$K_{2}$}
\put(150, 280){$K_{2}$}
\put(280, 280){$K_{2}$}
\put(-15, 40){$K_{2}$}
\put(115, 40){$K_{2}$}
\put(245, 40){$K_{2}$}
\put(14, 322){$t_{0}-5$}
\put(144, 322){$t_{0}-5$}
\put(274, 322){$t_{0}-5$}
\put(14, 102){$t_{0}-5$}
\put(144, 102){$t_{0}-5$}
\put(274, 102){$t_{0}-5$}
\put(110, 390){$+1$}
\put(110, 360){$-1$}
\put(110, 345){$-1$}
\put(110, 304){$-1$}
\put(110, 285){$-1$}
\put(110, 272){$-1$}
\put(110, 233){$+1$}
\put(-14, 258){$+1$}
\put(-14, 249){$+1$}
\put(-14, 240){$+1$}
\end{overpic}
\caption{$t_0\geq5$, $t_1=2$, $t_2=1$.  For $t_0$ odd, $K_0$ and $K_{2}$  bear the same orientation, while $K_0$ and $K_{1}$  bear the opposite orientation. For $t_{0}$ even, $K_0$ and $K_{1}$  bear the same orientation, while $K_0$ and $K_{2}$  bear the opposite orientation. }
\label{Figure:link21inot}
\end{figure}

(1) Suppose $t_0\geq5$ and $t_1=2$.

\begin{lemma}\label{t0>4t1=2t2=1}
If $t_0\geq5$, $t_1=2$ and $t_2=1$, then there exist $12$ strongly exceptional Legendrian $A_3$ links whose rotation numbers and corresponding $d_3$ invariants $(r_0, r_1, r_2; d_3)$ are $$(\pm(t_{0}-5), \mp1, 0; \frac{5}{2}), (\pm(t_{0}-3), \pm1,  0; \frac{5}{2}),
(\pm(t_{0}-1), \pm3,  0;  \frac{1}{2}),$$ $$(\pm(t_{0}-1), \mp1, \pm2; \frac{1}{2}),
(\pm(t_{0}+1), \pm1, \pm2; \frac{1}{2}), (\pm(t_{0}+3), \pm3, \pm2; -\frac{3}{2}).$$
\end{lemma}

\begin{proof}
There exist $12$ strongly exceptional Legendrian $A_3$ links shown in Figure~\ref{Figure:link21inot}. Using the trick of Lemma~\ref{Lemma:topological} and the proof of \cite[Theorem 1.2, (c3)]{go}, we can show that $K_0\cup K_1\cup K_2$ is a topological $A_3$ link. Their rotation numbers and corresponding $d_3$-invariants are as listed.
\end{proof}

(2) Suppose $t_0=4$ and $t_1=2$. 

\begin{lemma}\label{t0=4t1=2t2=1}
If $t_0=4$, $t_1=2$ and $t_2=1$, then there exist $10$ strongly exceptional Legendrian $A_3$ links whose rotation numbers and corresponding $d_3$-invariants $(r_0, r_1, r_2; d_3)$ are $$(\pm5, \pm1, \pm2; \frac{1}{2}), (\pm3, \mp1, \pm2; \frac{1}{2}), (\pm1, \pm1,  0; \frac{5}{2}),
(\pm3, \pm3,  0;  \frac{1}{2}), (\pm7, \pm3, \pm2; -\frac{3}{2}).$$
\end{lemma}

\begin{proof}
By \cite[Theorem 1.2, (c2), (d)]{go}, there are two Legendrian Hopf links $K'_0\cup K_1$ in  $(S^3, \xi_{\frac{1}{2}})$ with $(t'_0, r'_0)=(0,\pm1), (t_1, r_1)=(2, \pm1)$,  two Legendrian Hopf links $K''_0\cup K_2$ in  $(S^3, \xi_{-\frac{1}{2}})$ with $(t''_0, r''_0)=(3, \pm4), (t_2, r_2)=(1, \pm2)$, and a Legendrian Hopf link $K''_0\cup K_2$ in  $(S^3, \xi_{\frac{3}{2}})$ with $(t''_0, r''_0)=(3, 0 ), (t_2, r_2)=(1, 0)$.  By Lemma~\ref{Lemma:connected}, we can obtain $6$ strongly exceptional Legendrian $A_3$ links whose rotation numbers and $d_3$-invariants $(r_0, r_1, r_2; d_3)$ are 
$(\pm5, \pm1, \pm2; \frac{1}{2}), (\pm3,\mp1, \pm2; \frac{1}{2})$ and $ (\pm1, \pm1, 0; \frac{5}{2}).$

By \cite[Theorem 1.2, (c1), (c2)]{go}, there is a Legendrian Hopf link $K'_0\cup K_2$ in $(S^3, \xi_{\frac{1}{2}})$ with $(t'_0, r'_0)=(t_2, r_2)=(1,0)$, and four Legendrian Hopf links $K''_0\cup K_1$ with $(t''_0, r''_0)=(t_1, r_1)=(2, \pm3)$ in $(S^3, \xi_{-\frac{1}{2}})$ or $(2, \pm1)$ in $(S^3, \xi_{\frac{3}{2}})$.  By Lemma~\ref{Lemma:connected}, we can obtain $2$  more strongly exceptional Legendrian $A_3$ links whose rotation numbers and $d_3$-invariants $(r_0, r_1, r_2; d_3)$ are 
$(\pm3, \pm3, 0; \frac{1}{2}).$

By Lemma~\ref{Lemma:3.2.1} and Lemma~\ref{Lemma:appro}, there exist $2$ strongly exceptional Legendrian $A_3$ links in $(S^3, \xi_{-\frac{3}{2}})$ whose rotation numbers $(r_0, r_1, r_2)$ are
$(\pm7, \pm3, \pm2).$
Their exteriors have decorations $\pm(+)(--)(-)$.
\end{proof}

So there exist $10$ strongly exceptional Legendrian $A_3$ links with $t_{0}=4, t_{1}=2, t_{2}=1$. As a corollary, the $10$ contact structures on $\Sigma\times S^1$ with boundary slopes $s_{0}=4, s_{1}=-\frac{1}{2}, s_{2}=-1$ listed in Lemma~\ref{Lemma:t_1>1t_2=1}  are all appropriate tight.

(3) Suppose $t_0=3$ and $t_1=2$.  

\begin{lemma}\label{t0=3t1=2t2=1}
If $t_0=3$, $t_1=2$ and $t_2=1$, then there exist $8$ strongly exceptional Legendrian $A_3$ links whose rotation numbers and corresponding $d_3$ invariants $(r_0, r_1, r_2; d_3)$ are $$
(\pm2, \pm3,  0;  \frac{1}{2}), (\pm2, \mp1, \pm2; \frac{1}{2}),
(\pm4, \pm1, \pm2; \frac{1}{2}), (\pm6, \pm3, \pm2; -\frac{3}{2}).$$
\end{lemma}

\begin{proof}
By \cite[Theorem 1.2, (c2), (c1)]{go}, there are two Legendrian Hopf links $K'_0\cup K_1$ in $(S^3, \xi_{-\frac{1}{2}})$ with $(t'_0, r'_0)=(1, \pm2), (t_1, r_1)=(2, \pm3)$, and one Legendrian Hopf link $K''_0\cup K_2$ in  $(S^3, \xi_{\frac{1}{2}})$  with $(t''_0, r''_0)=(t_2, r_2)=(1,0)$. By Lemma~\ref{Lemma:connected}, we can obtain $2$ strongly exceptional Legendrian $A_3$ links whose rotation numbers and $d_3$-invariants $(r_0, r_1, r_2; d_3)$ are
$(\pm2, \pm3, 0; \frac{1}{2}).$

By \cite[Theorem 1.2, (d), (c2)]{go}, there are two Legendrian Hopf links $K'_0\cup K_1$ in $(S^3, \xi_{\frac{1}{2}})$ with $(t'_0, r'_0)=(0, \pm1), (t_1, r_1)=(2, \pm1)$, and two Legendrian Hopf links $K''_0\cup K_2$  in $(S^3, \xi_{-\frac{1}{2}})$  with $(t''_0, r''_0)=(2, \pm3),$  $ (t_2, r_2)=(1, \pm2)$. By Lemma~\ref{Lemma:connected}, we can obtain  $4$ strongly exceptional Legendrian $A_3$ links whose rotation numbers and $d_3$-invariants $(r_0, r_1, r_2; d_3)$ are
$(\pm2, \mp1,  \pm2; \frac{1}{2}) $ and $ (\pm4, \pm1, \pm2; \frac{1}{2}).$

By Lemma~\ref{Lemma:3.2.1} and Lemma~\ref{Lemma:appro}, there are $2$  strongly exceptional Legendrian $A_3$ links in $(S^3, \xi_{-\frac{3}{2}})$  whose rotation numbers $(r_0, r_1, r_2)$ are
$(\pm6, \pm3, \pm2).$
Their exteriors have decorations $\pm(+)(--)(-)$.
\end{proof}

So there are $8$ strongly exceptional Legendrian $A_3$ links with $t_{0}=3, t_{1}=2, t_{2}=1$. As a corollary, the $8$ contact structures on $\Sigma\times S^1$ with boundary slopes $s_{0}=3, s_{1}=-\frac{1}{2}, s_{2}=-1$  listed in Lemma~\ref{Lemma:t_1>1t_2=1} are all appropriate tight.

\begin{figure}[htb]
\begin{overpic}
[width=0.93\textwidth]
{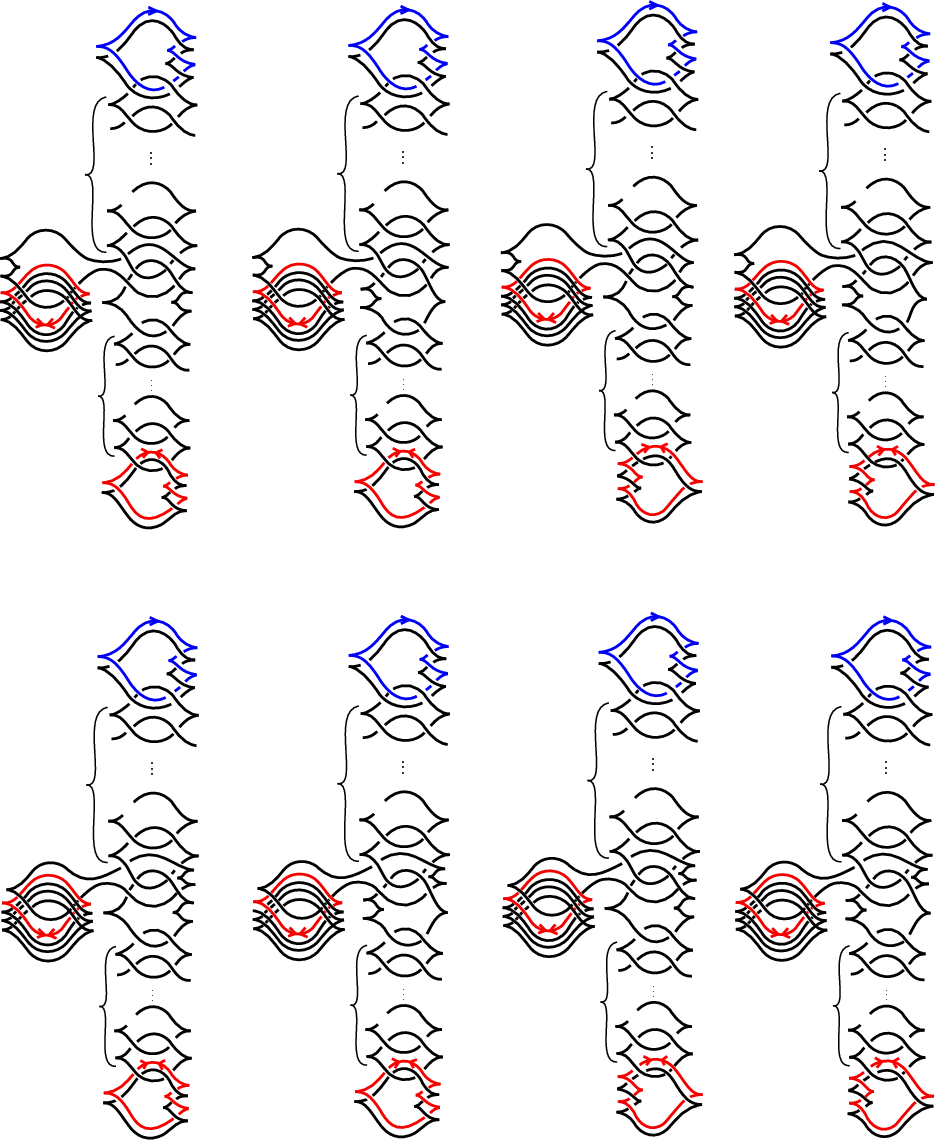}
\put(7, 430){$t_{0}-5$}
\put(13, 331){$t_{1}-3$}
\put(44, 500){$K_{0}$}
\put(57, 283){$K_{1}$}
\put(16, 392){$K_{2}$}
\end{overpic}
\caption{$t_0\geq5$,  $t_1\geq3$, $t_2=1$.  For $t_0+t_1$ even, $K_0$ and $K_1$ bear the same orientation, and for $t_0+t_1$ odd, the opposite one. For $t_0$ odd, $K_0$ and $K_2$ bear the same orientation, and for $t_0$ even, the opposite one. }
\label{Figure:link22inot}
\end{figure}

(4) Suppose $t_0\geq5$ and $t_1\geq3$.

\begin{lemma}\label{t0>4t1>2t2=1}
If $t_0\geq5$, $t_1\geq3$ and $t_2=1$, then there exist $16$ strongly exceptional Legendrian $A_3$ links whose rotation numbers and corresponding $d_3$ invariants $(r_0, r_1, r_2; d_3)$ are 
$$(\pm(t_{0}+1), \pm(t_{1}-1), \pm2; \frac{1}{2}), 
(\pm(t_{0}+3), \pm(t_{1}+1), \pm2; -\frac{3}{2}), $$
$$(\pm(t_{0}-1), \pm(1-t_{1}), \pm2; \frac{1}{2}), 
(\pm(t_{0}+1), \pm(3-t_{1}), \pm2;  \frac{1}{2}), (\pm(t_{0}-3), \pm(t_{1}-1), 0; \frac{5}{2}), $$
$$ 
(\pm(t_{0}-1), \pm(t_{1}+1), 0; \frac{1}{2}), (\pm(t_{0}-5), \pm(1-t_1), 0; \frac{5}{2}), (\pm(t_{0}-3), \pm(3-t_1), 0; \frac{5}{2}). $$
\end{lemma}

\begin{proof}
There exist $16$ strongly exceptional Legendrian $A_3$ links shown in Figure~\ref{Figure:link22inot}. Using the trick of Lemma~\ref{Lemma:topological} and the proof of \cite[Theorem 1.2, (c3), (c4)]{go}, we can show that $K_0\cup K_1\cup K_2$ is a topological $A_3$ link. Their rotation numbers and corresponding $d_3$-invariants are as listed. 
\end{proof}

(5) Suppose $t_0=4$ and $t_1\geq3$. 

\begin{lemma}\label{t0=4t1>2t2=1}
If $t_0=4$, $t_1\geq3$ and $t_2=1$, then there exist $14$ strongly exceptional Legendrian $A_3$ links whose rotation numbers and corresponding $d_3$ invariants $(r_0, r_1, r_2; d_3)$ are 
$$(\pm3, \pm(t_{1}+1), 0; \frac{1}{2}), (\mp1, \pm(1-t_1), 0; \frac{5}{2}), (\pm1, \pm(3-t_1), 0; \frac{5}{2}),  $$
$$(\pm5, \pm(t_{1}-1), \pm2; \frac{1}{2}), (\pm3, \pm(1-t_{1}), \pm2; \frac{1}{2}), (\pm7, \pm(t_{1}+1), \pm2; -\frac{3}{2}),(\pm5, \pm(3-t_{1}), \pm2; \frac{1}{2}). $$
\end{lemma}

\begin{proof}
By \cite[Theorem 1.2, (c3), (c1)]{go}, there are two Legendrian Hopf links $K'_0\cup K_1$ in $(S^3, \xi_{-\frac{1}{2}})$ with $(t'_0, r'_0)=(2, \pm3), t_1\geq3, r_1=\pm(t_1+1)$, two Legendrian Hopf links $K'_0\cup K_1$ in $(S^3, \xi_{\frac{3}{2}})$ with $(t'_0, r'_0)=(2, \pm1), t_1\geq3, r_1=\pm(t_1-1)$, two Legendrian Hopf links $K'_0\cup K_1$ in $(S^3, \xi_{\frac{3}{2}})$ with $(t'_0, r'_0)=(2, \mp1), t_1\geq3, r_1=\pm(t_1-3)$, and one Legendrian Hopf link $K''_0\cup K_2$ in $(S^3, \xi_{\frac{1}{2}})$ with $(t''_0, r''_0)=(t_2, r_2)=(1,0)$.  By Lemma~\ref{Lemma:connected}, we can obtain $6$ strongly exceptional Legendrian $A_3$ links whose rotation numbers and corresponding $d_3$-invariants $(r_0, r_1, r_2; d_3)$ are $(\pm3, \pm(t_{1}+1), 0; \frac{1}{2}), (\mp1, \pm(1-t_1), 0; \frac{5}{2})$ and $ (\pm1, \pm(3-t_1), 0; \frac{5}{2}). $

By \cite[Theorem 1.2, (d), (c2)]{go}, there are two Legendrian Hopf links $K'_0\cup K_1$  in $(S^3, \xi_{\frac{1}{2}})$ with $(t'_0, r'_0)=(0, \pm1), t_1\geq3, r_1=\pm(t_{1}-1)$,  two Legendrian Hopf links $K''_0\cup K_2$ in  $(S^3, \xi_{-\frac{1}{2}})$ with $(t''_0, r''_0)=(3, \pm4), (t_2, r_2)=(1, \pm2)$. By Lemma~\ref{Lemma:connected}, we can obtain $4$ more strongly exceptional Legendrian $A_3$ links whose rotation numbers and corresponding $d_3$-invariants $(r_0, r_1, r_2; d_3)$ are $(\pm5, \pm(t_{1}-1), \pm2; \frac{1}{2})$ and $ (\pm3, \pm(1-t_{1}), \pm2; \frac{1}{2}).$

By Lemma~\ref{Lemma:3.2.1} and Lemma~\ref{Lemma:appro}, there are  $2$ strongly exceptional Legendrian $A_3$ links in $(S^3, \xi_{-\frac{3}{2}})$  whose rotation numbers are
$(\pm7, \pm(t_{1}+1), \pm2).$
The decorations of their exteriors are $\pm(+)((-)(-))(-)$.

There are $2$ strongly exceptional Legendrian $A_3$ links in $(S^3, \xi_{\frac{1}{2}})$  whose rotation numbers $(r_0, r_1, r_2)$ are
$(\pm5, \pm(3-t_{1}), \pm2).$
The decorations of their exteriors are $\pm(+)((-)(+))(-)$. 
These exteriors can be embedded into an appropriate tight contact $\Sigma\times S^1$ with boundary slopes $4, -\frac{1}{2}, -1$ and decorations $\pm(+)(-+)(-)$. This can be achieved by adding basic slices $(T^2\times[0,1], -\frac{1}{t_{1}}, -\frac{1}{t_{1}-1})$, $\cdots$, $(T^2\times[0,1], -\frac{1}{3}, -\frac{1}{2})$ to the boundary $T_1$, as per the  Gluing Theorem \cite[Theorem 1.3]{h2}. So these exteriors are appropriate tight.
\end{proof}

(6) Suppose $t_0=3$ and $t_1\geq3$. 

\begin{lemma}\label{t0=3t1>2t2=1}
If $t_0=3$, $t_1\geq3$ and $t_2=1$, then there exist $12$ ($11$ if $t_1=3$) strongly exceptional Legendrian $A_3$ links whose rotation numbers and corresponding $d_3$-invariants $(r_0, r_1, r_2; d_3)$ are 
$$(\pm2, \pm(t_{1}+1), 0; \frac{1}{2}),  (0, \pm(3-t_1), 0; \frac{5}{2}), (\pm4, \pm(t_{1}-1), \pm2; \frac{1}{2}), 
(\pm2, \pm(1-t_{1}), \pm2; \frac{1}{2}), $$
$$(\pm6, \pm(t_{1}+1), \pm2; -\frac{3}{2}), (\pm4, \pm(3-t_{1}), \pm2; \frac{1}{2}).$$
\end{lemma}

\begin{proof}
By \cite[Theorem 1.2, (c3), (c2), (c1)]{go}, there are two  Legendrian Hopf links $K'_0\cup K_1$ in  $(S^3, \xi_{-\frac{1}{2}})$  with $(t'_0, r'_0)=(1, \pm2), t_1\geq3, r_{1}=\pm(t_{1}+1)$, two (one if $t_1=3$) Legendrian Hopf links $K'_0\cup K_1$ in  $(S^3, \xi_{\frac{3}{2}})$  with $(t'_0, r'_0)=(1, 0), t_1\geq3, r_{1}=\pm(t_{1}-3)$, and one Legendrian Hopf link $K''_0\cup K_2$ in $(S^3, \xi_{\frac{1}{2}})$ with $(t''_0, r''_0)=(t_2, r_2)=(1,0)$.   By Lemma~\ref{Lemma:connected}, we can obtain $4$ ($3$ if $t_1=3$) strongly exceptional Legendrian $A_3$ links whose rotation numbers and corresponding $d_3$-invariants $(r_0, r_1, r_2; d_3)$ are $(\pm2, \pm(t_{1}+1), 0; \frac{1}{2})$ and $  (0, \pm(3-t_1), 0; \frac{5}{2}). $

By \cite[Theorem 1.2, (d), (c2)]{go}, there are two Legendrian Hopf links $K'_0\cup K_1$  in  $(S^3, \xi_{\frac{1}{2}})$  with $(t'_0, r'_0)=(0,\pm1), t_1\geq3, r_{1}=\pm(t_{1}-1)$, and two Legendrian Hopf link $K''_0\cup K_2$  in  $(S^3, \xi_{-\frac{1}{2}})$ with $(t''_0, r''_0)=(2, \pm3),(t_2, r_2)=(1, \pm2)$. By Lemma~\ref{Lemma:connected}, we can obtain strongly exceptional Legendrian $A_3$ links with $t_0=3, t_1\geq3, t_2=1$. So there are $4$ strongly exceptional Legendrian $A_3$ links  whose rotation numbers and corresponding $d_3$-invariants $(r_0, r_1, r_2; d_3)$ are 
$(\pm4, \pm(t_{1}-1), \pm2; \frac{1}{2})$ and $ 
(\pm2, \pm(1-t_{1}), \pm2; \frac{1}{2}).$

By Lemma~\ref{Lemma:3.2.1} and Lemma~\ref{Lemma:appro}, there are $2$ strongly exceptional Legendrian $A_3$ links in $(S^3, \xi_{-\frac{3}{2}})$  whose rotation numbers $(r_0, r_1, r_2)$ are $(\pm6, \pm(t_{1}+1), \pm2).$
The decorations of their exteriors are $\pm(+)((-)(-))(-)$.

There are $2$ strongly exceptional Legendrian $A_3$ links in $(S^3, \xi_{\frac{1}{2}})$  whose rotation numbers $(r_0, r_1, r_2)$ are
$(\pm4, \pm(3-t_{1}), \pm2).$
The decorations of their exteriors are $\pm(+)((-)(+))(-)$. These exteriors are appropriate tight since they can be embedded into an appropriate tight contact $\Sigma\times S^1$ with boundary slopes $3, -\frac{1}{2}, -1$ and decorations $\pm(+)(-+)(-)$ by adding basic slices $(T^2\times[0,1], -\frac{1}{t_{1}}, -\frac{1}{t_{1}-1})$, $\cdots$, $(T^2\times[0,1], -\frac{1}{3}, -\frac{1}{2})$ to the boundary $T_1$. 
\end{proof}

So there are exactly $12$ (resp. $11$) strongly exceptional Legendrian $A_3$ links with $t_{0}=3$, $t_{1}\geq4$ (resp. $t_{1}=3$), $t_{2}=1$. If $t_0=t_1=3$ and $t_2=1$, then the decorations $(+)((-)(+))(+)$ and $(-)((+)(-))(-)$ correspond to the same Legendrian $A_3$ links with rotation numbers $r_0=r_1=r_2=0$.


(7) Suppose $t_{0}\leq2$. 

\begin{lemma}\label{t0<3t1>1t2=1}
If $t_{0}\leq2$, $t_1>1$ and $t_2=1$, then there exist $6-2t_0$ strongly exceptional Legendrian $A_3$ links in $(S^3, \xi_{\frac{3}{2}})$ whose rotation numbers are $$r_{0}\in\pm\{t_{0}-1, t_{0}+1, \cdots, -t_{0}+1, -t_{0}+3\}, r_{1}=\pm(t_{1}-1), r_{2}=0.$$  
\end{lemma}

\begin{proof}

By \cite[Theorem 1.2, (b1), (c1)]{go}, there is a Legendrian Hopf link $K'_0\cup K_1$ in $(S^3, \xi_{\frac{1}{2}})$ with $t'_0\leq1, r'_0\in\pm\{t'_0+1, t'_0+3, \cdots, -t'_0-1, -t'_0+1\}$, $t_1\geq2$, $r_1=\pm(t_{1}-1)$, and a Legendrian Hopf link $K''_0\cup K_2$ in $(S^3, \xi_{\frac{1}{2}})$  with $(t''_0, r''_0)=(t_2, r_2)=(1, 0)$.  By Lemma~\ref{Lemma:connected}, we can construct $6-2t_{0}$ strongly exceptional Legendrian $A_3$ links in $(S^3, \xi_{\frac{3}{2}})$ with $t_{0}\leq2, t_1>1, t_2=1$.  Their rotation numbers are as listed.
\end{proof}
These $6-2t_0$ strongly exceptional Legendrian $A_3$ links are stabilizations of the Legendrian $A_3$ links with $t_{0}=2,  t_1>1, t_2=1$.

The proof of Theorem~\ref{Theorem:t_1>1t_2=1} is completed.
\end{proof}

\subsubsection{$t_1\geq2$ and $t_2\geq2$.}
\begin{proof}[Proof of Theorem~\ref{Theorem:t_1>1t_2>1}]
The upper bound of strongly exceptional Legendrian $A_3$ links is given by Lemma~\ref{Lemma:t_1>1t_2>1}. We will show that these upper bounds can be attained.
\begin{figure}[htb]
\begin{overpic}
[width=0.95\textwidth]
{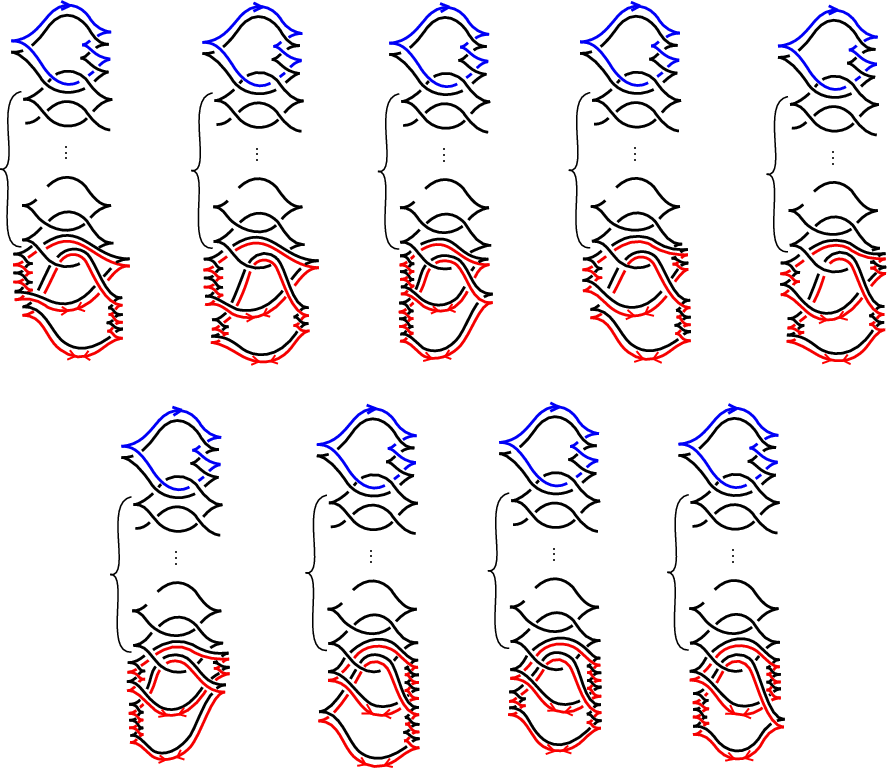}
\put(-30, 283){$t_{0}-4$}
\put(0, 358){$K_{0}$}
\put(32, 210){$K_{1}$}
\put(32, 187){$K_{2}$}
\end{overpic}
\caption{$t_0\geq4$, $t_{1}=t_{2}=2$.   For $t_0$ odd, $K_{0}$ and $K_{i}$ are given the same orientation, for $t_0$ even, the opposite one, where $i=1,2$. }
\label{Figure:link23inot}
\end{figure}

(1) Suppose $t_0\geq4$ and $t_1=t_2=2$.

\begin{lemma}\label{t0>3t1=2t2=2}
If $t_0\geq4$ and $t_1=t_2=2$, then there exist $18$ strongly exceptional Legendrian $A_3$ links whose rotation numbers and corresponding $d_3$-invariants $(r_0, r_1, r_2; d_3)$ are
$$(\pm(t_{0}-1), \pm3, \mp1; \frac{1}{2}), (\pm(t_{0}+1), \pm3, \pm1;  \frac{1}{2}),(\pm(t_{0}+3), \pm3, \pm3;   -\frac{3}{2}), $$ 
$$(\pm(t_{0}-3), \pm1, \mp1; \frac{5}{2}), (\pm(t_{0}-1), \pm1, \pm1;  \frac{5}{2}), (\pm(t_{0}+1), \pm1, \pm3; \frac{1}{2}),$$ 
$$(\pm(t_{0}-5), \mp1, \mp1;  \frac{5}{2}), (\pm(t_{0}-3), \mp1, \pm1;  \frac{5}{2}), (\pm(t_{0}-1), \mp1, \pm3; \frac{1}{2}).$$ 
\end{lemma}

\begin{proof}

If $t_0\geq4$ and $t_1=t_2=2$, then there exist $18$ strongly exceptional Legendrian $A_3$ links shown in Figure~\ref{Figure:link23inot}.  Using the trick of Lemma~\ref{Lemma:topological} and the proof of \cite[Theorem 1.2, (c3)]{go}, we can show that $K_0\cup K_1\cup K_2$ is a topological $A_3$ link. Their rotation numbers and corresponding $d_3$-invariants are as listed. 
\end{proof}

(2) Suppose $t_0=3$ and $t_1=t_2=2$. 
\begin{lemma}\label{t0=3t1=2t2=2}
If $t_0=3$ and $t_1=t_2=2$, then there exist $14$ strongly exceptional Legendrian $A_3$ links whose rotation numbers and corresponding $d_3$-invariants $(r_0, r_1, r_2; d_3)$ are
$$(\pm4, \pm3, \pm1; \frac{1}{2}), (\pm4, \pm1, \pm3; \frac{1}{2}), (\pm2, \pm3, \mp1; \frac{1}{2}), (\pm2, \mp1, \pm3; \frac{1}{2}),$$ 
$$(\mp2, \mp1, \mp1; \frac{5}{2}), (0, \mp1, \pm1; \frac{5}{2}), (\pm6, \pm3, \pm3; -\frac{3}{2}).$$ 
\end{lemma}
\begin{proof}

By \cite[Theorem 1.2, (c2), (d)]{go}, there are two Legendrian Hopf links $K'_0\cup K_1$ in $(S^3, \xi_{-\frac{1}{2}})$ with $(t'_0, r'_0)=(t_1, r_1)=(2, \pm3)$, two Legendrian Hopf links $K'_0\cup K_1$ in $(S^3, \xi_{\frac{3}{2}})$ with $(t'_0, r'_0)=(t_1, r_1)=(2, \pm1)$, and two Legendrian Hopf links $K''_0\cup K_2$ in $(S^3, \xi_{\frac{1}{2}})$ with $(t''_0, r''_0)=(0, \pm1), (t_2, r_2)=(2, \pm1)$. By Lemma~\ref{Lemma:connected}, we can obtain strongly exceptional Legendrian $A_3$ links with $t_0=3, t_1=t_2=2$. So by exchanging the roles of $K_1$ and $K_2$ there are  $12$ strongly exceptional Legendrian $A_3$ links  whose rotation numbers and corresponding $d_3$-invariants $(r_0, r_1, r_2; d_3)$ are 
$(\pm4, \pm3, \pm1; \frac{1}{2}), $ $(\pm4, \pm1, \pm3; \frac{1}{2}), $ $(\pm2, \pm3, \mp1; \frac{1}{2}), $ $(\pm2, \mp1, \pm3; \frac{1}{2}),$
$(\mp2, \mp1, \mp1; \frac{5}{2})$ and  $(0, \mp1, \pm1; \frac{5}{2}).$ 

By Lemma~\ref{Lemma:3.2.1} and Lemma~\ref{Lemma:appro}, there are $2$ strongly exceptional Legendrian $A_3$ links in $(S^3, \xi_{-\frac{3}{2}})$  whose rotation numbers $(r_0, r_1, r_2)$ are
$(\pm6, \pm3, \pm3).$ 
The decorations of their exteriors are $\pm(+)(--)(--)$.
\end{proof}

So there are exactly $14$ strongly exceptional Legendrian $A_3$ links with $t_{0}=3, t_{1}=2, t_{2}=2$. As a corollary, the $14$ contact structures on $\Sigma\times S^1$ with boundary slopes $s_{0}=3, s_{1}=-\frac{1}{2}, s_{2}=-\frac{1}{2}$ listed in Lemma~\ref{Lemma:t_1>1t_2>1}  are all appropriate tight.

(3) Suppose $t_0=t_1=t_2=2$. 

\begin{lemma}\label{t0=2t1=2t2=2}
If $t_0=t_1=t_2=2$, then there exist $10$ strongly exceptional Legendrian $A_3$ links whose rotation numbers and corresponding $d_3$-invariants $(r_0, r_1, r_2; d_3)$ are
$$(\pm3, \pm3, \pm1; \frac{1}{2}), (\pm3, \pm1, \pm3; \frac{1}{2}), (\pm1, \pm3, \mp1; \frac{1}{2}), (\pm1, \mp1, \pm3; \frac{1}{2}), (\pm5, \pm3, \pm3; -\frac{3}{2}). $$
\end{lemma}
\begin{proof}
By \cite[Theorem 1.2, (c2), (d)]{go}, there are two Legendrian Hopf links $K'_0\cup K_1$ in $(S^3, \xi_{-\frac{1}{2}})$ with $(t'_0, r'_0)=(1, \pm2), (t_1, r_1)=(2, \pm3)$, and two Legendrian Hopf links $K''_0\cup K_2$ in $(S^3, \xi_{\frac{1}{2}})$  with $(t''_0, r''_0)=(0, \pm1), (t_2, r_2)=(2, \pm1)$. By Lemma~\ref{Lemma:connected}, we can obtain strongly exceptional Legendrian $A_3$ links with $t_0=t_1=t_2=2$. So by exchanging the roles of $K_1$ and $K_2$ there are  $8$ strongly exceptional Legendrian $A_3$ links whose rotation numbers and corresponding $d_3$-invariants $(r_0, r_1, r_2; d_3)$ are are 
$(\pm3, \pm3, \pm1; \frac{1}{2}), $ $(\pm3, \pm1, \pm3; \frac{1}{2}), $ $(\pm1, \pm3, \mp1; \frac{1}{2}) $ and  $(\pm1, \mp1, \pm3; \frac{1}{2}).$

By Lemma~\ref{Lemma:3.2.1} and Lemma~\ref{Lemma:appro}, there are $2$ strongly exceptional Legendrian $A_3$ links in $(S^3, \xi_{-\frac{3}{2}})$  whose rotation numbers $(r_0, r_1, r_2)$ are 
$(\pm5, \pm3, \pm3).$
The decorations of their exteriors are $\pm(+)(--)(--)$.
\end{proof}

So there are exactly $10$ strongly exceptional Legendrian $A_3$ links with $t_{0}=2, t_{1}=2, t_{2}=2$. As a corollary, the $10$ contact structures on $\Sigma\times S^1$ with boundary slopes $s_{0}=2, s_{1}=-\frac{1}{2}, s_{2}=-\frac{1}{2}$ listed in Lemma~\ref{Lemma:t_1>1t_2>1} are all appropriate tight.


\begin{figure}[htb]
\begin{overpic}
[width=0.95\textwidth]
{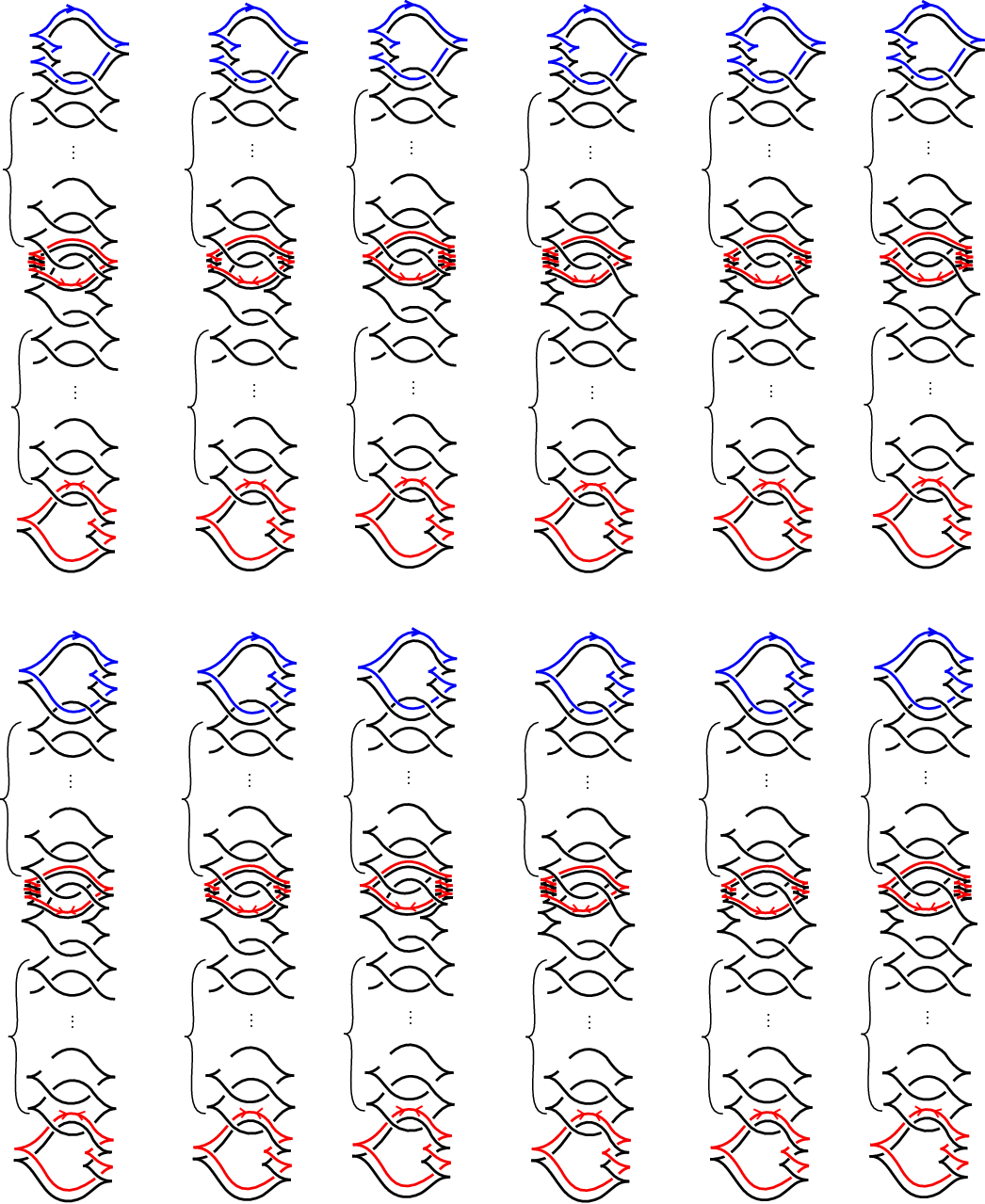}
\put(-30, 445){$t_{0}-4$}
\put(-28, 342){$t_{1}-3$}
\put(41, 512){$K_{0}$}
\put(23, 283){$K_{1}$}
\put(-4, 406){$K_{2}$}
\end{overpic}
\caption{$t_0\geq4$, $t_{1}\geq3$, $t_{2}=2$.  If $t_{0}+t_{1}$ is odd, then $K_{0}$ and $K_{1}$ bear the same orientation, if $t_{0}+t_{1}$ is even, then the opposite one. If $t_{0}$ is odd, then $K_{0}$ and $K_{2}$ bear the same orientation, if $t_{0}$ is even, then the opposite one. }
\label{Figure:link24inot}
\end{figure}

(4) Suppose $t_0\geq4$, $t_1\geq3$ and $t_2=2$.

\begin{lemma}\label{t0>3t1>2t2=2}
If $t_0\geq4$, $t_1\geq3$ and $t_2=2$, then there exist $24$ strongly exceptional Legendrian $A_3$ links  whose rotation numbers and corresponding $d_3$ invariants $(r_0, r_1, r_2; d_3)$ are
$$(\pm(t_{0}-3),  \pm(3-t_{1}),  \mp1;  \frac{5}{2}), (\pm(t_{0}-1),  \pm(3-t_{1}),  \pm1; \frac{5}{2}),$$
$$(\pm(t_{0}+1),  \pm(3-t_{1}),  \pm3; \frac{1}{2}), (\pm(t_{0}-5),  \pm(1-t_{1}),  \mp1; \frac{5}{2}),$$
$$(\pm(t_{0}-3),  \pm(1-t_{1}),  \pm1; \frac{5}{2}), (\pm(t_{0}-1),  \pm(1-t_{1}),  \pm3; \frac{1}{2}),$$
$$(\pm(t_{0}+1),  \pm(t_{1}-1),  \pm3;  \frac{1}{2}), (\pm(t_{0}-1),  \pm(t_{1}-1),  \pm1; \frac{5}{2}),$$
$$(\pm(t_{0}-3),  \pm(t_{1}-1),  \mp1;  \frac{5}{2}), (\pm(t_{0}+3),  \pm(t_{1}+1),  \pm3; -\frac{3}{2}),$$ 
$$(\pm(t_{0}+1),  \pm(t_{1}+1),  \pm1; \frac{1}{2}), (\pm(t_{0}-1),  \pm(t_{1}+1),  \mp1; \frac{1}{2}).$$
\end{lemma}
\begin{proof}
If $t_0\geq4$, $t_1\geq3$ and $t_2=2$, then there are exactly  $24$ strongly exceptional Legendrian $A_3$ links shown in Figure~\ref{Figure:link24inot}. Using the trick of Lemma~\ref{Lemma:topological} and the proof of \cite[Theorem 1.2, (c3), (c4)]{go}, we can show that $K_0\cup K_1\cup K_2$ is a topological $A_3$ link. Their rotation numbers and corresponding $d_3$-invariants are as listed.
\end{proof}

(5) Suppose $t_0=3, t_1\geq3$ and $ t_2=2$. 
\begin{lemma}\label{t0=3t1>2t2=2}
If $t_0=3, t_1\geq3, t_2=2$, then there exist $20$ strongly exceptional Legendrian $A_3$ links whose rotation numbers and corresponding $d_3$-invariants $(r_0, r_1, r_2; d_3)$ are
$$(\pm4,  \pm(t_{1}-1),  \pm3;  \frac{1}{2}), (\pm2,  \pm(1-t_{1}),  \pm3; \frac{1}{2}), (\mp2,  \pm(1-t_{1}),  \mp1; \frac{5}{2}), (0,  \pm(1-t_{1}),  \pm1; \frac{5}{2}),$$
$$(\pm4,  \pm(t_{1}+1),  \pm1; \frac{1}{2}), (\pm2,  \pm(t_{1}+1),  \mp1; \frac{1}{2}), (0, \pm(3-t_{1}), \mp1;  \frac{5}{2}), (\pm2, \pm(3-t_{1}), \pm1; \frac{5}{2}),$$
$$(\pm6, \pm(t_{1}+1), \pm3; -\frac{3}{2}), (\pm4, \pm(3-t_{1}), \pm3; \frac{1}{2}). $$
\end{lemma}
\begin{proof}

By \cite[Theorem 1.2, (d), (c2)]{go}, there are two Legendrian Hopf links $K'_0\cup K_1$  in $(S^3, \xi_{\frac{1}{2}})$ with $(t'_0, r'_0)=(0, \pm1), t_1\geq3, r_{1}=\pm(t_{1}-1)$, there are two Legendrian Hopf links $K''_0\cup K_2$ in $(S^3, \xi_{-\frac{1}{2}})$ with $(t''_0, r''_0)=(t_2, r_2)=(2, \pm3)$, and two Legendrian Hopf links $K''_0\cup K_2$ in $(S^3, \xi_{\frac{3}{2}})$ with $(t''_0, r''_0)=(t_1, r_2)=(2, \pm1)$. By Lemma~\ref{Lemma:connected}, we can obtain $8$ strongly exceptional Legendrian $A_3$ links whose rotation numbers and $d_3$-invariants $(r_0, r_1, r_2; d_3)$ are
$(\pm4,  \pm(t_{1}-1),  \pm3;  \frac{1}{2}), $ $ (\pm2,  \pm(1-t_{1}),  \pm3; \frac{1}{2}), $ $(\mp2,  \pm(1-t_{1}),  \mp1; \frac{5}{2}) $ and  $(0,  \pm(1-t_{1}),  \pm1; \frac{5}{2}).$

By \cite[Theorem 1.2, (c3), (d)]{go}, there are two Legendrian Hopf links $K'_0\cup K_1$ in  $(S^3, \xi_{-\frac{1}{2}})$ with $(t'_0, r'_0)=(2,\pm3), t_1\geq3, r_{1}=\pm(t_{1}+1)$, 
two Legendrian Hopf links $K'_0\cup K_1$ in  $(S^3, \xi_{\frac{3}{2}})$ with $(t'_0, r'_0)=(2,\mp1), t_1\geq3, r_{1}=\pm(t_{1}-3)$, and two Legendrian Hopf links $K''_0\cup K_2$ in $(S^3, \xi_{\frac{1}{2}})$ with $(t''_0, r''_0)=(0,\pm1), (t_2, r_2)=(2, \pm1)$. By Lemma~\ref{Lemma:connected}, we can obtain strongly exceptional Legendrian $A_3$ links with $t_0=3, t_1\geq3, t_2=2$. Then there are $8$ strongly exceptional Legendrian $A_3$ links whose rotation numbers and $d_3$-invariants are
$(\pm4,  \pm(t_{1}+1),  \pm1; \frac{1}{2}), $ $ (\pm2,  \pm(t_{1}+1),  \mp1; \frac{1}{2}), $ $(0, \pm(3-t_{1}), \mp1;  \frac{5}{2}) $ and  $(\pm2, \pm(3-t_{1}), \pm1; \frac{5}{2}).$

By Lemma~\ref{Lemma:3.2.1} and Lemma~\ref{Lemma:appro}, there are $2$ strongly exceptional Legendrian $A_3$ links in $(S^3, \xi_{-\frac{3}{2}})$  whose rotation numbers $(r_0, r_1, r_2)$ are $(\pm6, \pm(t_{1}+1), \pm3).$
The decorations of their exteriors are $\pm(+)((-)(-))(--)$.

There are $2$ strongly exceptional Legendrian $A_3$ links in $(S^3, \xi_{\frac{1}{2}})$  whose rotation numbers $(r_0, r_1, r_2)$ are
$(\pm4, \pm(3-t_{1}), \pm3).$
The decorations of their exteriors are $\pm(+)((-)(+))(--)$. These exteriors are appropriate tight since they can be embedded into an appropriate  tight contact $\Sigma\times S^1$ with boundary slopes $3, -\frac{1}{2}, -\frac{1}{2}$ and decorations $\pm(+)(-+)(--)$ by adding  basic slices $(T^2\times[0,1], -\frac{1}{t_{1}}, -\frac{1}{t_{1}-1})$, $\cdots$, $(T^2\times[0,1], -\frac{1}{3}, -\frac{1}{2})$ to the boundary $T_1$. 
\end{proof}

So there are exactly $20$ strongly exceptional Legendrian $A_3$ links with $t_{0}=3, t_{1}\geq3, t_{2}=2$. As a corollary, the $20$ contact structures on $\Sigma\times S^1$ with boundary slopes $s_{0}=3, s_{1}=-\frac{1}{t_1}, s_{2}=-\frac{1}{2}$ listed in Lemma~\ref{Lemma:t_1>1t_2>1}  are all appropriate tight.

(6) Suppose $t_0=2$, $t_1\geq3$ and $t_2=2$. 
\begin{lemma}\label{t0=2t1>2t2=2}
If $t_0=2$, $t_1\geq3$ and $t_2=2$, then there exist $16$ strongly exceptional Legendrian $A_3$ links whose rotation numbers and corresponding $d_3$-invariants $(r_0, r_1, r_2; d_3)$ are $$(\pm3,  \pm(t_{1}+1),  \pm1; \frac{1}{2}), (\pm1,  \pm(t_{1}+1),  \mp1; \frac{1}{2}),(\mp1,  \pm(3-t_{1}),  \mp1;  \frac{5}{2}), (\pm1,  \pm(3-t_{1}),  \pm1; \frac{5}{2}), $$  $$  (\pm1,  \pm(1-t_{1}),  \pm3; \frac{1}{2}), (\pm3,  \pm(t_{1}-1),  \pm3;  \frac{1}{2}), (\pm5, \pm(t_{1}+1), \pm3; -\frac{3}{2}), (\pm3, \pm(3-t_{1}), \pm3; \frac{1}{2}).$$

\end{lemma}
\begin{proof}
By \cite[Theorem 1.2, (c2), (c3), (d)]{go}, there are two  Legendrian Hopf links $K'_0\cup K_1$ in  $(S^3, \xi_{-\frac{1}{2}})$  with $(t'_0, r'_0)=(1, \pm2), t_1\geq3, r_{1}=\pm(t_{1}+1)$, two (one if $t_1=3$) Legendrian Hopf links $K'_0\cup K_1$ in  $(S^3, \xi_{\frac{3}{2}})$  with $(t'_0, r'_0)=(1, 0), t_1\geq3, r_{1}=\pm(t_{1}-3)$,  and two Legendrian Hopf links $K''_0\cup K_2$ in $(S^3, \xi_{\frac{1}{2}})$ with $(t''_0, r''_0)=(0, \pm1), (t_2, r_{2})=(2, \pm1)$. By Lemma~\ref{Lemma:connected}, we can obtain strongly exceptional Legendrian $A_3$ links with $t_0=2, t_1\geq3, t_2=2$. Then there are $8$ strongly exceptional Legendrian $A_3$ links  whose rotation numbers and corresponding $d_3$-invariants $(r_0, r_1, r_2; d_3)$ are $(\pm3,  \pm(t_{1}+1),  \pm1; \frac{1}{2}), $ $(\pm1,  \pm(t_{1}+1),  \mp1; \frac{1}{2}), $ $(\mp1,  \pm(3-t_{1}),  \mp1;  \frac{5}{2})$ and $(\pm1,  \pm(3-t_{1}),  \pm1; \frac{5}{2}).$

By \cite[Theorem 1.2, (d), (c2)]{go}, there are two Legendrian Hopf links $K'_0\cup K_1$ in $(S^3, \xi_{\frac{1}{2}})$ with $(t'_0, r'_0)=(0, \pm1), t_1\geq3, r_{1}=\pm(t_{1}-1)$, and two Legendrian Hopf links $K''_0\cup K_2$ in  $(S^3, \xi_{-\frac{1}{2}})$ with $(t''_0, r''_0)=(1, \pm2), (t_2, r_{2})=(2, \pm3)$. By Lemma~\ref{Lemma:connected}, we can obtain strongly exceptional Legendrian $A_3$ links with $t_0=2, t_1\geq3, t_2=2$. Then there are $4$ strongly exceptional Legendrian $A_3$ links  whose rotation numbers and corresponding $d_3$-invariants $(r_0, r_1, r_2; d_3)$ are
$(\pm1,  \pm(1-t_{1}),  \pm3; \frac{1}{2}) $ and $(\pm3,  \pm(t_{1}-1), \pm3;  \frac{1}{2}).$

By Lemma~\ref{Lemma:3.2.1} and Lemma~\ref{Lemma:appro}, there are $2$ strongly exceptional Legendrian $A_3$ links in $(S^3, \xi_{-\frac{3}{2}})$  whose rotation numbers $(r_0, r_1, r_2)$ are  $(\pm5, \pm(t_{1}+1), \pm3).$
The decorations of their exteriors are $\pm(+)((-)(-))(--)$.

There are $2$ strongly exceptional Legendrian $A_3$ links in $(S^3, \xi_{\frac{1}{2}})$  whose rotation numbers $(r_0, r_1, r_2)$ are
$(\pm3, \pm(3-t_{1}), \pm3).$
The decorations of their exteriors are $\pm(+)((-)(+))(--)$. These exteriors are appropriate tight since they can be embedded into an appropriate tight contact $\Sigma\times S^1$ with boundary slopes $2, -\frac{1}{2}, -\frac{1}{2}$ and decorations $\pm(+)(-+)(--)$ by adding  basic slices $(T^2\times[0,1], -\frac{1}{t_{1}}, -\frac{1}{t_{1}-1})$, $\cdots$, $(T^2\times[0,1], -\frac{1}{3}, -\frac{1}{2})$ to the boundary $T_1$. 
\end{proof}
So there are exactly $16$ strongly exceptional Legendrian $A_3$ links with $t_{0}=2, t_{1}\geq3, t_{2}=2$. As a corollary, the $16$ contact structures on $\Sigma\times S^1$ with boundary slopes $s_{0}=3, s_{1}=-\frac{1}{t_1}, s_{2}=-\frac{1}{2}$ listed in Lemma~\ref{Lemma:t_1>1t_2>1}  are all appropriate tight.

\begin{figure}[htb]
\begin{overpic}
[width=1 \textwidth]
{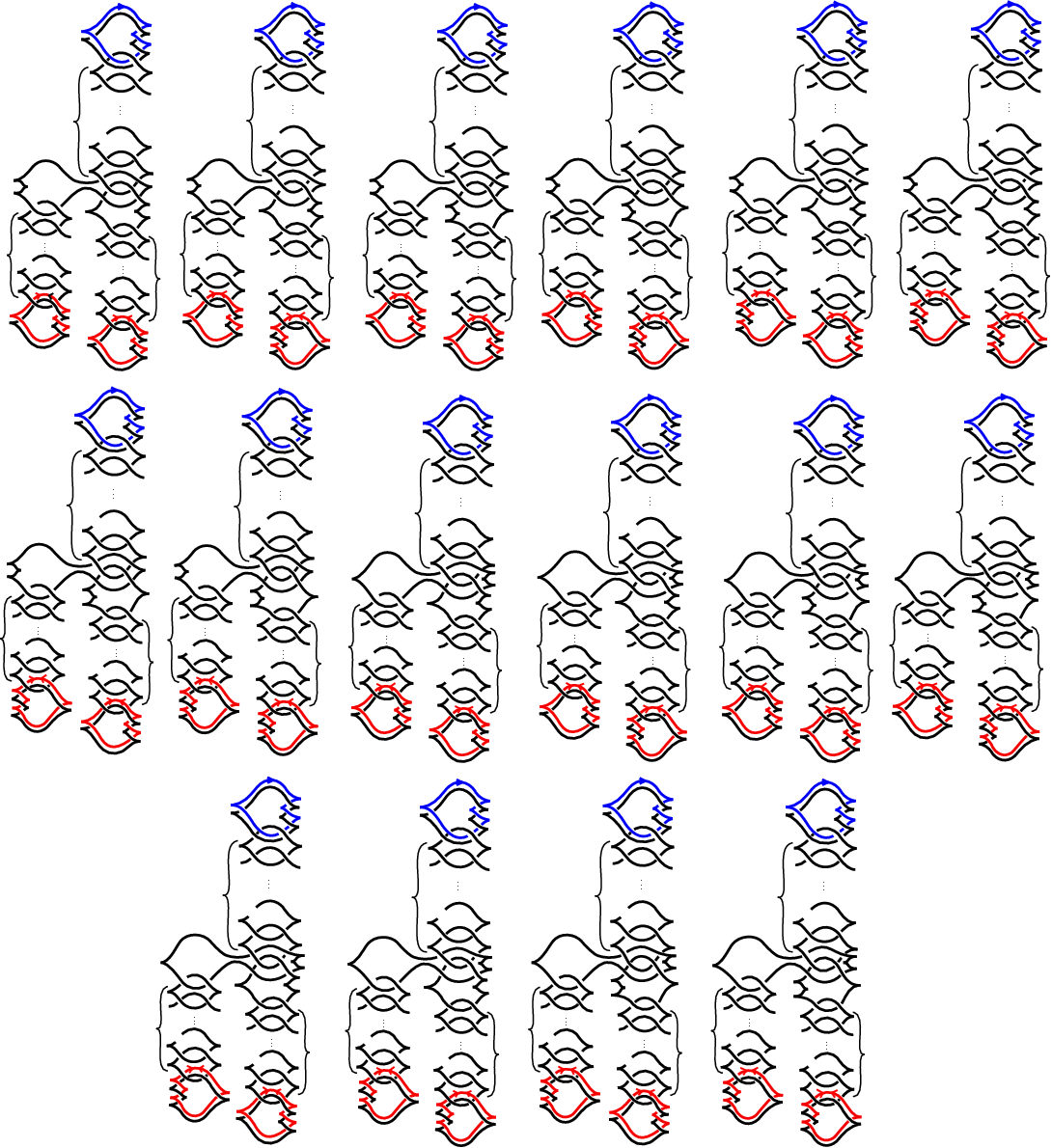}
\put(-1, 435){$t_{0}-4$}
\put(-30, 378){$t_{2}-3$}
\put(63, 435){$t_{1}-3$}
\put(70, 370){\vector(0, 1){62}} 
\put(-11, 353){$K_2$}
\put(64, 342){$K_1$}
\put(28, 483){$K_0$}
\end{overpic}
\caption{$t_0\geq4$, $t_{1}\geq 3$, $t_{2}\geq 3$. If $t_{0}+t_{i}$ is odd, then $K_{0}$ and $K_{i}$ bear the same orientation, $i=1,2$, if $t_{0}+t_{i}$ is even, then the opposite one.} 
\label{Figure:link30inot}
\end{figure}

(7) Suppose $t_0\geq4$, $t_1\geq3$ and $t_2\geq3$.

\begin{lemma}\label{t0>3t1>2t2>2}
If $t_0\geq4$, $t_1\geq3$ and $t_2\geq3$, then there exist $32$ strongly exceptional Legendrian $A_3$ links whose rotation numbers and corresponding $d_3$-invariants $(r_0, r_1, r_2; d_3)$ are 
$$(\pm(t_{0}+1), \pm(t_{1}-1), \pm(t_{2}+1);\frac{1}{2}), (\pm(t_{0}-1), \pm(1-t_{1}), \pm(t_{2}+1); \frac{1}{2}),$$
$$(\pm(t_{0}+3), \pm(t_{1}+1), \pm(t_{2}+1); -\frac{3}{2}), (\pm(t_{0}+1), \pm(3-t_{1}), \pm(t_{2}+1); \frac{1}{2}),$$
$$(\pm(t_{0}-1), \pm(t_{1}-1), \pm(3-t_{2}); \frac{5}{2}), (\pm(t_{0}-3), \pm(1-t_{1}), \pm(3-t_{2}); \frac{5}{2}),$$
$$(\pm(t_{0}+1), \pm(t_{1}+1), \pm(3-t_{2}); \frac{1}{2}), (\pm(t_{0}-1), \pm(3-t_{1}), \pm(3-t_{2}); \frac{5}{2}),$$
$$(\pm(t_{0}-1), \pm(t_{1}-1), \pm(t_{2}-1);  \frac{5}{2}), (\pm(t_{0}-3), \pm(1-t_{1}), \pm(t_{2}-1);  \frac{5}{2}),$$
$$(\pm(t_{0}+1), \pm(t_{1}+1), \pm(t_{2}-1);  \frac{1}{2}), (\pm(t_{0}-1), \pm(3-t_{1}), \pm(t_{2}-1);  \frac{5}{2}),$$
$$(\pm(t_{0}-3), \pm(t_{1}-1), \pm(1-t_{2});  \frac{5}{2}),(\pm(t_{0}-5), \pm(1-t_{1}), \pm(1-t_{2});  \frac{5}{2}),$$
$$(\pm(t_{0}-1), \pm(t_{1}+1), \pm(1-t_{2});  \frac{1}{2}), (\pm(t_{0}-3), \pm(3-t_{1}), \pm(1-t_{2});  \frac{5}{2}).$$
\end{lemma}

\begin{proof}
If $t_0\geq4$, $t_1\geq3$ and $t_2\geq3$, then there are exactly  $32$ strongly exceptional Legendrian $A_3$ links shown in Figure~\ref{Figure:link30inot}. Using the trick of Lemma~\ref{Lemma:topological} and the proof of \cite[Theorem 1.2, (c4)]{go}, we can show that $K_0\cup K_1\cup K_2$ is a topological $A_3$ link. Their rotation numbers and corresponding $d_3$-invariants are as listed.
\end{proof}

(8) Suppose $t_0=3$, $t_1\geq3$ and $t_2\geq3$. 

\begin{lemma}\label{t0=3t1>2t2>2}
If $t_0=3$, $t_1\geq3$ and $t_2\geq3$, then there exist  $28$ strongly exceptional Legendrian $A_3$ links whose rotation numbers and corresponding $d_3$-invariants $(r_0, r_1, r_2; d_3)$ are 
$$(\pm4, \pm(t_{1}+1), \pm(t_{2}-1); \frac{1}{2}), (\pm4, \pm(t_{1}-1), \pm(t_{2}+1); \frac{1}{2}),$$
$$(\pm2, \pm(t_{1}+1), \pm(1-t_{2}); \frac{1}{2}), (\pm2, \pm(1-t_{1}), \pm(t_{2}+1); \frac{1}{2}),$$
$$(\mp2, \pm(1-t_{1}), \pm(1-t_{2}); \frac{5}{2}), (0, \pm(t_{1}-1), \pm(1-t_{2}); \frac{5}{2}),$$
$$(0, \pm(3-t_{1}), \pm(1-t_{2}); \frac{5}{2}), (0, \pm(1-t_{1}), \pm(3-t_{2}); \frac{5}{2}),$$
$$(\pm2, \pm(3-t_{1}), \pm(t_{2}-1); \frac{5}{2}), (\pm2, \pm(t_{1}-1), \pm(3-t_{2}); \frac{5}{2}),$$
$$(\pm6, \pm(t_{1}+1), \pm(t_{2}+1); -\frac{3}{2}), (\pm2, \pm(3-t_{1}), \pm(3-t_{2}); \frac{5}{2}),$$
$$(\pm4, \pm(t_{1}+1), \pm(3-t_{2}); \frac{1}{2}), (\pm4, \pm(3-t_{1}), \pm(t_{2}+1); \frac{1}{2}).$$
\end{lemma}
\begin{proof}

By \cite[Theorem 1.2, (c3), (d)]{go}, there are two Legendrian Hopf links $K'_0\cup K_1$ in  $(S^3, \xi_{-\frac{1}{2}})$ with $(t'_0, r'_0)=(2,\pm3), t_1\geq3, r_{1}=\pm(t_{1}+1)$, two Legendrian Hopf links $K'_0\cup K_1$ in  $(S^3, \xi_{\frac{3}{2}})$ with $(t'_0, r'_0)=(2,\pm1), t_1\geq3, r_{1}=\pm(t_{1}-1)$, two Legendrian Hopf links $K'_0\cup K_1$ in  $(S^3, \xi_{\frac{3}{2}})$ with $(t'_0, r'_0)=(2,\mp1), t_1\geq3, r_{1}=\pm(t_{1}-3)$,  and two Legendrian Hopf links $K''_0\cup K_2$ in  $(S^3, \xi_{\frac{1}{2}})$ with $(t''_0, r''_0)=(0, \pm1), t_2\geq3, r_{2}=\pm(t_{2}-1)$. By Lemma~\ref{Lemma:connected}, we can obtain strongly exceptional Legendrian $A_3$ links with $t_0=3, t_1\geq3, t_2\geq3$. Then, after exchanging the roles of $K_1$ and $K_2$, there are $20$ strongly exceptional Legendrian $A_3$ links  whose rotation numbers and corresponding $d_3$-invariants $(r_0, r_1, r_2; d_3)$  are
$$(\pm4, \pm(t_{1}+1), \pm(t_{2}-1); \frac{1}{2}), (\pm4, \pm(t_{1}-1), \pm(t_{2}+1); \frac{1}{2}),$$
$$(\pm2, \pm(t_{1}+1), \pm(1-t_{2}); \frac{1}{2}), (\pm2, \pm(1-t_{1}), \pm(t_{2}+1); \frac{1}{2}),$$
$$(\mp2, \pm(1-t_{1}), \pm(1-t_{2}); \frac{5}{2}), (0, \pm(t_{1}-1), \pm(1-t_{2}); \frac{5}{2}),$$
$$(0, \pm(3-t_{1}), \pm(1-t_{2}); \frac{5}{2}), (0, \pm(1-t_{1}), \pm(3-t_{2}); \frac{5}{2}),$$
$$(\pm2, \pm(3-t_{1}), \pm(t_{2}-1); \frac{5}{2}), (\pm2, \pm(t_{1}-1), \pm(3-t_{2}); \frac{5}{2}).$$


By Lemma~\ref{Lemma:3.2.1} and Lemma~\ref{Lemma:appro}, there are $2$ strongly exceptional Legendrian $A_3$ links in $(S^3, \xi_{-\frac{3}{2}})$  whose rotation numbers $(r_0, r_1, r_2)$  are $(\pm6, \pm(t_{1}+1), \pm(t_{2}+1)).$
The decorations of their exteriors are $\pm(+)((-)(-))((-)(-))$.

There are other $6$ more strongly exceptional Legendrian $A_3$ links  whose rotation numbers and corresponding $d_3$-invariants $(r_0, r_1, r_2; d_3)$  are
$(\pm2, \pm(3-t_{1}), \pm(3-t_{2}); \frac{5}{2}),$  $(\pm4, \pm(t_{1}+1), \pm(3-t_{2}); \frac{1}{2})$
and $(\pm4, \pm(3-t_{1}), \pm(t_{2}+1); \frac{1}{2}).$
The decorations of their exteriors are  $$\pm(+)((-)(+))((-)(+)), \pm(+)((-)(-))((-)(+))~ \text{and}~ \pm(+)((-)(+))((-)(-)),$$ respectively. These exteriors are tight since they can be embedded into a tight contact $\Sigma\times S^1$ with boundary slopes $3, -\frac{1}{t_1}, -\frac{1}{2}$ and decorations $$\pm(+)((-)(+))(-+), \pm(+)((-)(-))(-+)~ \text{and} ~\pm(+)((-)(+))(--)$$ by adding basic slices $(T^2\times[0,1], -\frac{1}{t_{2}}, -\frac{1}{t_{2}-1})$, $\cdots$, $(T^2\times[0,1], -\frac{1}{3}, -\frac{1}{2})$ to the boundary $T_2$, respectively. 
\end{proof}

(9) Suppose $t_0=2$, $t_1\geq3$ and $t_2\geq3$.

\begin{lemma}\label{t0=2t1>2t2>2}
If $t_0=2$, $t_1\geq3$ and $t_2\geq3$, then there exist $24$ strongly exceptional Legendrian $A_3$ links whose rotation numbers and corresponding $d_3$-invariants $(r_0, r_1, r_2; d_3)$ are 
$$(\pm3, \pm(t_{1}+1), \pm(t_{2}-1); \frac{1}{2}), (\pm3, \pm(t_{1}-1), \pm(t_{2}+1); \frac{1}{2}),$$
$$(\pm1, \pm(t_{1}+1), \pm(1-t_{2}); \frac{1}{2}), (\pm1, \pm(1-t_{1}), \pm(t_{2}+1); \frac{1}{2}),$$
$$(\mp1, \pm(3-t_{1}), \pm(1-t_{2}); \frac{5}{2}), (\mp1, \pm(1-t_{1}), \pm(3-t_{2}); \frac{5}{2}),$$
$$(\pm1, \pm(3-t_{1}), \pm(t_{2}-1);  \frac{5}{2}), (\pm1, \pm(t_{1}-1), \pm(3-t_{2}); \frac{5}{2}),$$
$$(\pm5, \pm(t_{1}+1), \pm(t_{2}+1); -\frac{3}{2}), (\pm1, \pm(3-t_{1}), \pm(3-t_{2}); \frac{5}{2}),$$ $$ (\pm3, \pm(t_{1}+1), \pm(3-t_{2}); \frac{1}{2}), (\pm3, \pm(3-t_{1}), \pm(t_{2}+1); \frac{1}{2}).$$
\end{lemma}

\begin{proof}
By \cite[Theorem 1.2, (c2), (c3), (d)]{go}, there are two  Legendrian Hopf links $K'_0\cup K_1$ in $(S^3, \xi_{-\frac{1}{2}})$  with $(t'_0, r'_0)=(1, \pm2), t_1\geq3, r_{1}=\pm(t_{1}+1)$, two (one if $t_1=3$) Legendrian Hopf links $K'_0\cup K_1$ in  $(S^3, \xi_{\frac{3}{2}})$  with $(t'_0, r'_0)=(1, 0), t_1\geq3, r_{1}=\pm(t_{1}-3)$,  and two Legendrian Hopf links $K''_0\cup K_2$ in $(S^3, \xi_{\frac{1}{2}})$ with $(t''_0, r''_0)=(0, \pm1), t_2\geq3, r_{2}=\pm(t_{2}-1)$. By Lemma~\ref{Lemma:connected}, we can obtain strongly exceptional Legendrian $A_3$ links with $t_0=2, t_1\geq3, t_2\geq3$. So, after exchanging the roles of $K_1$ and $K_2$, there are $16$ strongly exceptional Legendrian $A_3$ links whose rotation numbers and corresponding $d_3$-invariants $(r_0, r_1, r_2; d_3)$ are
$$(\pm3, \pm(t_{1}+1), \pm(t_{2}-1); \frac{1}{2}), (\pm3, \pm(t_{1}-1), \pm(t_{2}+1); \frac{1}{2}),$$
$$(\pm1, \pm(t_{1}+1), \pm(1-t_{2}); \frac{1}{2}), (\pm1, \pm(1-t_{1}), \pm(t_{2}+1); \frac{1}{2}),$$
$$(\mp1, \pm(3-t_{1}), \pm(1-t_{2}); \frac{5}{2}), (\mp1, \pm(1-t_{1}), \pm(3-t_{2}); \frac{5}{2}),$$
$$(\pm1, \pm(3-t_{1}), \pm(t_{2}-1);  \frac{5}{2}), (\pm1, \pm(t_{1}-1), \pm(3-t_{2}); \frac{5}{2}).$$

By Lemma~\ref{Lemma:3.2.1} and Lemma~\ref{Lemma:appro}, there are $2$ strongly exceptional Legendrian $A_3$ links in $(S^3, \xi_{-\frac{3}{2}})$  whose rotation numbers $(r_0, r_1, r_2)$ are $(\pm5, \pm(t_{1}+1), \pm(t_{2}+1)).$
The decorations of their exteriors are $\pm(+)((-)(-))(--)$.

There are $6$ strongly exceptional Legendrian $A_3$ links whose rotation numbers and corresponding $d_3$-invariants  $(r_0, r_1, r_2; d_3)$  are
$(\pm1, \pm(3-t_{1}), \pm(3-t_{2}); \frac{5}{2}), $ $(\pm3, \pm(t_{1}+1), \pm(3-t_{2}); \frac{1}{2})$
and $(\pm3, \pm(3-t_{1}), \pm(t_{2}+1); \frac{1}{2}).$
The decorations of their exteriors are $$\pm(+)((-)(+))((-)(+)), \pm(+)((-)(-))((-)(+))~\text{ and} ~\pm(+)((-)(+))((-)(-)),$$ respectively. These exteriors are appropriate tight since they can be embedded into an appropriate  tight contact $\Sigma\times S^1$ with boundary slopes $2, -\frac{1}{t_1}, -\frac{1}{2}$ and decorations $$\pm(+)((-)(+))(-+), \pm(+)((-)(-))(-+)~ \text{and} ~\pm(+)((-)(+))(--)$$ by adding basic slices $(T^2\times[0,1], -\frac{1}{t_{2}}, -\frac{1}{t_{2}-1})$, $\cdots$, $(T^2\times[0,1], -\frac{1}{3}, -\frac{1}{2})$ to the boundary $T_2$, respectively. 
\end{proof}

(10) Suppose $t_{0}\leq1$. 

\begin{lemma}\label{t0<2t1>1t2>1}
If $t_{0}\leq1$, $t_1\geq2$ and $t_2\geq2$, then there exist $8-4t_0$ strongly exceptional Legendrian $A_3$ links in $(S^3, \xi_{\frac{3}{2}})$ whose rotation numbers are $$r_{0}\in\pm\{t_{0}+1, t_{0}+3, \cdots, -t_{0}+1, -t_{0}+3\}, r_{1}=\pm(t_{1}-1), r_{2}=\pm(t_{2}-1); $$
$$r_{0}\in\pm\{t_{0}-1, t_{0}+1,\cdots, -t_{0}-1, -t_{0}+1\}, r_{1}=\pm(1-t_{1}), r_{2}=\pm(t_{2}-1).$$
\end{lemma}
\begin{proof}

By \cite[Theorem 1.2, (d)]{go}, there are two Legendrian Hopf links $K'_0\cup K_1$  in $(S^3, \xi_{\frac{1}{2}})$ with $(t'_0, r'_0)=(0, \pm1), t_1\geq3, r_{1}=\pm(t_{1}-1)$. By \cite[Theorem 1.2. (b1)]{go}, there are $2(1-t''_0)$ Legendrian Hopf links $K_0\cup K_2$ in $(S^3, \xi_{\frac{1}{2}})$ with
$t''_0\leq0$, $r''_{0}\in\pm\{t''_{0}+1, t''_{0}+3, \cdots, -t''_{0}-1, -t''_{0}+1\}$, $t_2\geq2$, $r_{2}=\pm(t_{2}-1).$ Using Lemma~\ref{Lemma:connected}, we construct $8-4t_0$ Legendrian $A_3$ links in $(S^3, \xi_{3/2})$ with $t_0\leq1, t_1\geq2, t_{2}\geq2$. Their rotation numbers are as listed.
\end{proof}
These $8-4t_0$ strongly exceptional Legendrian $A_3$ links are stabilizations of the Legendrian $A_3$ links with $t_{0}=1, t_1\geq2, t_2\geq2$.

The proof of Theorem~\ref{Theorem:t_1>1t_2>1} is completed.
\end{proof}

\subsection{$t_{1}<0$ and $t_{2}>0$.}

\begin{lemma}\label{Lemma:3.3.1}
For any $t_0\in\mathbb{Z}$, there are $6$ exceptional Legendrian $A_3$ links whose exteriors have $0$-twisting vertical Legendrian circles, and the signs of basic slices in $L'_0, L'_1, L'_2$ are $\pm(+--), \pm(++-)$ and $\pm(+-+)$, respectively. Their rotation numbers are 
$$r_{0}=\pm(t_{0}+1), r_{1}=\pm(1-t_{1}), r_{2}=\pm(t_{2}+1);  r_{0}=\pm(t_{0}+1), r_{1}=\pm(t_{1}+1), r_{2}=\pm(t_{2}+1); $$ $$r_{0}=\pm(t_{0}-3), r_{1}=\pm(1-t_{1}), r_{2}=\pm(1-t_{2}).$$ 
The corresponding $d_3$-invariants are independent of $t_0$ if $t_1$ and $t_2$ are fixed.
\end{lemma}

\begin{proof}
The first statement follows from Lemma~\ref{Lemma:tight} and Lemma~\ref{Lemma:construction}. The calculation of rotation numbers is analogous to that in the proof of Lemma~\ref{Lemma:3.2.1}.
\end{proof}

\subsubsection{$t_1<0$ and $t_2=1$.} 
The boundary slopes of $\Sigma\times S^1$ are $s_0=t_0$, $s_1=-\frac{1}{t_1}$ and $s_2=-1$.

\begin{proof}[Proof of Theorem~\ref{Theorem:t_1<0t_2=1}] The upper bound of strongly exceptional Legendrian $A_3$ links is given by Lemma~\ref{Lemma:t_1<0t_2=1}. We will show that these upper bounds can be attained.

(1) Suppose $t_0\geq4$. 

\begin{lemma}\label{t0>3t1<0t2=1}
If $t_0\geq4$, $t_1<0$ and $t_2=1$, then there exist  $2-2t_1$ strongly exceptional Legendrian $A_3$ links in $(S^3, \xi_{-\frac{1}{2}})$ with rotation numbers $$r_{0}=\pm(t_{0}+1), r_{1}\in\mp\{t_{1}-1, t_{1}+1,\cdots,-t_{1}-1\}, r_{2}=\pm2; $$ and $2-2t_1$ strongly exceptional Legendrian $A_3$ links in $(S^3, \xi_{\frac{3}{2}})$ with rotation numbers $$r_{0}=\pm(t_{0}-3), r_{1}\in\mp\{t_{1}-1, t_{1}+1,\cdots,-t_{1}-1\}, r_{2}=0. $$
\end{lemma}

\begin{proof}
There exist $4-4t_1$ strongly exceptional Legendrian representatives shown in Figure~\ref{Figure:link15inot}. Using the trick of Lemma~\ref{Lemma:topological} and the proof of \cite[Theorem 1.2, (b1), (c3)]{go}, we can show that $K_0\cup K_1\cup K_2$ is a topological $A_3$ link.  Their rotation numbers and corresponding $d_3$-invariants are
$$r_{0}=\pm(t_{0}+1), r_{1}\in\mp\{t_{1}-1, t_{1}+1,\cdots,-t_{1}-1\}, r_{2}=\pm2; d_3=-\frac{1}{2}, $$ 
$$r_{0}=\pm(t_{0}-3), r_{1}\in\mp\{t_{1}-1, t_{1}+1,\cdots,-t_{1}-1\}, r_{2}=0; d_3=\frac{3}{2}. $$
\end{proof}

\begin{figure}[htb]
\begin{overpic}
{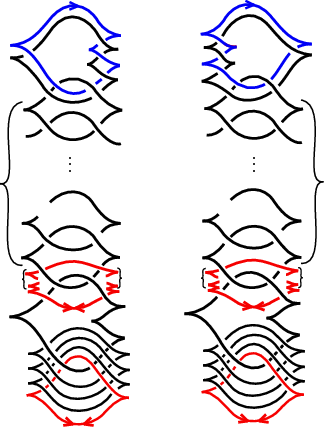}
\put(0, 190){$K_{0}$}
\put(145, 190){$K_{0}$}
\put(20, 50){$K_{1}$}
\put(106, 50){$K_{1}$}
\put(10, 0){$K_{2}$}
\put(96, 0){$K_{2}$}
\put(1, 68){$k_{1}$}
\put(60, 68){$l_{1}$}
\put(87, 68){$k_{1}$}
\put(147, 68){$l_{1}$}
\put(-32, 115){$t_{0}-4$}
\put(158, 115){$t_{0}-4$}
\put(60, 178){$+1$}
\put(60, 150){$-1$}
\put(60, 130){$-1$}
\put(60, 95){$-1$}
\put(60, 80){$-1$}
\put(60, 55){$-1$}
\put(63, 33){$+1$}
\put(63, 24){$+1$}
\put(63, 15){$+1$}

\end{overpic}
\caption{$t_0\geq4$, $t_1\leq 0$, $t_2=1$. $k_1+l_1=-t_1$. If $t_{0}$ is even, then $K_{0}$ and $K_{i}$, $i=1,2$, bear the same orientations, if $t_{0}$ is odd, then the opposite orientation.}
\label{Figure:link15inot}
\end{figure}
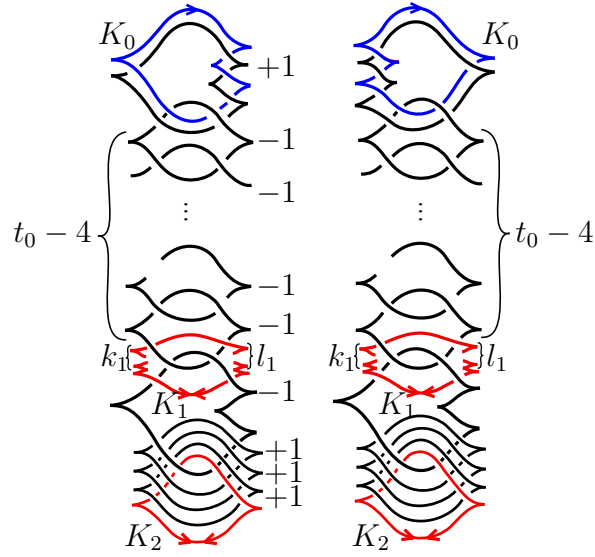

(2) Suppose $t_0=3$.

\begin{lemma}\label{t0=3t1<0t2=1}
If $t_0=3$, $t_1<0$ and $t_2=1$, then there exist $2-2t_1$ strongly exceptional Legendrian $A_3$ links in $(S^3, \xi_{-\frac{1}{2}})$ with rotation numbers $$r_{0}=\pm4, r_{1}\in\mp\{t_{1}-1, t_{1}+1,\cdots,-t_{1}-1\}, r_{2}=\pm2; $$ and $2-t_1$ strongly exceptional Legendrian $A_3$ links in $(S^3, \xi_{\frac{3}{2}})$ with rotation numbers $$r_{0}=0, r_{1}\in\mp\{t_{1}-1, t_{1}+1,\cdots,-t_{1}-1\}, r_{2}=0. $$
\end{lemma}

\begin{proof}
By \cite[Theorem 1.2]{go}, there are two Legendrian Hopf links $K_0\cup K_2$ with $(t_0,r_0)=(3,\pm4)$ and $(t_2, r_2)=(1, \pm2)$ in $(S^3, \xi_{-\frac{1}{2}})$, and a Legendrian Hopf links  with $(t_0,r_0)=(3, 0)$ and $(t_2, r_2)=(1, 0)$ in $(S^3, \xi_{\frac{3}{2}})$. Let $K_1$ be a local Legendrian meridian of $K_0$; then there are $-3t_1$ strongly exceptional Legendrian $A_3$ links. Their rotation numbers and corresponding $d_3$-invariants are
$$r_{0}=\pm4, r_{1}\in\{t_{1}+1, t_{1}+3, \cdots, -t_{1}-1\}, r_{2}=\pm2; d_3=-\frac{1}{2},$$ 
$$r_{0}=0, r_{1}\in\{t_{1}+1, t_{1}+3, \cdots, -t_{1}-1\}, r_{2}=0; d_3=\frac{3}{2}.$$
By Lemma~\ref{Lemma:3.3.1} and Lemma~\ref{Lemma:appro}, there are  $4$ strongly exceptional Legendrian $A_3$ links whose rotation numbers and corresponding $d_3$-invariants $(r_0, r_1, r_2; d_3)$ are $(\pm4, \pm(1-t_{1}), \pm2; -\frac{1}{2})$ and $(0, \pm(t_{1}-1), 0; \frac{3}{2}).$
\end{proof}


(3) Suppose $t_0=2$.

\begin{lemma}
If $t_0=2$, $t_1<0$ and $t_2=1$, then there exist $2-2t_1$ strongly exceptional Legendrian $A_3$ links in $(S^3, \xi_{-\frac{1}{2}})$ with rotation numbers $$r_{0}=\pm3, r_{1}\in\mp\{t_{1}-1, t_{1}+1,\cdots,-t_{1}-1\}, r_{2}=\pm2; $$ and $2$ strongly exceptional Legendrian $A_3$ links in $(S^3, \xi_{\frac{3}{2}})$ with rotation numbers $$r_{0}=\pm1, r_{1}=\pm(t_{1}-1), r_{2}=0. $$
\end{lemma}
\begin{proof}
If $t_0=2$, then, by \cite[Theorem 1.2. (c2)]{go}, there exist two Legendrian Hopf links $K_0\cup K_2$ in $(S^3, \xi_{-\frac{1}{2}})$ with $(t_0,r_0)=(2,\pm3)$ and $(t_2, r_2)=(1, \pm2)$. Let $K_1$ be a local Legendrian meridian of $K_0$, then  by Lemma~\ref{Lemma:meridian}  we can realize $-2t_1$ strongly exceptional Legendrian $A_3$ links in $(S^3, \xi_{-\frac{1}{2}})$ whose rotation numbers are $$r_{0}=\pm3, r_{1}\in\{t_{1}+1, t_{1}+3, \cdots, -t_{1}-1\}, r_{2}=\pm2.$$
By Lemma~\ref{Lemma:3.3.1} and Lemma~\ref{Lemma:appro}, there are $4$ strongly exceptional Legendrian $A_3$ links whose rotation numbers and corresponding $d_3$-invariants $(r_0, r_1, r_2; d_3)$ are $(\pm3, \pm(1-t_{1}), \pm2; -\frac{3}{2})$ and $ (\pm1, \pm(t_{1}-1), 0; \frac{1}{2}).$
\end{proof}

(4) Suppose $t_0\leq1$. 

\begin{lemma}\label{t0<2t1<0t2=1}
If $t_0\leq1$, $t_1<0$ and $t_2=1$, then there exist $t_{0}t_{1}-2t_1$ strongly exceptional Legendrian $A_3$ links in $(S^3, \xi_{\frac{1}{2}})$ with rotation numbers $$r_{0}\in\{t_{0}-1, t_{0}+1, \cdots, 1-t_0\}, r_{1}\in\{t_{1}+1, t_{1}+3, \cdots,-t_{1}-1\}, r_{2}=0.$$ 
\end{lemma}
\begin{proof}

By \cite[Theorem 1.2. (b2), (e)]{go}, there are $2-t_0$  strongly exceptional Legendrian Hopf links $K_0\cup K_2$ in $(S^3, \xi_{1/2})$ with $$r_{0}\in\{t_{0}-1, t_{0}+1,\cdots, 1-t_{0}\}, t_2=1, r_{2}=0.$$ Let $K_1$ be a local Legendrian meridian of $K_0$. Then  by Lemma~\ref{Lemma:meridian} there are
$(2-t_0)(-t_1)=t_{0}t_{1}-2t_{1}$ strongly exceptional Legendrian $A_3$ links  in $(S^3, \xi_{1/2})$
with rotation numbers are as listed.
\end{proof}
These $t_0t_1-2t_1$ strongly exceptional Legendrian $A_3$ links are stabilizations of the Legendrian $A_3$ links with $t_{0}=1, t_1=-1, t_2=1$.

The proof of Theorem~\ref{Theorem:t_1<0t_2=1} is completed.
\end{proof}

\subsubsection{$t_1<0$ and $t_2\geq2$.}

The boundary slopes of $\Sigma\times S^1$ are $s_0=t_0$, $s_1=-\frac{1}{t_1}$ and $s_2=-\frac{1}{t_2}$.

\begin{proof}[Proof of Theorem~\ref{Theorem:t_1<0t_2>1}]
The upper bound of strongly exceptional Legendrian $A_3$ links is given by Lemma~\ref{Lemma:t_1<0t_2>1}. We will show that the upper bounds can be attained except in the cases that $t_0=1$, $t_1<0$ and $t_2=3$.

(1) Suppose $t_0\geq3$ and $t_2=2$. 

\begin{lemma}
If $t_0\geq3$, $t_{1}<0$ and $t_2=2$, then there exist $6-6t_1$ strongly exceptional Legendrian $A_3$ links whose rotation numbers and corresponding $d_3$-invariants are $$r_{0}=\pm(t_{0}+1), r_{1}\in\pm\{t_{1}+1, \cdots,-t_{1}-1, -t_{1}+1\}, r_{2}=\pm3; d_3=-\frac{1}{2},$$
$$r_{0}=\pm(t_{0}-1), r_{1}\in\pm\{t_{1}+1, \cdots,-t_{1}-1, -t_{1}+1\}, r_{2}=\pm1; d_3=\frac{3}{2},$$ 
$$r_{0}=\pm(t_{0}-3), r_{1}\in\pm\{t_{1}+1, \cdots,-t_{1}-1, -t_{1}+1\}, r_{2}=\mp1; d_3=\frac{3}{2}.$$ 
\end{lemma}
\begin{proof}
There are  $6-6t_1$ strongly exceptional Legendrian $A_3$ links shown in Figure~\ref{Figure:link17inot}. Using the trick of Lemma~\ref{Lemma:topological} and the proof of \cite[Theorem 1.2, (b1), (c3)]{go}, we can show that $K_0\cup K_1\cup K_2$ is a topological $A_3$ link. Their rotation numbers and corresponding $d_3$-invariants are as listed.
\end{proof}
\begin{figure}[htb]
\begin{overpic}
{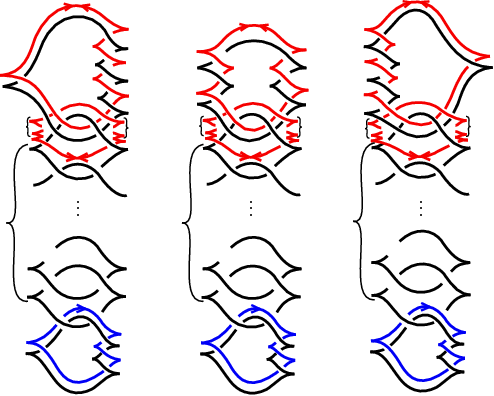}
\put(0, 29){$K_{0}$}
\put(86, 29){$K_{0}$}
\put(166, 29){$K_{0}$}
\put(30, 143){$K_{1}$}
\put(115, 143){$K_{1}$}
\put(197, 143){$K_{1}$}
\put(-2, 165){$K_{2}$}
\put(90, 172){$K_{2}$}
\put(170, 183){$K_{2}$}
\put(2, 128){$k_{1}$}
\put(63, 128){$l_{1}$}
\put(86, 128){$k_{1}$}
\put(145, 128){$l_{1}$}
\put(166, 128){$k_{1}$}
\put(228, 128){$l_{1}$}
\put(-27, 80){$t_{0}-3$}
\put(56, 80){$t_{0}-3$}
\put(140, 80){$t_{0}-3$}
\put(64, 165){$+1$}
\put(64, 112){$-1$}
\put(64, 92){$-1$}
\put(64, 57){$-1$}
\put(64, 45){$-1$}
\put(64, 20){$+1$}
\end{overpic}
\caption{$t_0\geq3$, $t_1\leq 0$, $t_2=2$. $k_1+l_1=-t_1$. If $t_{0}$ is even, then $K_{0}$ and $K_{2}$ bear the same orientation, if $t_{0}$ is odd, then the opposite one.  If $t_{0}$ is odd, then $K_{0}$ and $K_{1}$ bear the same orientation, if $t_{0}$ is even, then the opposite one.}
\label{Figure:link17inot}
\end{figure}
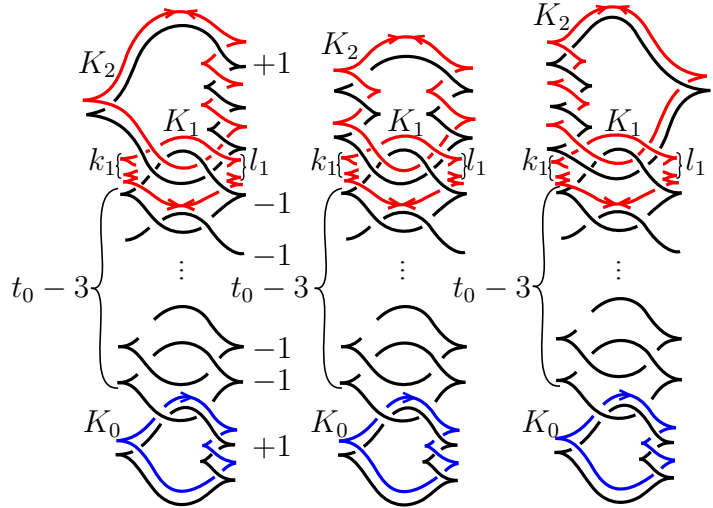

(2) Suppose $t_0=t_2=2$.

\begin{lemma}
If $t_0=t_2=2$ and $t_1<0$, then there exist $6-4t_1$ strongly exceptional Legendrian $A_3$ links whose rotation numbers and corresponding $d_3$-invariants are $$r_{0}=\pm3, r_{1}\in\pm\{t_{1}+1, \cdots,-t_{1}-1, -t_{1}+1\}, r_{2}=\pm3; d_3=-\frac{1}{2},$$
$$r_{0}=\pm1, r_{1}=\pm(1-t_{1}), r_{2}=\pm1; d_3=\frac{3}{2},$$ 
$$r_{0}=\mp1, r_{1}\in\pm\{t_{1}+1, \cdots,-t_{1}-1, -t_{1}+1\}, r_{2}=\mp1; d_3=\frac{3}{2}.$$    
\end{lemma}

\begin{proof}
If $t_0=t_2=2$, then by \cite[Theorem 1.2, (c2)]{go}, there are two strongly exceptional Legendrian Hopf links $K_0\cup K_2$ in $(S^3, \xi_{-\frac{1}{2}})$ with $(t_0, r_0)=(2, \pm3)$ and $(t_2, r_2)=(2, \pm3)$, and two strongly exceptional Legendrian Hopf links $K_0\cup K_2$ in $(S^3, \xi_{\frac{3}{2}})$ with $(t_0, r_0)=(2, \pm1)$ and $(t_2, r_2)=(2, \pm1)$. Let $K_1$ be a local Legendrian meridian of $K_0$, then  by Lemma~\ref{Lemma:meridian} there are $-4t_1$ strongly exceptional Legendrian $A_3$ links. Their rotation numbers and corresponding $d_3$-invariants are
$$r_{0}=\pm3, r_{1}\in\{t_{1}+1, t_{1}+3,\cdots,-t_{1}-1\}, r_{2}=\pm3; d_3=-\frac{1}{2},$$ 
$$r_{0}=\pm1, r_{1}\in\{t_{1}+1, t_{1}+3,\cdots,-t_{1}-1\}, r_{2}=\pm1; d_3=\frac{3}{2}. $$
By Lemma~\ref{Lemma:3.3.1} and Lemma~\ref{Lemma:appro}, there are  $4$ strongly exceptional Legendrian $A_3$ links whose rotation numbers and corresponding $d_3$-invariants $(r_0, r_1, r_2; d_3)$  are $(\pm3, \pm(1-t_{1}), \pm3; -\frac{1}{2})$ and $(\mp1, \pm(1-t_{1}), \mp1; \frac{3}{2}).$
The decorations of their exteriors are  $$\pm(+)(\underbrace{- \cdots -}_{-t_{1}})(--)~\text{and}~\pm(+)(\underbrace{- \cdots -}_{-t_{1}})(++),$$ respectively.

By \cite[Theorem 1.2, (b2), (d)]{go}, there are two Legendrian Hopf links in $(S^3, \xi_{\frac{1}{2}})$ with $(t'_0, r'_0)=(1,0), t_{1}<0, r_{1}=\mp(t_{1}-1)$, there are two Legendrian Hopf links in $(S^3, \xi_{\frac{1}{2}})$ with $(t'_0, r'_0)=(0,\pm1), (t_2, r_2)=(2,\pm1)$. By Lemma~\ref{Lemma:connected}, we can construct  $2$  strongly exceptional Legendrian $A_3$ in $(S^3, \xi_{\frac{3}{2}})$ links with $t_0=t_2=2, t_1<0$. Their rotation numbers $(r_0, r_1, r_2)$  are 
$(\pm1, \pm(1-t_{1}), \pm1).$ 
The decorations of their exteriors are  $$\pm(+)(\underbrace{- \cdots -}_{-t_{1}})(+-).$$
\end{proof}

So there are $6-4t_1$ strongly exceptional Legendrian $A_3$ links with $t_{0}=2, t_{1}<0, t_{2}=2$. As a corollary, the $6-4t_1$ contact structures on $\Sigma\times S^1$ with boundary slopes $s_{0}=2, s_{1}=-\frac{1}{t_1}, s_{2}=-\frac{1}{2}$  listed in Lemma~\ref{Lemma:t_1<0t_2>1} are all appropriate tight.

(3) Suppose $t_0=1$ and $t_2=2$.

\begin{lemma}
If $t_0=1$, $t_1<0$ and $t_2=2$, there exist $6-2t_1$ strongly exceptional Legendrian $A_3$ links whose rotation numbers and corresponding $d_3$-invariants are $$r_{0}=\pm2, r_{1}\in\pm\{t_{1}+1, \cdots,-t_{1}-1, -t_{1}+1\}, r_{2}=\pm3; d_3=-\frac{1}{2},$$ 
$$r_{0}=\mp2, r_{1}=\pm(1-t_{1}), r_{2}=\mp1; d_3=\frac{3}{2},$$
$$r_{0}=0, r_{1}=\pm(1-t_{1}), r_{2}=\pm1; d_3=\frac{3}{2}.$$
\end{lemma}

\begin{proof}

If $t_0=1$ and $t_2=2$, then, by \cite[Theorem 1.2]{go}, there are two strongly exceptional Legendrian Hopf links $K_{0}\cup K_{2}$ with $(t_0, r_0)=(1, \pm2)$ and $(t_2,r_2)=(2,\pm3)$ in $(S^3, \xi_{-\frac{1}{2}})$. Let $K_1$ be a local Legendrian meridian of $K_0$, then by Lemma~\ref{Lemma:meridian}  we can realize $-2t_1$ strongly exceptional Legendrian $A_3$ links in $(S^3, \xi_{-\frac{1}{2}})$ whose rotation numbers are $$r_{0}=\pm2, r_{1}\in\{t_{1}+1, t_{1}+3, \cdots, -t_{1}-1\}, r_{2}=\pm3.$$

By Lemma~\ref{Lemma:3.3.1} and Lemma~\ref{Lemma:appro}, there are $4$ strongly exceptional Legendrian $A_3$ links whose rotation numbers and corresponding $d_3$-invariants $(r_0, r_1, r_2; d_3)$ are $(\pm2, \pm(1-t_{1}), \pm3; -\frac{1}{2}) $ and $ (\mp2, \pm(1-t_{1}), \mp1; \frac{3}{2}).$ The decorations of their exteriors are  $$\pm(+)(\underbrace{- \cdots -}_{-t_{1}})(--)~\text{and}~\pm(+)(\underbrace{- \cdots -}_{-t_{1}})(++),$$ respectively.

By \cite[Theorem 1.2, (d)]{go}, there are two Legendrian Hopf links in $(S^3, \xi_{\frac{1}{2}})$ with $(t'_0, r'_0)=(0,\mp1), t_{1}<0, r_{1}=\mp(t_{1}-1)$, there are two Legendrian Hopf links in $(S^3, \xi_{\frac{1}{2}})$ with $(t'_0, r'_0)=(0,\pm1), (t_2, r_2)=(2,\pm1)$. By Lemma~\ref{Lemma:connected}, we can construct  $2$  strongly exceptional Legendrian $A_3$ links in $(S^3, \xi_{\frac{3}{2}})$ with $t_0=1, t_1<0, t_2=2$. Their rotation numbers $(r_0, r_1, r_2)$ are  
$(0, \pm(1-t_{1}), \pm1).$
The decorations of their exteriors are  $$\pm(+)(\underbrace{- \cdots -}_{-t_{1}})(+-).$$
\end{proof}

So there are $6-2t_1$ strongly exceptional Legendrian $A_3$ links with $t_{0}=1, t_{1}<0, t_{2}=2$. As a corollary, the $6-2t_1$ contact structures on $\Sigma\times S^1$ with boundary slopes $s_{0}=1, s_{1}=-\frac{1}{t_1}, s_{2}=-\frac{1}{2}$  listed in Lemma~\ref{Lemma:t_1<0t_2>1} are all appropriate tight.

(4) Suppose $t_0\geq3$ and $t_2\geq3$.

\begin{lemma}
If $t_0\geq3$, $t_1<0$ and $t_2\geq3$, then there are $8-8t_1$ strongly exceptional Legendrian $A_3$ links whose rotation numbers and corresponding $d_3$-invariants are 
$$r_{0}=\pm(t_{0}+1), r_{1}\in\pm\{t_{1}+1, t_{1}+3,\cdots,-t_{1}+1\}, r_{2}=\pm(t_{2}+1); d_3=-\frac{1}{2},$$
$$r_{0}=\pm(t_{0}-1), r_{1}\in\pm\{t_{1}+1, t_{1}+3,\cdots,-t_{1}+1\}, r_{2}=\pm(t_{2}-1); d_3=\frac{3}{2},$$
$$r_{0}=\pm(t_{0}-3), r_{1}\in\pm\{t_{1}+1, t_{1}+3,\cdots,-t_{1}+1\}, r_{2}=\pm(1-t_{2}); d_3=\frac{3}{2},$$ 
$$r_{0}=\pm(t_{0}-1), r_{1}\in\pm\{t_{1}+1, t_{1}+3,\cdots,-t_{1}+1\}, r_{2}=\pm(3-t_{2}); d_3=\frac{3}{2}.$$
\end{lemma}

\begin{proof}
If $t_0\geq3$ and $t_2\geq3$, then there are $8-8t_1$ strongly exceptional Legendrian $A_3$ links shown in Figure~\ref{Figure:link16inot}. Using the trick of Lemma~\ref{Lemma:topological} and the proof of \cite[Theorem 1.2, (b1), (c4)]{go}, we can show that $K_0\cup K_1\cup K_2$ is a topological $A_3$ link. Their rotation numbers are as listed.
\end{proof}

\begin{figure}[htb]
\begin{overpic}
{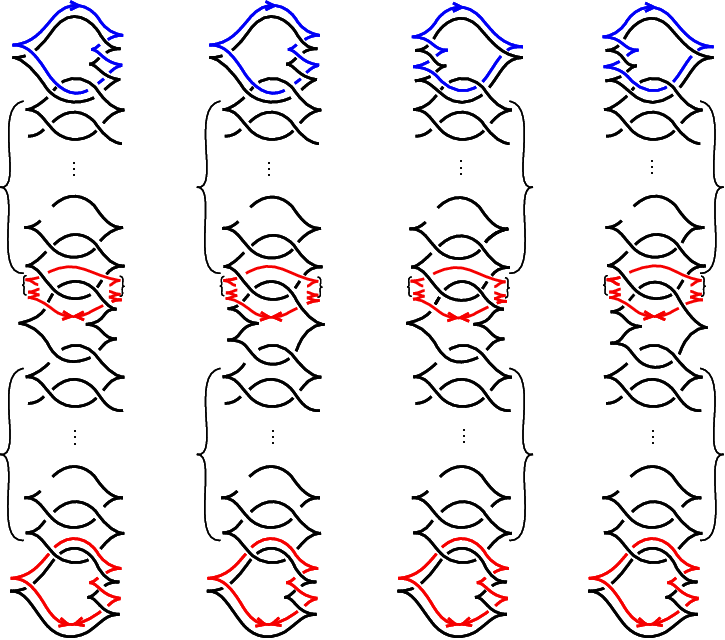}
\put(0, 292){$K_{0}$}
\put(95, 292){$K_{0}$}
\put(245, 292){$K_{0}$}
\put(335, 292){$K_{0}$}
\put(26, 145){$K_{1}$}
\put(128, 145){$K_{1}$}
\put(215, 145){$K_{1}$}
\put(313, 145){$K_{1}$}
\put(26, 14){$K_{2}$}
\put(122, 14){$K_{2}$}
\put(215, 14){$K_{2}$}
\put(305, 14){$K_{2}$}
\put(1, 166){$k_{1}$}
\put(61, 166){$l_{1}$}
\put(96, 166){$k_{1}$}
\put(155, 166){$l_{1}$}
\put(187, 166){$k_{1}$}
\put(247, 166){$l_{1}$}
\put(280, 166){$k_{1}$}
\put(341, 166){$l_{1}$}
\put(60, 235){$-1$}
\put(60, 250){$-1$}
\put(60, 280){$+1$}
\put(-30, 85){$t_{2}-3$}
\put(-30, 215){$t_{0}-3$}
\put(350, 85){$t_{2}-3$}
\put(350, 215){$t_{0}-3$}
\put(65, 85){$t_{2}-3$}
\put(65, 215){$t_{0}-3$}
\put(260, 85){$t_{2}-3$}
\put(260, 215){$t_{0}-3$}
\put(60, 197){$-1$}
\put(60, 180){$-1$}
\put(60, 140){$-1$}
\put(60, 120){$-1$}
\put(60, 105){$-1$}
\put(60, 65){$-1$}
\put(60, 45){$-1$}
\put(60, 5){$+1$}

\end{overpic}
\caption{$t_0\geq3$, $t_1\leq 0$, $t_2\geq3$. $k_1+l_1=-t_1$. For $t_{0}+t_{2}$ even, $K_0$ and $K_2$ bear the same orientation, for $t_{0}+t_{2}$ odd, the opposite one. For $t_{0}$ odd, $K_0$ and $K_1$ bear the same orientation, for $t_{0}$ even, the opposite one. }
\label{Figure:link16inot}
\end{figure}
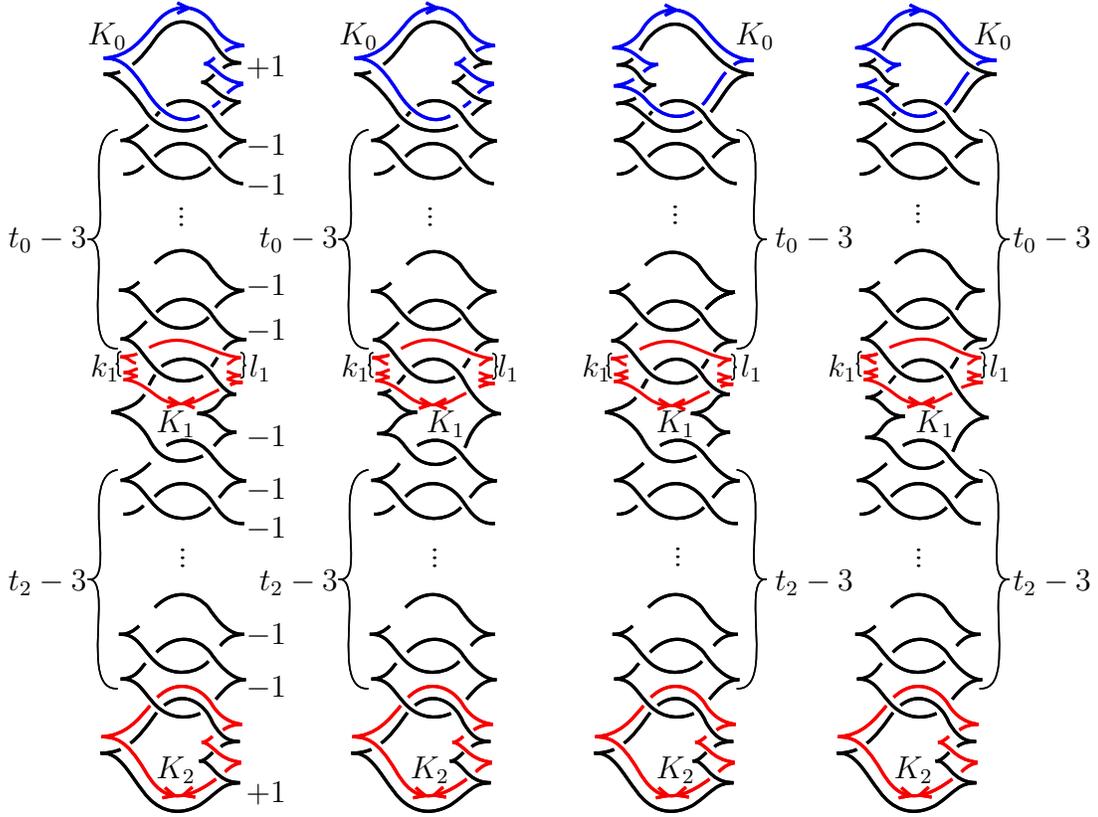

(5) Suppose $t_0=2$ and $t_2\geq3$.

\begin{lemma}
If $t_0=2$, $t_1<0$ and $t_2\geq3$, then there exist $8-6t_1$ strongly exceptional Legendrian $A_3$ links whose rotation numbers and corresponding $d_3$-invariants are 
$$r_{0}=\pm3, r_{1}\in\pm\{t_{1}+1, t_{1}+3,\cdots,-t_{1}+1\}, r_{2}=\pm(t_{2}+1); d_3=-\frac{1}{2},$$
$$r_{0}=\pm1, r_{1}\in\pm\{t_{1}+1, t_{1}+3,\cdots,-t_{1}+1\}, r_{2}=\pm(t_{2}-1); d_3=\frac{3}{2},$$
$$r_{0}=\mp1, r_{1}=\pm(1-t_{1}), r_{2}=\pm(1-t_{2}); d_3=\frac{3}{2},$$ 
$$r_{0}=\pm1, r_{1}\in\pm\{t_{1}+1, t_{1}+3,\cdots,-t_{1}+1\}, r_{2}=\pm(3-t_{2}); d_3=\frac{3}{2}.$$
\end{lemma}
\begin{proof}
If $t_0=2$ and $t_2\geq3$, then, by \cite[Theorem 1.2, (c3)]{go}, there are two Legendrian Hopf links $K'_0\cup K_2$ in  $(S^3, \xi_{-\frac{1}{2}})$ with $(t'_0, r'_0)=(2,\pm3), t_2\geq3, r_{2}=\pm(t_{2}+1)$, two Legendrian Hopf links $K'_0\cup K_2$ in  $(S^3, \xi_{\frac{3}{2}})$ with $(t'_0, r'_0)=(2,\pm1), t_2\geq3, r_{2}=\pm(t_{2}-1)$, two Legendrian Hopf links $K'_0\cup K_2$ in  $(S^3, \xi_{\frac{3}{2}})$ with $(t'_0, r'_0)=(2,\mp1), t_2\geq3, r_{2}=\pm(t_{2}-3)$.  Let $K_1$ be a local Legendrian meridian of $K_0$; then by Lemma~\ref{Lemma:meridian}, we can realize $-6t_1$ strongly exceptional Legendrian representatives. There are $-2t_1 $ of them belong to $(S^3, \xi_{-\frac{1}{2}})$ with rotation numbers 
$$r_{0}=\pm3, r_{1}\in\{t_{1}+1, t_{1}+3,\cdots,-t_{1}-1\}, r_{2}=\pm(t_{2}+1). $$
There are $-4t_1 $ of them belong to $(S^3, \xi_{\frac{3}{2}})$ with rotation numbers
$$r_{0}=\pm1, r_{1}\in\{t_{1}+1, t_{1}+3,\cdots,-t_{1}-1\}, r_{2}=\pm(t_{2}-1); $$ 
$$r_{0}=\mp1, r_{1}\in\{t_{1}+1, t_{1}+3,\cdots,-t_{1}-1\}, r_{2}=\pm(t_{2}-3). $$ 

By Lemma~\ref{Lemma:3.3.1} and Lemma~\ref{Lemma:appro}, there are  $4$ strongly exceptional Legendrian $A_3$ links whose rotation numbers and corresponding $d_3$-invariants $(r_0, r_1, r_2; d_3)$ are $(\pm3, \pm(1-t_{1}), \pm(t_{2}+1); -\frac{1}{2}) $ and $(\mp1, \pm(1-t_{1}), \pm(1-t_{2}); \frac{3}{2}).$
The decorations of their exteriors are  $$\pm(+)(\underbrace{- \cdots -}_{-t_{1}})((-)(-))~\text{and}~ \pm(+)(\underbrace{- \cdots -}_{-t_{1}})((+)(+)),$$ respectively.

There are $4$ strongly exceptional Legendrian $A_3$ links  whose rotation numbers and corresponding $d_3$-invariants  $(r_0, r_1, r_2; d_3)$  are 
$(\pm1, \pm(1-t_{1}), \pm(t_{2}-1); \frac{3}{2})  $ and  $(\pm1, \pm(1-t_{1}), \pm(3-t_{2}); \frac{3}{2}).$
The decorations of their exteriors are  $$\pm(+)(\underbrace{- \cdots -}_{-t_{1}})((-)(+))~ \text{and} ~\pm(+)(\underbrace{- \cdots -}_{-t_{1}})((+)(-)),$$ respectively. 
These exteriors are appropriate tight since they can be embedded into an appropriate tight contact $\Sigma\times S^1$ with boundary slopes $2, -\frac{1}{t_1}, -\frac{1}{2}$ and decorations $\pm(+)(\underbrace{- \cdots -}_{-t_{1}})(-+)$ by adding basic slices $(T^2\times[0,1], -\frac{1}{t_{2}}, -\frac{1}{t_{2}-1})$, $\cdots$, $(T^2\times[0,1], -\frac{1}{3}, -\frac{1}{2})$ to the boundary $T_2$, respectively. 
\end{proof}


(6)  Suppose $t_0=1$ and $t_2\geq3$.

\begin{lemma}
If $t_0=1$, $t_1<0$ and $t_2\geq4$ (resp. $t_2=3$), then there exist $8-4t_1$ (resp. $8-3t_1$) strongly exceptional Legendrian $A_3$ links whose rotation numbers and corresponding $d_3$-invariants are 
$$r_{0}=\pm2, r_{1}\in\{t_{1}+1, t_{1}+3, \cdots, -t_{1}-1\}\cup\{\pm(1-t_1)\}, r_{2}=\pm(t_{2}+1); d_3=-\frac{1}{2},$$
$$r_{0}=0, r_{1}\in\{t_{1}+1, t_{1}+3, \cdots, -t_{1}-1\}\cup\{\pm(t_{1}-1)\}, r_{2}=\pm(t_{2}-3); d_3=\frac{3}{2},$$
$$r_{0}=\mp2, r_{1}=\pm(1-t_{1}), r_{2}=\pm(1-t_{2}); d_3=\frac{3}{2},$$
$$r_{0}=0, r_{1}=\pm(1-t_{1}), r_{2}=\pm(t_{2}-1); d_3=\frac{3}{2}.$$
\end{lemma}
\begin{proof}
If $t_0=1$ and $t_2=3$, then, by \cite[Theorem 1.2, (c2)]{go}, there are two  strongly exceptional Legendrian Hopf links $K_0\cup K_2$ in $(S^3, \xi_{-\frac{1}{2}})$ with $(t_0, r_0)=(1, \pm2)$ and $(t_2,r_2)=(3,\pm4)$, and one strongly exceptional Legendrian Hopf links $K_0\cup K_2$ in $(S^3, \xi_{\frac{3}{2}})$ with $(t_0, r_0)=(1, 0)$ and $(t_2,r_2)=(3,0)$. Let $K_1$ be a local Legendrian meridian of $K_0$, then by Lemma~\ref{Lemma:meridian}  we can realize $-3t_1$ strongly exceptional Legendrian $A_3$ links  whose rotation numbers and corresponding $d_3$-invariants are 
$$r_{0}=\pm2, r_{1}\in\{t_{1}+1, t_{1}+3, \cdots, -t_{1}-1\}, r_{2}=\pm4; d_3=-\frac{1}{2},$$
$$r_{0}=0, r_{1}\in\{t_{1}+1, t_{1}+3, \cdots, -t_{1}-1\}, r_{2}=0; d_3=\frac{3}{2}.$$

If $t_0=1$ and $t_2\geq4$, then, by \cite[Theorem 1.2, (c3)]{go}, there are two strongly exceptional Legendrian Hopf links $K_0\cup K_2$ in $(S^3, \xi_{-\frac{1}{2}})$ with $(t_0, r_0)=(1, \pm2)$, $t_2\geq4$, and $r_2=\pm(t_{2}+1)$, and two strongly exceptional Legendrian Hopf links $K_0\cup K_2$ in $(S^3, \xi_{\frac{3}{2}})$ with $(t_0, r_0)=(1, 0)$, $t_2\geq4$ and $r_2 =\pm(t_{2}-3)$. Let $K_1$ be a local Legendrian meridian of $K_0$, then by Lemma~\ref{Lemma:meridian} we can realize $-4t_1$ strongly exceptional Legendrian $A_3$ links whose rotation numbers and corresponding $d_3$-invariants are 
$$r_{0}=\pm2, r_{1}\in\{t_{1}+1, t_{1}+3, \cdots, -t_{1}-1\}, r_{2}=\pm(t_{2}+1); d_3=-\frac{1}{2},$$
$$r_{0}=0, r_{1}\in\{t_{1}+1, t_{1}+3, \cdots, -t_{1}-1\}, r_{2}=\pm(t_{2}-3); d_3=\frac{3}{2}.$$

For any $t_2\geq3$, by Lemma~\ref{Lemma:3.3.1} and Lemma~\ref{Lemma:appro}, there are $4$ strongly exceptional Legendrian $A_3$ links whose rotation numbers and corresponding $d_3$-invariants $(r_0, r_1, r_2; d_3)$  
 are $(\pm2, \pm(1-t_{1}), \pm(t_{2}+1); -\frac{1}{2})$ and $ (\mp2, \pm(1-t_{1}), \pm(1-t_{2}); \frac{3}{2}).$
The decorations of their exteriors are  $$\pm(+)(\underbrace{- \cdots -}_{-t_{1}})((-)(-))~\text{and}~\pm(+)(\underbrace{- \cdots -}_{-t_{1}})((+)(+)),$$ respectively.

For any $t_2\geq3$, there are $4$ strongly exceptional Legendrian $A_3$ links whose rotation numbers and corresponding $d_3$-invariants $(r_0, r_1, r_2; d_3)$   are 
$(0, \pm(1-t_{1}), \pm(t_{2}-1); \frac{3}{2}) $  and $ (0, \pm(t_{1}-1), \pm(t_{2}-3); \frac{3}{2}).$
The decorations of their exteriors are  $$\pm(+)(\underbrace{- \cdots -}_{-t_{1}})((+)(-))~ \text{and} ~\pm(+)(\underbrace{- \cdots -}_{-t_{1}})((-)(+)),$$ respectively.
These exteriors are appropriate tight since they can be embedded into an appropriate  tight contact $\Sigma\times S^1$ with boundary slopes $1, -\frac{1}{t_1}, -\frac{1}{2}$ and decorations $\pm(+)(\underbrace{- \cdots -}_{-t_{1}})(+-)$ by adding basic slices $(T^2\times[0,1], -\frac{1}{t_{2}}, -\frac{1}{t_{2}-1})$, $\cdots$, $(T^2\times[0,1], -\frac{1}{3}, -\frac{1}{2})$ to the boundary $T_2$, respectively. 
\end{proof}

So, there are exactly $8-4t_1$ (resp. exactly $8-3t_1$) strongly exceptional Legendrian $A_3$ links with $t_0=1, t_1<0, t_2\geq4$ (resp. $t_2=3$). If $t_0=1$, $t_1<0$ and $t_2=3$, then the decorations $$(+)(\underbrace{+ \cdots +}_{l}\underbrace{- \cdots -}_{k})((-)(+)) ~\text{and}~ (-)(\underbrace{- \cdots -}_{k+1}\underbrace{+ \cdots +}_{l-1})((+)(-))$$ correspond to the same Legendrian $A_3$ links with rotation numbers $r_0=r_2=0, r_1=l-k-1$, where $k\geq0, l\geq1, k+l=-t_1$.

(7) Suppose $t_0\leq0$. 

\begin{lemma}\label{t0<1t1<0t2>1}
If $t_{0}\leq0$, $t_{1}<0$ and $t_{2}>1$, then there exist $2t_{0}t_{1}-2t_{1}$ strongly exceptional Legendrian $A_3$ links in $(S^3, \xi_{\frac{1}{2}})$ whose rotation numbers are $$r_{0}\in\pm\{t_{0}+1, t_{0}+3, \cdots, -t_{0}-1, -t_{0}+1\},$$ $$ r_{1}\in\{t_{1}+1, t_{1}+3, \cdots, -t_{1}-1\}, r_{2}=\pm(t_{2}-1).$$
\end{lemma}
\begin{proof}
By \cite[Theorem 1.2. (b1)]{go}, there are $2(1-t_0)$ Legendrian Hopf links $K_0\cup K_2$ in $(S^3, \xi_{1/2})$ whose rotation numbers are
$$r_{0}\in\pm\{t_{0}+1, t_{0}+3, \cdots, -t_{0}-1, -t_{0}+1\}, t_2\geq2, r_{2}=\pm(t_{2}-1).$$ Let $K_1$ be a local Legendrian meridian of $K_0$. Then by Lemma~\ref{Lemma:meridian} there are $2(1-t_0)(-t_1)=2t_{0}t_{1}-2t_{1}$ isotopy classes. Their rotation numbers are as listed.
\end{proof}
These $2t_{0}t_{1}-2t_{1}$ strongly exceptional Legendrian $A_3$ links are stabilizations of the Legendrian $A_3$ links with $t_{0}=0, t_1=-1, t_2>1$.

The proof of Theorem~\ref{Theorem:t_1<0t_2>1} is completed.
\end{proof}

\subsection{$t_{1}=0$.}  The boundary slopes of $\Sigma\times S^1$ are $s_0=t_0$, $s_1=\infty$ and $s_2=-\frac{1}{t_2}$. The appropriate tight contact structures on $\Sigma\times S^1$ can be decomposed as  $L'_0\cup L'_2 \cup \Sigma'\times S^1$. 

\begin{lemma}\label{Lemma:3.4.1}
For any $t_0\in\mathbb{Z}$, there are $4$  exceptional Legendrian $A_3$ links whose signs of basic slices in $L'_0,  L'_2$ are $\pm(+-)$ and $\pm(++)$, respectively. Their rotation numbers are $$r_0=\pm(t_{0}+1), r_{1}=\pm1, r_{2}=\pm(t_{2}+1); r_0=\pm(t_{0}-3), r_{1}=\pm1, r_{2}=\pm(1-t_{2}).$$
The corresponding $d_3$-invariants are independent of $t_0$ if  $t_2$ is fixed.
\end{lemma}

\begin{proof}
 The first statement follows from  Lemma~\ref{Lemma:tight1} and Lemma~\ref{Lemma:construction}.  Suppose the signs of the basic slices in  $L'_0$ and $L'_2$ are $+$ and $-$, respectively. Then $$r_{0}=-(\frac{1}{t_{2}}\ominus\frac{0}{1})\bullet \frac{0}{1}-(\frac{0}{1}\ominus\frac{-1}{0})\bullet\frac{0}{1}+(\frac{1}{0}\ominus\frac{t_0}{1})\bullet\frac{0}{1}=-(t_{0}+1),$$
$$r_{1}=(\frac{-t_{0}}{1}\ominus\frac{-1}{0})\bullet\frac{1}{0}=-1,$$
$$r_{2}=(\frac{-t_{0}}{1}\ominus\frac{-1}{0})\bullet\frac{1}{0}-(\frac{1}{0}\ominus\frac{0}{1})\bullet\frac{0}{1}-(\frac{0}{1}\ominus\frac{-1}{t_2})\bullet\frac{1}{0}=-(t_{2}+1).$$
The computation of other cases are similar.
\end{proof}

\begin{lemma}\label{Lemma:3.4.2}
Suppose $t_0\leq2$, $t_1=0$ and $t_2\geq2$. Then there are $4$ strongly exceptional Legendrian $A_3$ links  in $(S^3, \xi_{\frac{3}{2}})$ whose rotation numbers are 
$$r_0=\pm(t_{0}-1), r_{1}=\pm1, r_{2}=\pm(t_{2}-1); r_0=\pm(t_{0}-3), r_{1}=\pm1, r_{2}=\pm(1-t_{2}).$$   
\end{lemma}

\begin{proof}
By \cite[Theorem 1.2, (d)]{go}, there are two Legendrian Hopf links $K'_0\cup K_1$ in $(S^3, \xi_{\frac{1}{2}})$ with $t'_0 \leq1, r'_0=\pm(t'_0-1), (t_1, r_1)=(0, \pm1)$, and two Legendrian Hopf links $K''_0\cup K_2$  in $(S^3, \xi_{\frac{1}{2}})$ with $(t''_0, r''_0)=(0, \pm1), t_2\geq2, r_2=\pm(t_{2}-1)$. By Lemma~\ref{Lemma:connected}, we can obtain strongly exceptional Legendrian $A_3$ links with $t_0\leq2, t_1=0, t_2\geq2$. So there are $4$ strongly exceptional Legendrian $A_3$ links  in $(S^3, \xi_{\frac{3}{2}})$ whose rotation numbers are as listed. 
\end{proof}

\begin{proof}[Proof of Theorem~\ref{Theorem:t1=0}] The upper bound of strongly exceptional Legendrian $A_3$ links is given by Lemma~\ref{Lemma:t1eq0}. 
We will show that these upper bounds can be attained.

(1) Suppose $t_2\leq 0$.

\begin{lemma} \label{t1=0t2<1}
If $t_1=0$ and $t_2\leq 0$, then there exist $2-2t_{2}$ strongly exceptional Legendrian $A_3$ links in $(S^3, \xi_{\frac{1}{2}})$ whose rotation numbers are $$r_0=\pm(t_{0}-1),  r_{1}=\pm1, r_{2}\in\pm\{t_{2}+1, t_{2}+3, \cdots,-t_{2}+1\}.$$ 
\end{lemma}

\begin{proof}
If $t_2\leq 0$ and $t_0\leq0$, there exist $2(1-t_{2})$ strongly exceptional Legendrian $A_3$ links shown in Figure~\ref{Figure:link5inot}. Similar to the proof of \cite[Lemma 5.1, part (iii), Figure 6]{go}, we can show that the link $K_0\cup K_1\cup K_2$ in Figure~\ref{Figure:link5inot} is indeed a topological $A_3$ link. By performing the same calculations as in the proof of Theorem 1.2 (d) in \cite{go}, we can determine that their rotation numbers are as listed. Moreover, 
the corresponding $d_3$-invariant is $\frac{1}{2}$. 

If $t_2\leq 0$ and $t_0=1$ (resp. $t_0\geq2$), then there exist $2(1-t_{2})$ strongly exceptional Legendrian $A_3$ links shown in  Figure~\ref{Figure:link11inot} (resp. Figure~\ref{Figure:link10inot}) with $k_1=l_1=0$. Their rotation numbers and the corresponding $d_3$-invariants are as listed.
\end{proof}

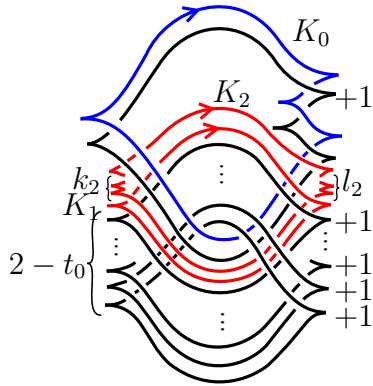
\begin{figure}[htb]
\begin{overpic}
{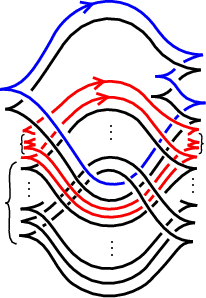}
\put(-27, 42){$2-t_{0}$}
\put(96, 105){$+1$}
\put(96, 42){$+1$}
\put(96, 32){$+1$}
\put(96, 22){$+1$}
\put(96, 58){$+1$}
\put(-3, 74){$k_2$}
\put(99, 73){$l_2$}
\put(80, 130){$K_0$}
\put(50, 107){$K_2$}
\put(-7, 64){$K_1$}
\end{overpic}
\caption{$t_0\leq 0$,  $t_{1}=0$, $t_{2}\leq0$, $k_{2}+l_{2}=-t_2$. }
\label{Figure:link5inot}
\end{figure}

(2) Suppose $t_2=1$.

\begin{lemma}\label{t1=0t2=1}
If $t_1=0$ and $t_2=1$, then there exist $4$ strongly exceptional Legendrian $A_3$ links whose rotation numbers and corresponding $d_3$-invariants $(r_0, r_1, r_2; d_3)$ are $$(\pm(t_{0}-3), \pm1, 0; \frac{3}{2}), (\pm(t_{0}+1), \pm1, \pm2; -\frac{1}{2}).$$
\end{lemma}

\begin{proof}
If $t_2=1$ and $t_0\geq4$, then there exist $4$ strongly exceptional Legendrian $A_3$ links shown in Figure~\ref{Figure:link15inot} with $k_1=l_1=0$. Their rotation numbers and corresponding $d_3$-invariants are as listed.

Suppose $t_2=1$ and $t_0\leq3$. By \cite[Theorem 1.2, (d), (c1)]{go}, there are two Legendrian Hopf links $K'_0\cup K_1$ in $(S^3, \xi_{\frac{1}{2}})$ with $t'_0\leq1, r'_0=\pm(t'_0-1), (t_1, r_1)=(0, \pm1)$, and one Legendrian Hopf links $K''_0\cup K_2$ in $(S^3, \xi_{\frac{1}{2}})$ with $(t''_0, r''_0)=(t_2, r_2)=(1, 0)$. By Lemma~\ref{Lemma:connected}, we can obtain strongly exceptional Legendrian $A_3$ links with $t_0\leq3, t_1=0, t_2=1$. So there are $2$ strongly exceptional Legendrian $A_3$ links in $(S^3, \xi_{\frac{3}{2}})$  whose rotation numbers $(r_0, r_1, r_2)$   are 
$(\pm(t_{0}-3), \pm1, 0).$

Moreover, by Lemma~\ref{Lemma:3.4.1} and Lemma~\ref{Lemma:appro}, there are other $2$ Legendrian $A_3$ links in $(S^3, \xi_{-\frac{1}{2}})$ whose rotation numbers $(r_0, r_1, r_2)$ are  $(\pm(t_{0}+1), \pm1, \pm2).$
\end{proof}

(3) Suppose $t_2=2$.

\begin{lemma}\label{t1=0t2=2}
If $t_1=0$ and $t_2=2$, then there exist $6$ strongly exceptional Legendrian $A_3$ links whose rotation numbers and corresponding $d_3$-invariants $(r_0, r_1, r_2; d_3)$ are $$(\pm(t_{0}+1), \pm1, \pm3; -\frac{1}{2}), (\pm(t_{0}-1), \pm1, \pm1; \frac{3}{2}), (\pm(t_{0}-3), \pm1, \mp1; \frac{3}{2}).$$
\end{lemma}

\begin{proof}
If $t_2=2$ and $t_0\geq3$, then there exist $6$ strongly exceptional Legendrian $A_3$ links shown in Figure~\ref{Figure:link17inot} with $k_1=l_1=0$. Their rotation numbers and corresponding $d_3$-invariants are as listed.

If $t_2=2$ and $t_0\leq2$, then by Lemma~\ref{Lemma:3.4.1} and Lemma~\ref{Lemma:appro}, there exist $2$ strongly exceptional Legendrian $A_3$ links in $(S^3, \xi_{-\frac{1}{2}})$ whose rotation numbers $(r_0, r_1, r_2)$ are  $(\pm(t_{0}+1), \pm1, \pm3).$

Moreover, by Lemma ~\ref{Lemma:3.4.2}, there exist $4$ strongly exceptional Legendrian $A_3$ links in $(S^3, \xi_{\frac{3}{2}})$  whose rotation numbers $(r_0, r_1, r_2)$ are 
$(\pm(t_{0}-1), \pm1, \pm1)$ and $  (\pm(t_{0}-3), \pm1, \mp1).$ 
\end{proof}

(4) Suppose $t_2\geq3$.

\begin{lemma}\label{t1=0t2>2}
If $t_1=0$ and $t_2\geq3$, then there exist $8$ strongly exceptional Legendrian $A_3$ links whose rotation numbers and corresponding $d_3$-invariants $(r_0, r_1, r_2; d_3)$ are $$(\pm(t_{0}+1), \pm1, \pm(t_2+1); -\frac{1}{2}), (\pm(t_{0}-1), \pm1, \pm(t_2-1); \frac{3}{2}), $$ $$(\pm(t_{0}-1), \pm1, \pm(3-t_{2}); \frac{3}{2}), (\pm(t_{0}-3), \pm1, \pm(1-t_2); \frac{3}{2}).$$
\end{lemma}
\begin{proof}
If $t_2\geq3$ and  $t_0\geq3$, then there are exactly $8$ strongly exceptional Legendrian $A_3$ links shown in Figure~\ref{Figure:link16inot} with $k_1=l_1=0$. Their rotation numbers and corresponding $d_3$-invariants are as listed.

Suppose $t_2\geq3$ and $t_0\leq2$. By Lemma~\ref{Lemma:3.4.1} and Lemma~\ref{Lemma:appro}, there exist $2$ strongly exceptional Legendrian $A_3$ links in $(S^3, \xi_{-\frac{1}{2}})$ whose rotation numbers $(r_0, r_1, r_2)$ are $(\pm(t_{0}+1), \pm1, \pm(t_{2}+1)).$

By Lemma ~\ref{Lemma:3.4.2}, there exist $4$ strongly exceptional Legendrian $A_3$ links in $(S^3, \xi_{\frac{3}{2}})$ whose rotation numbers $(r_0, r_1, r_2)$ are
$(\pm(t_{0}-1), \pm1, \pm(t_{2}-1)) $ and $(\pm(t_{0}-3), \pm1, \pm(1-t_{2})).$ 

Moreover, there are $2$ strongly exceptional Legendrian $A_3$ links in $(S^3, \xi_{\frac{3}{2}})$  whose rotation numbers $(r_0, r_1, r_2)$ are 
$(\pm(t_{0}-1), \pm1, \pm(3-t_{2})).$
The decorations of their exteriors are $\pm(+)((-)(+))$. These exteriors are appropriate tight since they can be embedded into an appropriate tight contact $\Sigma\times S^1$ with boundary slopes $t_0, \infty, -\frac{1}{2}$ and decoration $\pm(+)(-+)$ by adding   basic slices $(T^2\times[0,1], -\frac{1}{t_{2}}, -\frac{1}{t_{2}-1})$, $\cdots$, $(T^2\times[0,1], -\frac{1}{3}, -\frac{1}{2})$ to the boundary $T_2$. 
\end{proof}
The proof of Theorem~\ref{Theorem:t1=0} is completed.
\end{proof}

\begin{proof}[Proof of Theorem~\ref{classification}]
It follows from the proof of Theorems~\ref{Theorem:t_1<0t_2<0}, ~\ref{Theorem:t_1=t_2=1}, ~\ref{Theorem:t_1>1t_2=1}, ~\ref{Theorem:t_1>1t_2>1},
~\ref{Theorem:t_1<0t_2=1}, ~\ref{Theorem:t_1<0t_2>1}, ~\ref{Theorem:t1=0}.
\end{proof}

\section{Stabilizations}\label{Section:Stabilizations}
The aim of this section is to elucidate Remark~\ref{Remark:Stabilization}.
\subsection{Stabilizations of the component $K_0$.}
For the strongly exceptional Legendrian $A_3$ links with $t_{1}, t_{2}\neq0$ and $t_{0}+\lceil-\frac{1}{t_1}\rceil+\lceil-\frac{1}{t_2}\rceil\geq2$, their exteriors have $0$-twisting vertical Legendrian circles.  So by Lemma \ref{Lemma:stabilization}, the component $K_0$ can always be destabilized. For the strongly  exceptional Legendrian $A_3$ links with $t_{1}=0$, their exteriors obviously have $0$-twisting vertical Legendrian circles. By the same reason, the component $K_0$ can be destabilized.

As examples, we list the mountain ranges of the component $K_0$ in some Legendrian $A_3$ links with fixed $t_1, t_2$.

(1) Strongly exceptional Legendrian $A_3$ links in $(S^3, \xi_{\frac{5}{2}})$ with $r_0=\pm(t_0-5), r_1=r_2=0$, where $t_0\geq5, t_1=t_2=1$. See Lemmas~\ref{t0>5t1=1t2=1} and ~\ref{t0=5t1=1t2=1}. Their exteriors have decorations $\pm(+)(+)(+)$. The mountain range is depicted in the upper left of Figure~\ref{Figure:Mountain}. It is infinite on the upper side. 

(2) Strongly exceptional Legendrian $A_3$ links in $(S^3, \xi_{\frac{5}{2}})$ with $r_0=\pm(t_0-3), r_1=\pm(t_1-1), r_2=\pm(1-t_2)$, where $t_0, t_1, t_2\geq3$. See Lemmas~\ref{t0>3t1>2t2>2} and ~\ref{t0=3t1>2t2>2}. Their exteriors have decorations $\pm(+)((+)(-))((+)(+))$. The mountain range is depicted in the lower left of Figure~\ref{Figure:Mountain}. It is infinite on the upper side. 

(3) Strongly exceptional Legendrian $A_3$ links in $(S^3, \xi_{\frac{5}{2}})$  with $r_0=\pm(t_0-5), r_1=\pm(1-t_1), r_2=\pm(1-t_2)$, where $t_0, t_1, t_2\geq3$. See Lemmas~\ref{t0>3t1>2t2>2} and ~\ref{t0=3t1>2t2>2}. 
Their exteriors have decorations $\pm(+)((+)(+))((+)(+))$. The mountain range is  depicted in the upper right of Figure~\ref{Figure:Mountain}. It is infinite on the upper side. 

(4) Exceptional Legendrian $A_3$ links in $(S^3, \xi_{\frac{1}{2}})$  with $r_0=\pm(t_0-1), r_1=\pm(1-t_1), r_2=\pm(t_{2}+1)$, where $t_1, t_2\geq3$. Their exteriors have decorations $\pm(+)((+)(+))((-)(-))$. See Lemmas~\ref{Lemma:3.2.1}, ~\ref{t0>3t1>2t2>2}, ~\ref{t0=3t1>2t2>2} and ~\ref{t0=2t1>2t2>2}. The mountain range of such links is  depicted in the lower right of Figure~\ref{Figure:Mountain}. It is infinite on both the upper and lower sides. The exteriors of such $A_3$ links have decorations $\pm(+)((+)(+))((-)(-))$. If $t_0\geq2$, then they are strongly exceptional. If $t_0<2$, then, based on Lemma~\ref{Lemma:3.2.1} and Lemma~\ref{Lemma:tor>0}, they are exceptional but not strongly exceptional.

In a more general setting, with a fixed decoration and nonzero integers $t_1$ and $t_2$, if $L'_0$ and the innermost basic slices of $L'_1$ and $L'_2$ have the same signs (possibly after shuffling), then the components $K_0$ of the strongly exceptional Legendrian $A_3$ links exhibit mountain ranges with shapes resembling a `V' or an `X' truncated from the lower side, as shown in the first three subfigures in Figure~\ref{Figure:Mountain}.

For the strongly exceptional Legendrian $A_3$ links with fixed  $t_{1}, t_{2}\neq0$ and $t_{0}+\lceil-\frac{1}{t_1}\rceil+\lceil-\frac{1}{t_2}\rceil\leq1$, the mountain ranges of the component $K_0$ can be observed through Lemmas~\ref{t0<0t1<0t2<0}, Lemma~\ref{t0<4t1=1t2=1},
Lemma~\ref{t0<3t1>1t2=1}, 
Lemma~\ref{t0<2t1>1t2>1}, 
Lemma~\ref{t0<2t1<0t2=1} and 
Lemma~\ref{t0<1t1<0t2>1}.

\begin{figure}[htb]
\begin{overpic}
{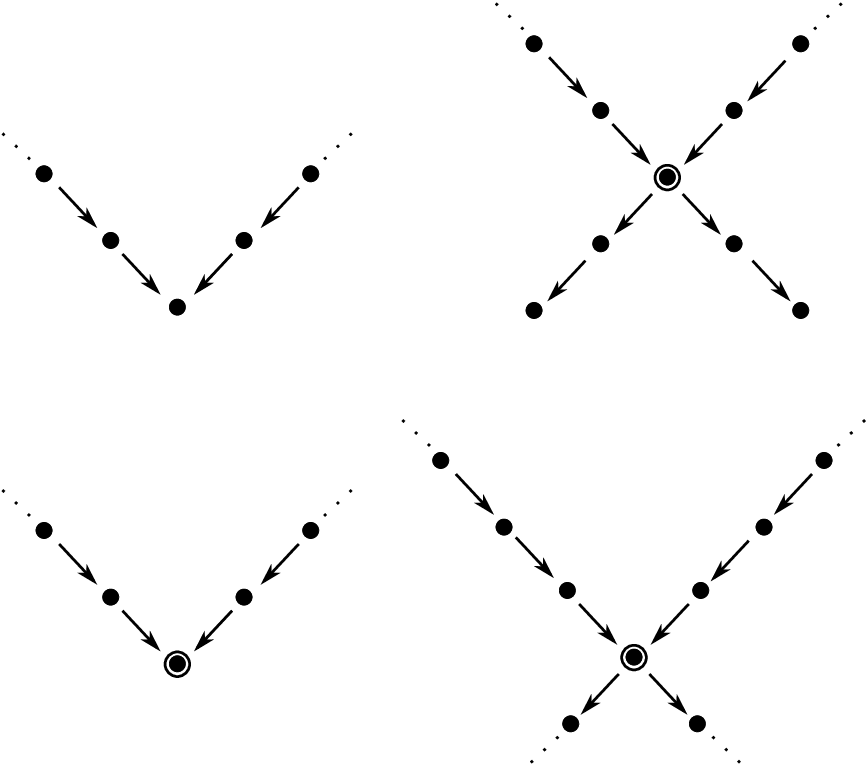}
\put(30, 217){$t_0=5$}
\put(0, 340){$(1)$}
\put(30, 46){$t_0=3$}
\put(0, 20){$(2)$}
\put(220, 217){$t_0=3$}
\put(220, 340){$(3)$}
\put(220, 82){$t_0=2$}
\put(220, 20){$(4)$}

\end{overpic}
\caption{The mountain ranges of some strongly exceptional Legendrian $A_3$ links with fixed $t_1$ and $t_2$. Each dot represents a Legendrian $A_3$ link. A dot with a circle represents  two Legendrian $A_3$ links. Each arrow represents a stabilization.  }
\label{Figure:Mountain}
\end{figure}

\begin{figure}[htb]
\begin{overpic}
{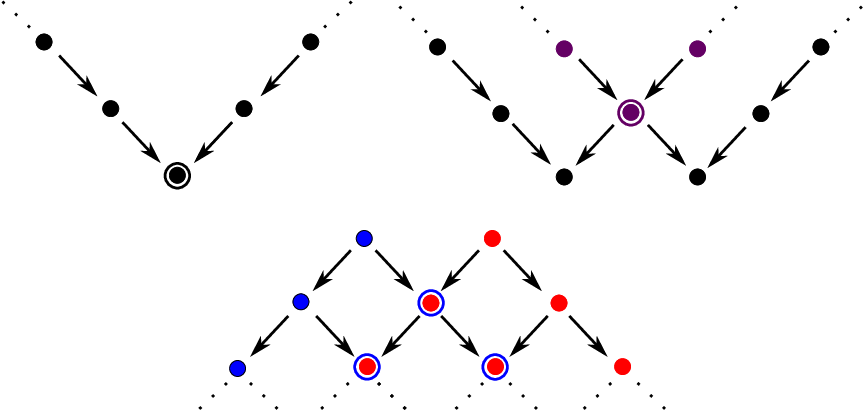}
\put(30, 110){$t_2=1$}
\put(220, 110){$t_2=2$}
\put(120, 80){$t_2=0$}
\end{overpic}
\caption{The mountain ranges of the strongly exceptional Legendrian $A_3$ links with fixed $t_{0}$ and $t_{1}=0$.
}
\label{Figure:Mountain1}
\end{figure}

\subsection{Stabilizations of the component $K_2$ when $t_1=0$.} 
The strongly exceptional Legendrian $A_3$ links with $t_1=0$ are classified in Theorem~\ref{Theorem:t1=0}.
The exteriors of such Legendrian $A_3$ links contain $0$-twisting vertical Legendrian circles.
By Lemma \ref{Lemma:stabilization}, the component $K_2$ is always  destabilizable unless $t_2=0$. 

We list the mountain range of the component $K_2$ of the strongly exceptional Legendrian $A_3$ links with fixed $t_{0}$ and $t_{1}=0$.  The exteriors of such Legendrian $A_3$ links can be decomposed into $L'_0\cup L'_2\cup \Sigma'\times S^1$.  Recall that if $t_{2}\geq3$ then $L'_2$ consists of $2$ basic slices and is not a continued fraction block. If $t_{2}=2$ then $L'_2$ is a continued fraction block consisting of $2$ basic slices. If $t_{2}=1$ then $L'_2$ is a basic slice. If $t_{2}=0$ then $L'_2$ is an empty set. If $t_{2}\leq -1$ then $L'_2$ is a continued fraction block consisting of $-t_{2}$ basic slices.

(1) Strongly exceptional Legendrian $A_3$ links in $(S^3, \xi_{\frac{3}{2}})$ with $r_0=\pm(t_0-3), r_1=\pm1, r_{2}=\pm(1-t_2)$, where $t_1=0, t_2\geq1$. See Lemmas~\ref{t1=0t2=1}, ~\ref{t1=0t2=2} and ~\ref{t1=0t2>2}. The signs of the basic slices in $L'_0$ and $L'_2$ are all the same. 
The mountain range is depicted in the upper left of Figure~\ref{Figure:Mountain1}.

(2) Strongly exceptional Legendrian $A_3$ links in $(S^3, \xi_{-\frac{1}{2}})$ with $r_0=\pm(t_0 +1), r_1=\pm1, r_{2}=\pm(t_2 +1)$, where $t_1=0, t_2\geq1$. See Lemmas~\ref{t1=0t2=1}, ~\ref{t1=0t2=2} and ~\ref{t1=0t2>2}. The sign of $L'_0$ and the sign of each of the basic slices in $L'_2$ are opposite. 
The mountain range can be depicted in the upper left of Figure~\ref{Figure:Mountain1} either.

(3) Strongly exceptional Legendrian $A_3$ links in $(S^3, \xi_{\frac{3}{2}})$ with $r_0=\pm(t_0 -1), r_1=\pm1, r_{2}=\pm(t_2 -1)$ (or $ r_{2}=\pm(3-t_2)$), where $t_1=0, t_2\geq2$. See Lemmas~\ref{t1=0t2=2} and ~\ref{t1=0t2>2}.  Their exteriors have decorations $\pm(+)((+)(-))$ (or $\pm(+)((-)(+))$) if $t_2\geq3$, and $\pm(+)(+-)$ if $t_2=2$. Note that when $t_2=2$, $L'_2$ is a continued fraction block, and hence the two decorations $(+)(+-)$ and $(+)(-+)$ (or $(-)(-+)$ and $(-)(+-)$) stand for the same Legendrian $A_3$ link. So the mountain range can be depicted in the upper right of Figure~\ref{Figure:Mountain1}.

(4) Strongly exceptional Legendrian $A_3$ links in $(S^3, \xi_{\frac{1}{2}})$ with $r_0=\pm(t_0 -1), r_1=\pm1, r_{2}\in\pm\{t_2 +1, t_2 +3,\cdots, -t_{2}+1\}$, where $t_1=0, t_2\leq0$. See Lemma~\ref{t1=0t2<1}. It is easy to know that the mountain range can be depicted in the lower of Figure~\ref{Figure:Mountain1}.

In conclusion, the whole mountain range of the strongly exceptional Legendrian $A_3$ links with  fixed $t_{0}$ and $t_{1}=0$ consists of two copies of the upper left subfigure, the upper right subfigure and  the lower subfigure of Figure~\ref{Figure:Mountain1}.

\section{Some Computations}
Here we summarise how to compute the classical invariants of Legendrian realizations $A_3 = K_0 \cup K_1 \cup K_2$ of the connected sum of two Hopf links, and the $d_3$-invariant of the contact $3$-sphere $S^3$ containing the realizations. We compute the invariants of the first surgery diagram on the top left of Figure  ~\ref{Figure:link23inot}. Similar arguments apply to all remaining examples. For the example in Figure  ~\ref{Figure:link23inot}, the linking matrix $M$ is the $(t_0-1) \times (t_0-1)$-matrix, which we form by ordering the surgery curves from bottom to top where all are oriented clockwise:

\begin{displaymath}
M =
\begin{bmatrix}
-2 & -1 & -1  \\
-1 & -2 & -1  \\
-1 & -1 & -2 & -1 \\
& & -1 & -2 & -1 \\
& & & -1 & -2 & -1 \\
& & & & & & \ddots \\
& & & & & & & & -1 \\
& & & & & & & -1 & -2 & -1 \\
& & & & & & & & -1 & -1 \\
\end{bmatrix}.
\end{displaymath} 
The determinant of $M$ is ${\tt det}\, M = (-1)^{t_0 -1}$.
\subsection{The $d_3$-invariant.} Let $(Y, \xi) = \partial X$ be a contact $3$-manifold given by contact $(\pm 1)$-surgeries on a Legendrian link $\mathbb{L}\in(S^3, \xi_{st})$, all of which have non-vanishing Thurston-Bennequin invariant. We compute the $d_3$-invariant of $(Y, \xi)$ with $c_1(\xi)$ torsion by following the formula from \cite[Corollary 3.6]{DGS}:
$$ d_3(\xi) = \frac{1}{4}(c^2 -3 \sigma(X) - 2 \chi(X)) + q,$$
where $q$ is the number of $(+1)$-surgery components in $\mathbb{L}$, and $c \in H^2(X)$ is the cohomology class determined by $c(\Sigma_i)$, for each $L_i \in \mathbb{L}$ where $\Sigma_i$ Seifert surface of $L_i$ glued with the core disk of the corresponding handle. We read $\sigma$ and $\chi$ from the surgery diagram in Figure~\ref{Figure:link23inot}. The signature $\sigma$ is the signature of the linking matrix $M$. The surgery diagram is topologically  equivalent to $(t_0 -1)$ unlinked $-1$-framed unknots, so the signature is $\sigma(X) = -(t_0 -1)$. The Euler characteristic is $\chi(X) = t_0 -1 +1= t_0$ since each surgery knot corresponds to attaching a $2$-handle. We compute $c^2$ by following the algorithm in \cite{DGS}, $c^2 = {\tt \bold{x}}^tM{\tt \bold{x}} = <{\tt \bold{x}},{\tt \underline{rot}}>$ where $\underline{{\tt rot}} = (rot(L_1), \ldots, rot(L_n))$ is the vector rotation number of the Legendrian surgery knots $L_i \subset \mathbb{L}$, and $\tt \bold{x}$ is the solution vector of $M{\tt \bold{x}} = {\tt \underline{rot}}$. For the surgery diagram on top left of Figure~\ref{Figure:link23inot}, the vector rotation number is 
 $$\underline{{\tt rot}} = (2, -2, 0, \ldots, 0, 1)^t .$$
 The solution vector $\tt \bold{x}$ is
 \begin{center}
   $ {\tt \bold{x}} = (-1, 3,*,\ldots,*,-(t_0-1))^t$ \, for $t_0$ even,   
 \end{center}
 and 
 \begin{center}
   $ {\tt \bold{x}} = (-3,1,*,\ldots,*,-(t_0-1))^t$ \, for $t_0$ odd.  
 \end{center}
This gives $c^2 = <{\tt \bold{x}},{\tt \underline{rot}}> = -6 -2 +0+\cdots +0 -(t_0-1) = -7 -t_0.$
Observing that $q=3$ in this example, we compute
$$ d_3 = \frac{1}{4}(-7-t_0 -3(-(t_0-1))-2t_0) + 3 
=\frac{1}{2}.$$

\subsection{The Thurston-Bennequin invariant and the rotation number.} We use the formulae in \cite[Lemma 6.6]{loss}  to compute the Thurston-Bennequin invariant and the rotation number of a Legendrian knot $L$ in a contact $(\pm 1)$-surgery diagram of surgery link $\mathbb{L}$ with the linking matrix $M$. The Thurston-Bennequin invariant is 
$$ tb(L) = tb(L_0) + \frac{{\tt det}\, M_0}{{\tt det}\, M}, $$
where $tb(L_0)$ is the Thurston-Bennequin invariant of $L$ as a knot in $(S^3, \xi_{st})$ before the contact surgeries, and $M_0$ is the extended linking matrix which is the linking matrix of $ L_0 \cup \mathbb{L}$ with the convention that $lk(L_0, L_0) = 0$. The rotation number of $L$ after surgery is 
$$ rot(L) = rot(L_0) - < \underline{{\tt rot}}, M^{-1}\underline{{\tt lk}}>,$$

 where $rot(L_0)$ is the rotation number of $L$ before surgeries, and
$\underline{{\tt rot}}$ is the vector rotation number of the Legendrian surgery knots $L_i \subset \mathbb{L}$, and $\underline{{\tt lk}} = (lk(L,L_1), \ldots, lk(L,L_n))$ is the vector of the linking numbers.

For the surgery diagram on the top left of Figure~\ref{Figure:link23inot}, we assume 
that $K_0$, $K_1$ and $K_2$ are oriented clockwise. So the extended linking matrices for $K_0$, $K_1$ and $K_2$ are respectively:
\begin{displaymath}
M_0 =
\begin{bmatrix}
0 & 0 &  \cdots & 0 & -1 & -2  \\
0 &   \\
\vdots & & & M  \\
0  \\
-1 & \\
-2 &  \\
\end{bmatrix}, \; M_1 =
\begin{bmatrix}
0 & -1 & -3 & -1 & 0 \cdots & 0 \\
-1 &   \\
-3  \\
-1 & & & M    \\
0 &   \\
\vdots   \\
0  \\
\end{bmatrix}, 
\end{displaymath} 

\begin{displaymath}
M_2 =
\begin{bmatrix}
0 & -3 & -1 & -1 & 0 \cdots & 0 \\
-3 &   \\
-1  \\
-1 & & & M    \\
0 &   \\
\vdots   \\
0  \\
\end{bmatrix}.
\end{displaymath} 

The determinants are ${\tt det}\, M_0 = (-1) ^{t_0 - 1} (t_0 + 2)$ and ${\tt det}\, M_1 = {\tt det}\, M_2 = 5 (-1) ^{t_0 - 1}$. We compute the Thurston-Bennequin invariants as follows,
\begin{center}
  $\displaystyle tb(K_0) = -2 + \frac{(-1) ^{t_0 - 1}(t_0 + 2)}{(-1) ^{t_0 - 1}} = t_0$, and 
$\displaystyle tb(K_1) = tb(K_2) = -3 + \frac{5(-1) ^{t_0 - 1}}{(-1) ^{t_0 - 1}} = 2.$  \end{center}

Recall that for $t_0$ odd, $K_{0}$ and $K_{i}$ are given the same orientation, for $t_0$ even, the opposite one, where $i=1,2$.
If $t_0$ is odd, then $K_i$ is oriented clockwise.  If $t_0$ is even, then  $K_i$ is oriented counter-clockwise. We compute the rotation numbers as follows,
\begin{equation} \displaystyle
\begin{array}{ccl}
 r_0 &=&  1 - \langle \begin{bmatrix}
2 \\
-2 \\
0
\\
\vdots \\
0\\
1

\end{bmatrix}, M^{-1} \begin{bmatrix}
0 \\
0\\
\vdots \\
0 \\
-1\\
-2

\end{bmatrix}\rangle \nonumber  = 

\bigskip  1 - \langle \begin{bmatrix}
2 \\
-2 \\
0
\\
\vdots \\
0\\
1

\end{bmatrix}, \begin{bmatrix}
(-1)^{t_0 -1} \\
(-1)^{t_0 -1} \\
* \\
* \\
* \\
t_0

\end{bmatrix}\rangle \\  &=& 
\displaystyle  1 - (2-2+0+\cdots +0 + t_0) = - (t_0-1),
\end{array}
\end{equation}

\begin{equation} \displaystyle
\begin{array}{ccl}
 r_1 &=& 2(-1)^{t_0}  - \langle \begin{bmatrix}
2 \\
-2 \\
0
\\
\vdots \\
0\\
1

\end{bmatrix}, M^{-1}  \begin{bmatrix}
(-1)^{t_0} \\
3(-1)^{t_0}\\
(-1)^{t_0}\\
0\\
\vdots \\
0

\end{bmatrix} \rangle \nonumber  = 

\bigskip 2(-1)^{t_0} - \langle  \begin{bmatrix}
2 \\
-2 \\
0
\\
\vdots \\
0\\
1

\end{bmatrix},  \begin{bmatrix}
0\\
2(-1)^{t_0 -1} \\
* \\
* \\
* \\
1

\end{bmatrix} \rangle \\  &=& 
\displaystyle 2(-1)^{t_0} - (0 +4(-1)^{t_0}+0+\cdots +0 + 1) = 
\left \{ \begin{array}{ll}
1 & \textrm{ if $t_0$ is odd,}\\
-3 & \textrm{ if $t_0$ is even,}
\end{array} \right.
\end{array}
\end{equation}

\begin{equation} \displaystyle
\begin{array}{ccl}
 r_2 &=&  2(-1)^{t_0 -1} - \langle  \begin{bmatrix}
2 \\
-2 \\
0
\\
\vdots \\
0\\
1

\end{bmatrix} , M^{-1} 
\begin{bmatrix}
3(-1)^{t_0} \\
(-1)^{t_0}\\
(-1)^{t_0}\\
0\\
\vdots \\
0

\end{bmatrix} \rangle \nonumber  = 

\bigskip 2(-1)^{t_0-1}  - \langle  \begin{bmatrix}
2 \\
-2 \\
0
\\
\vdots \\
0\\
1

\end{bmatrix},  \begin{bmatrix}
2(-1)^{t_0-1}\\
0 \\
* \\
* \\
* \\
1

\end{bmatrix} \rangle \\  &=& 
\displaystyle 2(-1)^{t_0-1}  - (4(-1)^{t_0-1}+0+\cdots +0 + 1) = 
\left\{ \begin{array}{ll}
-3 & \textrm{if $t_0$ is odd,}\\
1 & \textrm{if $t_0$ is even.}
\end{array} \right.
\end{array}
\end{equation}

\end{document}